\documentclass{amsart}
\usepackage{amssymb,amsmath,amscd,xy,graphicx,textcomp}

\newtheorem{theorem}{Theorem}[section]
\newtheorem{lemma}[theorem]{Lemma}
\newtheorem{corollary}[theorem]{Corollary}
\newtheorem{definition}[theorem]{Definition}

\newtheorem{remark}[theorem]{\it Remark}

\newtheorem{proposition}[theorem]{Proposition}
\newtheorem{conjecture}[theorem]{Conjecture}

\xyoption{arrow}

\xyoption{matrix}

\def\SL{\mathrm{SL}}

\def\C{\mathbb{C}}
\def\R{\mathbb{R}}
\def\Z{\mathbb{Z}}

\def\Z{\mathbb{Z}}

\def\tree{\mathcal{T}}

\def\Pncc{\includegraphics[bb = 0 17.5 40 0, scale = .65]{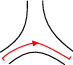}} 

\def\Pnck{\includegraphics[bb = 0 17.5 40 0, scale = .65]{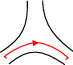}} 

\def\Pnkc{\includegraphics[bb = 0 17.5 40 0, scale = .65]{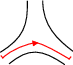}} 

\def\Pnkk{\includegraphics[bb = 0 17.5 40 0, scale = .65]{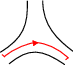}}

\def\Pcnc{\includegraphics[bb = 0 17.5 40 0, scale = .65]{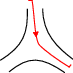}} 

\def\Pcnk{\includegraphics[bb = 0 17.5 40 0, scale = .65]{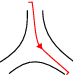}} 

\def\Pknc{\includegraphics[bb = 0 17.5 40 0, scale = .65]{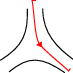}} 

\def\Pknk{\includegraphics[bb = 0 17.5 40 0, scale = .65]{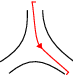}}

\def\Pccn{\includegraphics[bb = 0 17.5 40 0, scale = .65]{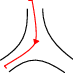}} 

\def\Pckn{\includegraphics[bb = 0 17.5 40 0, scale = .65]{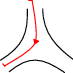}} 

\def\Pkcn{\includegraphics[bb = 0 17.5 40 0, scale = .65]{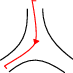}} 

\def\Pkkn{\includegraphics[bb = 0 17.5 40 0, scale = .65]{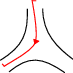}}

\def\tri{\includegraphics[bb = 0 17.5 40 0, scale = .65]{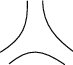}}

\title{Compactifications of character varieties \\and skein relations on conformal blocks}
\author{Christopher Manon}
\thanks{}

\begin{document}

\begin{abstract}
Let $M_C(G)$ be the moduli space of semistable principal $G-$bundles over a smooth curve $C$.  We show that a flat degeneration of this space $M_{C_{\Gamma}}(G)$ associated to a singular stable curve $C_{\Gamma}$ contains the free group character variety $\mathcal{X}(F_g, G)$ as a dense, open subset, where $g = genus(C).$  In the case $G = SL_2(\C)$ we describe the resulting compactification explicitly, and in turn we conclude that the coordinate ring of $M_{C_{\Gamma}}(SL_2(\C))$ is presented by homogeneous skein relations.  Along the way, we prove the parabolic version of these results over stable, marked curves $(C_{\Gamma}, \vec{p}_{\Gamma})$.
\end{abstract} 

\maketitle

\smallskip
\noindent

\section{Introduction}

We explore a relationship between the moduli space $M_C(G)$ of semistable principal $G$-bundles on a stable projective curve $C$, 
and the character variety $\mathcal{X}(F_g, G)$ of the free group on $g = genus(C)$ generators. As the curve $C$ varies in the moduli of smooth curves $\mathcal{M}_g,$ the spaces $M_C(G)$ form a flat family of projective schemes, and this family can be extended (\cite{M4}) to a flat family on the Deligne-Mumford compactification $\bar{\mathcal{M}}_g.$ We show that the fiber $M_{C_{\Gamma}}(G)$ over a maximally singular stable curve $C_{\Gamma}$ contains $\mathcal{X}(F_g, G)$ as a dense open subspace.  Our methods are algebraic and combinatorial, and have as their centerpiece a relationship between the coordinate ring $\C[\mathcal{X}(F_g, G)]$ and a degenerated algebra of non-Abelian theta functions $V_{C_{\Gamma}}(G),$ which serves as the projective coordinate ring of $M_{C_{\Gamma}}(G).$ 

\begin{theorem}\label{maing0}
For $G$ a simple, simply connected complex group, the algebra $V_{C_{\Gamma}}(G)$ is a Rees algebra of $\C[\mathcal{X}(F_g, G)].$ The character variety $\mathcal{X}(F_g, G)$ is a dense, open subspace of $M_{C_{\Gamma}}(G)$. 
\end{theorem}

 Narasimhan and Seshadri introduced character varieties into the study of semistable principal bundles in \cite{NS}, where they show that $M_C(G)$ is homeomorphic to $\mathcal{X}(\pi_1(C), K)$ for $K \subset G$ a maximal compact subgroup.  In contrast, we work with a free group character variety for the complex group $G$, similar to Florentino's work \cite{F} on Schottky uniformization.  Florentino defines a natural map between $\mathcal{X}(F_g, G)$ and $M_C(G)$ for $C$ smooth $G = GL_n(\C)$, and studies where this map is a submersion.  Theorem \ref{maing0} is an analogue of this result for singular curves, and echoes the principle that structures in the moduli of principal bundles on curves simplify when degenerated to the stable boundary of $\mathcal{M}_g$.  Degeneration techniques have been used in the study of vector bundles and principal bundles at least since the work of Gieseker \cite{Gi}, in particular Abe \cite{A} and the author \cite{M4}, \cite{M10} have used these methods to establish structural properties of the coordinate ring of $M_C(G)$.

We begin with a description of our results in the case $G = SL_2(\C).$  The coordinate ring of $\mathcal{X}(F_g, SL_2(\C))$ has two interesting
combinatorial structures: it is presented by skein relations on a trivalent ribbon graph $\Gamma$ with $\beta_1(\Gamma) = g$, and it has a basis of spin diagram functions, represented by integer labellings of the edges of $\Gamma.$  We focus for now on skein relations, and direct the reader to Section \ref{sl2theory} for their construction.

\begin{figure}[htbp]
\centering
\includegraphics[scale = 0.5]{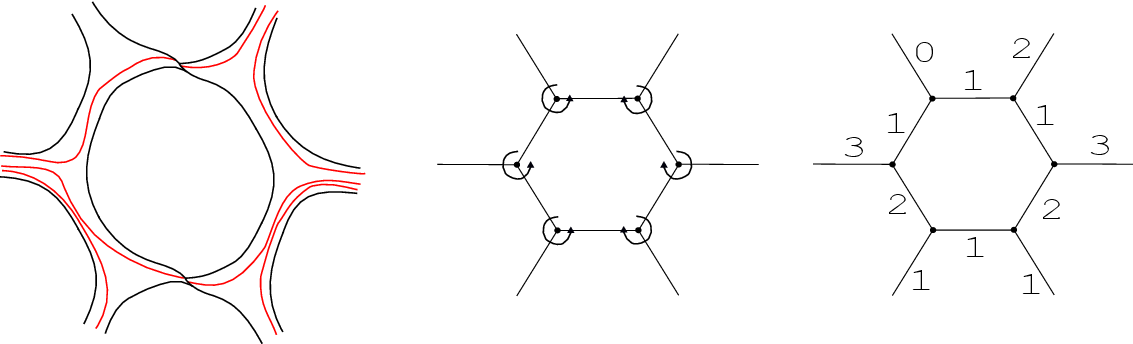}
\caption{Left: Paths on a ribbon graph. Right: A spin diagram on the underlying graph. }
\label{ribwire}
\end{figure}

  A ribbon structure on a graph $\Gamma$ is an assignment of a cyclic ordering to the edges in the link of each vertex, this is sufficient combinatorial data to build an orientable surface from $\Gamma$, see Figure \ref{ribwire}. Roughly, the skein algebra associated to $\Gamma$ is a vector space spanned by arrangements of paths inside the thickened graph up to isotopy equivalence.
Multiplication in the skein algebra is computed by taking unions of these arrangements in general position, and resolving crossings with the skein relations. 

\begin{figure}[htbp]
\centering
\includegraphics[scale = 0.7]{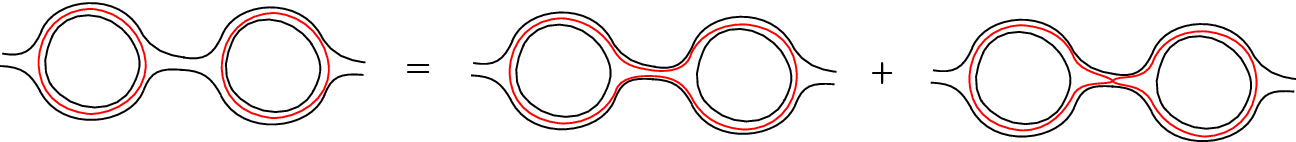}
\caption{the skein relation $tr(M)tr(N) = tr(MN) + tr(MN^{-1}).$}
\label{ribwire2}
\end{figure} 

Skein relations are appealing in part because they present equations from linear algebra in a graphical way, for example Figure \ref{ribwire2} represents a relation on traces of two $2\times 2$ matrices.   We let  $F_{\Gamma}(L) \subset \C[\mathcal{X}(F_g, SL_2(\C))]$ be the subspace spanned by elements with $\leq L$ paths through each vertex $v \in V(\Gamma).$  Each of these spaces is finite dimensional, and they form an increasing filtration on the algebra $\C[\mathcal{X}(F_g, SL_2(\C))]$ with an associated Rees algebra $R_{\Gamma}(\C[\mathcal{X}(F_g, SL_2(\C))])$ $=\bigoplus_{L \geq 0} F_{\Gamma}(L)$.

 Recall that the stack $\bar{\mathcal{M}}_g$ of semistable curves of genus $g$ is stratified by combinatorial types of arrangements of nodal singularities, and that the lowest strata are isolated points $C_{\Gamma}$ indexed by connected trivalent graphs $\Gamma$ with no leaves and $\beta_1(\Gamma) = g.$ The following theorem is a consequence of Theorem \ref{maing0}, and places the Rees algebra $R_{\Gamma}(\C[\mathcal{X}(F_g, SL_2(\C))])$ in the theory of semistable $SL_2(\C)$ principal bundles. 

\begin{figure}[htbp]
\centering
\includegraphics[scale = 0.35]{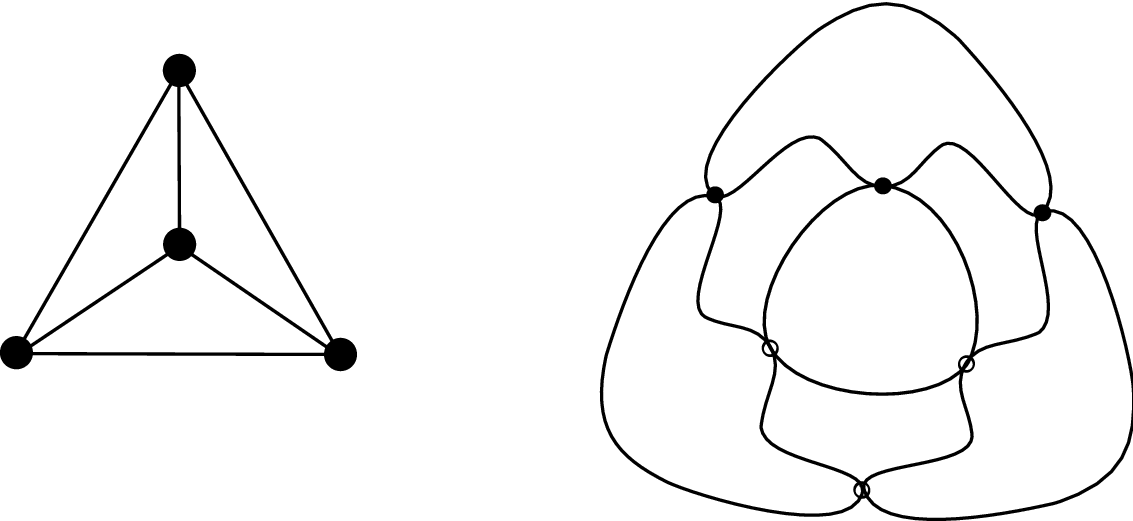}
\caption{A stable curve with graph}
\label{fig:1}
\end{figure} 

\begin{theorem}\label{maingeo}
The Rees algebra $R_{\Gamma}(\C[\mathcal{X}(F_g, SL_2(\C))])$ is isomorphic to the projective coordinate ring $V_{C_{\Gamma}}(SL_2(\C))$ of $M_{C_{\Gamma}}(SL_2(\C))$.  As a consequence, each $M_{C_{\Gamma}}(SL_2(\C))$ is a compactification of $\mathcal{X}(F_g, SL_2(\C)).$ 
\end{theorem}

\begin{corollary}\label{present}
The algebra $V_{C_{\Gamma}}(SL_2(\C))$ is presented by homogeneous skein relations. 
\end{corollary}

Our approach to Theorem \ref{maing0} is to construct $\mathcal{X}(F_g, G)$ and $M_{C_{\Gamma}}(G)$ in a parallel way using $GIT$ and a recipe derived from the graph $\Gamma$, taking advantage of representation theoretic structures in the coordinate rings of both $\mathcal{X}(F_g, G)$ and $M_{C_{\Gamma}}(G)$.   The graded components $V_C(L)$ of the projective coordinate ring $V_C(G)$ of $M_C(G)$ are called non-Abelian theta functions, and they are known to coincide with the spaces of conformal blocks associated to the Wess-Zumino-Novikov-Witten (WZNW) model of conformal field theory on $C$ with respect to the Lie algebra $\mathfrak{g} = Lie(G)$.    The WZNW theory assigns a finite dimensional vector space $V_{C, \vec{p}}(\vec{\lambda}, L)$  to each stable, marked curve $(C, \vec{p})$ in the Deligne-Mumford stack of stable curves $\bar{\mathcal{M}}_{g, n}$ for every non-negative integer $L$ and tuple $(\lambda_1, \ldots, \lambda_n) = \vec{\lambda}$ of integral dominant weights chosen from the level $L$ alcove $\Delta_L \subset \Delta$ (see Section \ref{factorcor} below) in a Weyl chamber of $\mathfrak{g}.$  Results of Kumar, Narasimhan, and Ramanathan \cite{KNR}, Faltings \cite{Fal}, Beauville and Laszlo \cite{BL}, and Pauly \cite{P} ( in the parabolic case) identify these spaces with the spaces of global sections of line bundles on moduli spaces of parabolic principal $G$-bundles on the curves $(C, \vec{p})$. Conformal blocks come with the following properties, proved by Tsuchiya, Ueno, and Yamada in \cite{TUY}.

\begin{enumerate}
\item (Flatness) The spaces $V_{C, \vec{p}}(\vec{\lambda}, L)$ form a vector bundle $V(\vec{\lambda}, L)$ over $\bar{\mathcal{M}}_{g, n}.$\\
\item (Correlation) The space $V_{\mathbb{P}^1, \vec{p}}(\vec{\lambda}, L)$ can be realized as a subspace of the space of invariants in the $n-$fold tensor product of irreducible representations $(V(\lambda_1)\otimes \ldots \otimes V(\lambda_n))^{\mathfrak{g}}.$\\
\item (Factorization) For a stable curve $C$ with nodal singularity $q \in C$ and partial normalization $(\tilde{C}, q_1, q_2)$ with induced marked points $q_1, q_2,$  $$V_{C, \vec{p}}(\vec{\lambda}, L) = \bigoplus_{\alpha \in \Delta_L} V_{\tilde{C}, \vec{p}, q_1, q_2}(\vec{\lambda}, \alpha, \alpha^*, L)$$.\\
\end{enumerate}

\begin{figure}[htbp]
\centering
\includegraphics[scale = 0.35]{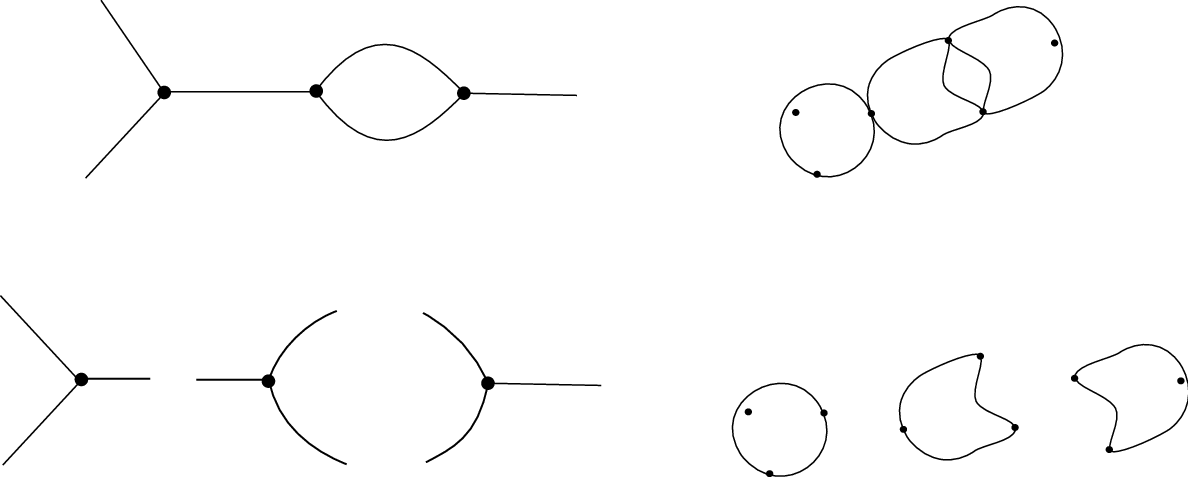}
\caption{Normalization of a stable curve}
\label{fig:2}
\end{figure}

Mimicking the factorization property (3) above, we try to build the spaces  $M_{C_{\Gamma}}(G)$ out of the moduli spaces $M_{C, \vec{p}}(\vec{\lambda}, L)$  of parabolic semistable principal bundles on $(C, \vec{p}) \in \bar{\mathcal{M}}_{g, n}$ following ideas of Hurtubise, Jeffrey, Sjamaar, \cite{HJS}, and Jeffrey, Weitsman \cite{JW}.   By decomposing a stable curve $C$ into its marked irreducible components $(C_i, \vec{q}_i)$, one attempts to reconstruct a point in $M_{C_{\Gamma}}(G)$ as a collection of points on parabolic moduli over the component curves with compatible parabolic data at the induced marked points.  However, the parabolic data appearing at a gluing point for general bundles in $M_{C_{\Gamma}}(G)$  is certainly not limited to a single choice of $\mathfrak{g}-$weight. Indeed, in the factorization statement $(3)$ above, one requires the sum over $\alpha \in \Delta_L$, as opposed to a single $\alpha$ to do the job.   For this reason, we require a space associated to each $C_i, \vec{q}_i$ which makes all possible parabolic data available.  A space which ''sees" all of this data is studied in \cite{M4}, the spectrum $\bar{K}_{C_i, \vec{q}_i}(G)$ of the total coordinate ring $V_{C_i, \vec{q}_i}(G)$ of the moduli stack $\mathcal{M}_{C_i, \vec{q}_i}(G)$ of quasi-parabolic principal bundles on $C_i, \vec{q}_i$.  This space is more flexible in two ways, first any of the parabolic moduli can be obtained as a $T^n\times \C^*$ $GIT$ quotient from $\bar{K}_{C, \vec{p}}(G),$ where $T \subset G$ is a maximal torus. 

\begin{equation}
\bar{K}_{C, \vec{p}}(G)/_{\vec{\lambda}, L} T^n\times\C^* = M_{C, \vec{p}}(\vec{\lambda}, L)\\
\end{equation}

\noindent
Second, in sympathy with the factorization property, for a stable curve $(C, \vec{p})$, with normalization $\coprod (C_i, \vec{p}_i, \vec{q}_i)$ the scheme $\bar{K}_{C, \vec{p}}(G)$ is $almost$ a torus $GIT$ quotient of the product $\prod \bar{K}_{C_i, \vec{p}_i, \vec{q}_i}(G),$ as shown in \cite{M4}, these two spaces are related by a flat degeneration. Here $m$ is the number of pairs of points $q_1, q_2$ introduced by the normalization.

\begin{equation}
\bar{K}_{C, \vec{p}}(G) \Rightarrow [\prod \bar{K}_{C_i, \vec{p}_i, \vec{q}_i}(G)]/ (T \times \C^*)^m\\
\end{equation}

  While this degeneration statement is an approximation of the property we need, it is too weak for our purposes. We pass to another ''universal" space $B_{C, \vec{p}}(G)$, and its affine cone $\bar{B}_{C, \vec{p}}(G)$, constructed in Section \ref{algebra}.  Mirroring $\bar{K}_{C, \vec{p}}(G)$, any moduli space of parabolic bundles can be obtained from $\bar{B}_{C, \vec{p}}(G)$ by an extended $G^n \times \C^*$ GIT quotient.  Below $\mathcal{O}(\lambda)$ denotes the flag variety of $G$ obtained as the orbit through the highest weight vector $[v_{\lambda^*}] \in \mathbb{P}(V(\lambda^*))$. 

\begin{equation}
\bar{B}_{C, \vec{p}}(G) \times \mathcal{O}(\lambda_1)\times \ldots \times \mathcal{O}(\lambda_n)/_{\mathcal{L}(\vec{\lambda})}G^n \times \C^* = M_{C, \vec{p}}(\vec{\lambda}, L)\\
\end{equation}

This quotient is taken with respect to the linearization $\mathcal{L}(\vec{\lambda})$ defined by the trivial bundle on $\bar{B}_{C, \vec{p}}(G)$
and the $G-$linearized line bundles $\mathcal{L}(\lambda_i)$ on the $\mathcal{O}(\lambda_i)$, obtained as the pullbacks of $\mathcal{O}(1)$ on $\mathbb{P}(V(\lambda_i^*)).$  The spaces $B_{C, \vec{p}}(G)$ still fit into a flat family over $\bar{\mathcal{M}}_{g, n}$, and the second property above becomes equality on the nose (Section \ref{stablecurvesection}, Proposition \ref{ringfact}).

\begin{theorem}\label{mainfact}
The following holds for a stable curve $(C, \vec{p})$ with normalization $\coprod (C_i, \vec{p}_i, \vec{q}_i)$. 

\begin{equation}
\bar{B}_{C, \vec{p}}(G) = [\prod \bar{B}_{C_i, \vec{p}_i, \vec{q}_i}(G)]/ (G\times \C^*)^m\\
\end{equation}

\end{theorem}

When $\vec{p} = \emptyset$, $B_{C}(G) = M_C(G)$ so we obtain a construction of $M_{C_{\Gamma}}(G)$ based around the combinatorics of the graph $\Gamma.$  In order to give the parallel construction of $\mathcal{X}(F_g, G)$, we bring in spaces $M_{g, n}(G)$ (Section \ref{mgnsection}). For $g, n$ with $2g + n \geq 3$, the space $M_{g, n}(G)$ can be assembled from copies of $M_{0, 3}(G)$ by $GIT$ quotient, in sympathy with the construction defined for Theorem \ref{mainfact}, and when $n =0$ we obtain the character variety, $M_{g, 0}(G) = \mathcal{X}(F_g, G).$   We also enrich the picture for the scheme $\bar{K}_{C, \vec{p}}(G)$ and its projectivization $K_{C, \vec{p}}(G) = Proj(V_{C, \vec{p}}(G))$ (the grading is defined by the level $L$ of the conformal blocks) by bringing in the affine $GIT$ quotients $P_{g, n}(G) = M_{g, n}(G)/U^n$, where $U \subset G$ is a maximal unipotent subgroup.  The following theorem is proved by first verifying the $0, 3$ case, and then following parallel $GIT$ constructions of $M_{g, n}(G)$ and $\bar{B}_{C_{\Gamma}, \vec{p}_{\Gamma}}(G)$ out of the $0, 3$ spaces. 

\begin{theorem}\label{maingen1}
The coordinate ring of $\bar{B}_{C_{\Gamma}, \vec{p}_{\Gamma}}(G)$ is a Rees algebra of $\C[M_{g, n}(G)]$, and $M_{g, n}(G)$ is a dense, open subscheme of $B_{C_{\Gamma}, \vec{p}_{\Gamma}}(G).$
The coordinate ring of $\bar{K}_{C_{\Gamma}, \vec{p}_{\Gamma}}(G)$ is a Rees algebra of $\C[P_{g, n}(G)]$, and $P_{g, n}(G)$ is a dense, open subscheme of $K_{C_{\Gamma}, \vec{p}_{\Gamma}}(G).$
\end{theorem}

\noindent
When $n >0,$ $M_{g, n}(G) = G^{g+n-1}$, so as a consequence $B_{C_{\Gamma}, \vec{p}_{\Gamma}}(G)$ is rational when $n > 0.$ When $n = 0$ we have $K_{C_{\Gamma}}(G) = B_{C_{\Gamma}}(G) = M_{C_{\Gamma}}(G)$ and $M_{g, 0}(G) = \mathcal{X}(F_g, G)$, proving Theorem \ref{maing0}.

The proof of Theorem \ref{maingen1} in Section \ref{stablecurvesection} uses a description of conformal blocks of level $L$ as regular
functions $f \in \C[M_{g, n}(G)]$ which satisfy $v_i(f) \leq L$ for a collection of discrete valuations $v_i: \C[M_{g, n}(G)] \to \Z \cup \{-\infty\},$ which are in bijection with the vertices of $\Gamma$.  In the $SL_2(\C)$ case we use these valuations to give a combinatorial description of conformal blocks
and a stratification of $B_{C_{\Gamma}, \vec{p}_{\Gamma}}(SL_2(\C))$ and $K_{C_{\Gamma}, \vec{p}_{\Gamma}}(SL_2(\C)).$ 

In Section \ref{sl2theory} we show that the valuations $v_i$ can be computed in a straightforward way on the class of regular functions $\C[M_{g, n}(SL_2(\C))]$ associated to isotopy classes of paths in $\Gamma$ mentioned above.   For one of these functions $\tau \in \C[M_{g, n}(SL_2(\C))]$, $v_i(\tau)$ is the number of paths passing through the $i-$th vertex of $\Gamma$, so we obtain the filtration from Theorem \ref{maingeo}.  In particular the set of monomials in  isotopy classes which are planar with respect to the ribbon structure on $\Gamma$ form a basis of $V_{C_{\Gamma}, \vec{p}}(L)$. In the case $n =0$ this gives an interesting interpretation of the dimension of $V_{C_{\Gamma}}(L)$ as the set of monomials in cyclic equivalence classes of words in the free group $F_g$ which satisfy certain length conditions. 

Closed stratifications of $B_{C_{\Gamma}, \vec{p}_{\Gamma}}(SL_2(\C))$ and $K_{C_{\Gamma}, \vec{p}_{\Gamma}}(SL_2(\C))$  are obtained by intersecting the irreducible components of their boundary divisors $D_{\Gamma} = B_{C_{\Gamma}, \vec{p}_{\Gamma}}(SL_2(\C))$ $\setminus M_{g,n}(SL_2(\C))$, $E_{\Gamma} = K_{C_{\Gamma}, \vec{p}_{\Gamma}}(SL_2(\C))$ $\setminus P_{g,n}(SL_2(\C))$. 

\begin{theorem}\label{strattheorem}
The following hold for $D_{\Gamma}$ and $E_{\Gamma}$. 

\begin{enumerate}
\item The intersection of all the irreducible components in $D_{\Gamma}$ (respectively $E_{\Gamma}$) is a connected projective toric variety $D_{V(\Gamma)}$ (respectively $E_{V(\Gamma)}$).
\item If $\Gamma$ has a leaf, or is not bipartite, the stratification poset is a Boolean lattice on the set of vertices $V(\Gamma)$, and the codimension of a stratum obtained by intersecting a set of components $S$ is $|S|.$
\item If $\Gamma$ has no leaves and is bipartite, the stratification poset is the Boolean lattice on $V(\Gamma)$ modulo the lattice ideal composed of those sets $T$ with $V(\Gamma) \setminus T$ contained in one of the sets of the partition defined by the bipartite structure.  The codimension of a stratum obtained by intersecting a set of irreducible components $S$ which is not in this ideal is $|S|$.
\end{enumerate}

\end{theorem}

\noindent
 Theorem \ref{strattheorem} is proved by showing that the toric degeneration of $K_{C_{\Gamma}, \vec{p}_{\Gamma}}(SL_2(\C))$ constructed in \cite{M4} extends to $B_{C_{\Gamma}, \vec{p}_{\Gamma}}(SL_2(\C))$, and that the stratification on these spaces by components of their boundary divisors degenerates componentwise to part of a stratification of the toric degenerations by torus orbits.  In particular, toric degenerations of the spaces $M_{g, n}(SL_2(\C))$  and $P_{g, n}(SL_2(\C))$ are constructed, and the associated convex bodies are shown to be dense open sets in the polytopes assigned to $K_{C_{\Gamma}, \vec{p}_{\Gamma}}(SL_2(\C))$ and $B_{C_{\Gamma}, \vec{p}_{\Gamma}}(SL_2(\C)).$  The integral points in these convex bodies are precisely the $SL_2(\C)$ spin diagrams with topology $\Gamma$, so Theorem \ref{strattheorem} may be seen as an illustration of how the spin diagram combinatorics determines geometric structures. 

  Lifts of the generators and relations for the coordinate rings of toric degenerations of $K_{C_{\Gamma}, \vec{p}_{\Gamma}}(SL_2(\C))$ and $B_{C_{\Gamma}, \vec{p}_{\Gamma}}(SL_2(\C))$ give presentations of the coordinate rings of these spaces, and Theorem \ref{maingeo} provides these lifts: homogeneous skein relations on isotopy classes of paths. The fact that the spaces $B_{C_{\Gamma}, \vec{p}_{\Gamma}}(SL_2(\C))$ and $K_{C_{\Gamma}, \vec{p}_{\Gamma}}(SL_2(\C))$ sit in a flat family with the corresponding spaces over smooth curves $(C, \vec{p})$ then implies that the equations which cut out  $B_{C, \vec{p}}(SL_2(\C))$ and $K_{C, \vec{p}}(SL_2(\C))$ are generically (in $(C, \vec{p})$) deformations of skein ideals.  This should be compared to the conjectural presentation in the $g = 0$ case in \cite{StV}. Similar degenerations are also studied in \cite{StXu} in the $ g= 0$ case and \cite{A} in the $n = 0$ case.

\subsection{Remarks, Questions, and Conjectures}

The construction in Section \ref{mgnsection} implies a procedure for compactifying the character variety $\mathcal{X}(F_g, G)$. 

\begin{proposition}\label{ucompact}
For every $G^3$-equivariant compactification of the space $M_{0, 3}(G) = G^2$ (where the actions are the left diagonal action and the two right actions, see Section \ref{mgnsection}), there is a compactification of $\mathcal{X}(F_g, G)$ for any trivalent graph $\Gamma$ with $\beta_1(\Gamma) = g.$
\end{proposition}

 Roughly, one looks for $G^3-$stable filtrations on $\C[M_{0, 3}(G)]$, these translate to combinatorial features of structures on a ribbon graph related to the representation theory of $G$.  It would be interesting to characterize which of these compactifications lead to $\Gamma-$compactifications of $M_{g, n}(G)$ which are all flat deformations of each other.  Our construction arrises from $G^3$-stable valuations which stem from the theory of conformal blocks, extrapolating from our results for $SL_2(\C)$ in Section \ref{compact}, we conjecture that these always have a simple description in type $A,$ see Section \ref{algebra} for the relevant definitions.

\begin{conjecture}\label{l1gen}
The algebra $W_{0, 3}(SL_m(\C))$ is generated by the extended conformal blocks of level $1.$ 
\end{conjecture}

 We note that results in \cite{M10} imply this conjecture in the case $G = SL_3(\C).$   A resolution could come from developing a standard monomial theory on the coordinate ring of $M_{0, 3}(SL_m(\C)) = SL_m(\C) \times SL_m(\C)$ which is compatible with the filtration defined by the  conformal blocks. It would be interesting to relate such a structure to the cluster algebra structure on (dense open subschemes of) the spaces $M_{g, n}(G)$ and $P_{g, n}(G)$ given in \cite{FG} (see also \cite{MSW}).  

A relationship between $\mathcal{X}(F_g, G)$ and $M_C(G)$ for $C$ a smooth curve is studied by Florentino in \cite{F}, in relation to Schottky normalization of vector bundles on curves. We do not know if our realization of $\mathcal{X}(F_g, G)$ as a dense open subset of $M_{C_{\Gamma}}(G)$ for singular curves is in any sense a degeneration of Florentino's map, but it would be interesting to determine if a relationship between these constructions exists.   It also remains to relate the spaces $B_{C, \vec{p}}(G)$ to other ''universal" spaces of principal bundles, for example the space constructed by Bhosle, Biswas, and Hurtubise, \cite{BBH} in the type $A$ case.

The proof of Theorem \ref{strattheorem} constitutes part of an analysis of the face poset of the phylogenetic statistical polytopes (or rather the ''spin diagram polytope") studied in \cite{BBKM}, \cite{BW}, and \cite{Bu}. Each of these polytopes provides a combinatorial model for the spaces we study, so a full description of this poset would be interesting for phylogenetics, the character varieties, and the combinatorics of spin diagrams.  These polytopes also have generalizations for other groups, see \cite{KM}, \cite{M10}. 

\subsection{Acknowledgements}
We thank Sean Lawton for many useful conversations about free group character varieties, Neil Epstein for his helpful remarks
on Rees algebras, Geir Agnarsson for sharing his knowledge of graph theory, and Steven Sam for a helpful discussion on the material in Subsection \ref{genus3}.  We also thank Kaie Kubjas and Nick Early for useful remarks on an earlier version of this manuscript.

\subsection{Outline of the paper}

In Section $\ref{graphsection}$ we give the background on the graphs $\Gamma$ which stratify $\bar{\mathcal{M}}_{g, n}.$  In Section $\ref{mgnsection}$ we give a construction of the spaces $M_{g, n}(G).$ In Sections $\ref{factorcor}, \ref{algebra}, \ref{stablecurvesection}$ we bring in the theory of conformal blocks, and prove the Rees algebra statement Theorem \ref{maingen1}.  In Section \ref{sl2theory} we specialize to the $SL_2(\C)$ case and describe the skein algebra structure on $M_{g, n}(SL_2(\C))$.  In Section $\ref{compact}$ we explore the Skein relations in the setting of conformal blocks, and in Section $\ref{stratificationsection}$ we describe the boundary stratifications and prove Theorem \ref{strattheorem}.

\newpage

\tableofcontents

\section{Graphs and curves}\label{graphsection}

In what follows $\Gamma$ denotes a graph with non-leaf vertices $V(\Gamma)$ and edges $E(\Gamma)$.  We let $L(\Gamma) \subset E(\Gamma)$ denote
the set of edges which contain a leaf.  For a vertex $v \in V(\Gamma),$ $\eta(v)$ is the set of edges which contain $v$ and $n(v) = |\eta(v)|$ is the valence of $v.$  The Deligne-Mumford stack $\bar{\mathcal{M}}_{g, n}$ of stable projective curves comes with a stratification $\cup \mathcal{M}_{g, n}(\Gamma, \gamma)$ by stability type. This data is captured in a graph $\Gamma$ with a labelling $\gamma: V(\Gamma) \to \Z_{\geq 0}$ called the internal genus, this satisies the following conditions.

\begin{definition}[semistable graph]
A labelled graph $[\Gamma, \gamma]$ is said to be semistable of genus $g$ if $2\gamma(v) + \eta(v) \geq 3$ for each vertex $v \in V(\Gamma)$ and $\sum \gamma(v) + \beta_1(\Gamma) = g.$ 
\end{definition}

We let $[\hat{\Gamma}, \gamma]$ be the labelled forest obtained from $[\Gamma, \gamma]$ by splitting each internal edge of $\Gamma.$  Notice that the non-leaf vertices of $\hat{\Gamma}$ can be canonically identified with those of $\Gamma,$ and for each vertex $v$ there is a connected component $\Gamma_v \subset \hat{\Gamma}$ isomorphic to its link in $\Gamma.$

\begin{figure}[htbp]
\centering
\includegraphics[scale = 0.2]{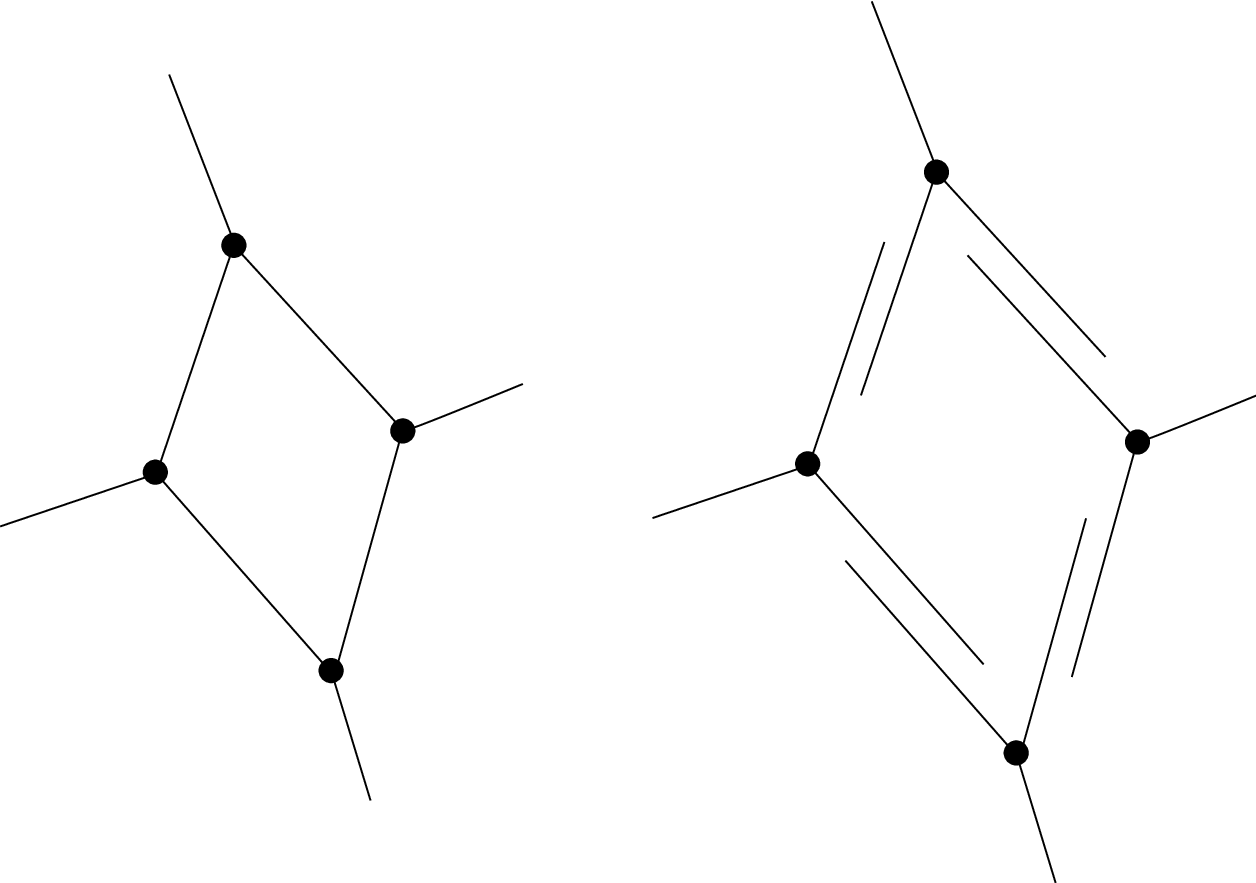}
\caption{A graph $\Gamma$ with forest $\hat{\Gamma}$}
\label{gammahatgamma}
\end{figure}

There is a natural quotient graph $\Gamma/e$ obtained by deleting an interior edge $e,$ and bringing its endpoints $v, w$ together to form a vertex $u$ with label $\gamma(v) + \gamma(w)$. We call a composition of these maps an admissable map. Admissable maps define a partial ordering on the genus $g$ semistable graphs with $n$ leaves, where $[\Gamma', \gamma'] \leq [\Gamma, \gamma]$ if $\Gamma$ can be obtained from $\Gamma'$ by a collapsing a subset $S \subset E(\Gamma).$  A stratum $\mathcal{M}_{g, n}(\Gamma', \gamma')$ appears in the boundary of $\mathcal{M}_{g, n}(\Gamma, \gamma)$ precisely when $[\Gamma', \gamma'] \leq [\Gamma, \gamma].$   A point $(C, \vec{p}) \in \mathcal{M}_{g, n}(\Gamma, \gamma)$ can be viewed as a union
of smooth marked curves $(C_v, \vec{q}_v)$ with a bijection between $\vec{q}_v$ and the edges in $\eta(v) = E(\Gamma_v)$, and $genus(C_v) = \gamma(v).$  This decomposition is canonical, and gives an isomorphism $\mathcal{M}_{g, n}(\Gamma, \gamma) \cong \prod_{v \in V(\Gamma)} \mathcal{M}_{\gamma(v), n(v)}.$

We will focus on trivalent graphs $[\Gamma, \gamma]$ with $\gamma(v) = 0$ for all $v \in V(\Gamma),$ when it is clear that we are using this flavor of semistable graph, we will drop the $\gamma.$   The following proposition establishes that these graphs are all connected by a series of local operations, for a proof see \cite{Bu}, Section 3.

\begin{definition}
We say two trivalent graphs $\Gamma, \Gamma'$ are mutation equivalent if there is a sequence of graphs $\Gamma = \Gamma_1, \ldots, \Gamma_k = \Gamma'$ such that $\Gamma_i/ e_i = \Gamma_{i+1}/e_{i+1}$ for some edges $e_i \in E(\Gamma_i)$.
\end{definition}

\begin{figure}[htbp]
\centering
\includegraphics[scale = 0.3]{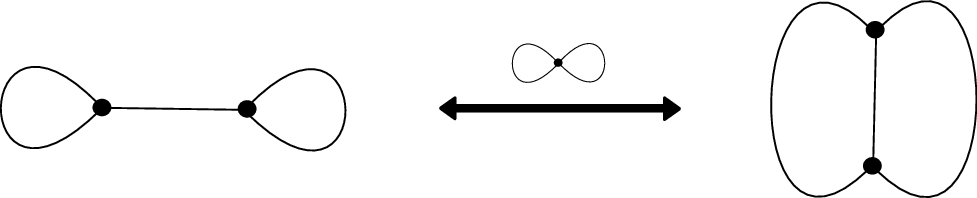}
\caption{Mutation of two genus $2$ graphs.}
\label{mutationeq}
\end{figure} 

\begin{proposition}\label{mutantequivalent}
Any two trivalent graphs with the same genus $g$ and number of leaves $n$ are mutation equivalent. 
\end{proposition}

\noindent
The stratum $\mathcal{M}_{g, n}(\Gamma)$ corresponding to a trivalent graph is a product of copies of $\mathcal{M}_{0, 3} = \{pt\}$ by the remarks above, and is therefore a single point $(C_{\Gamma}, \vec{p}_{\Gamma}).$

We define $\Gamma_{g, n}$ to be the unique graph with one vertex, $n$ leaves, and $\beta_1(\Gamma_{g, n}) = g.$  For any graph $\Gamma$ with $\beta_1(\Gamma) = g,$ and $n$ leaves we can find an admissable map $\pi_{\tree}: \Gamma \to \Gamma_{g, n}$ by collapsing a spanning tree $\tree \subset \Gamma$ to a single vertex.  We note that this map places the loop edges of $\Gamma_{g, n}$ in bijection with the complement of $E(\tree)$ in
$E(\Gamma).$

\section{Valuations, filtrations, and Rees algebras}\label{valfilrees}

In this section we review properties of valuations and filtrations on a commutative domain. We describe  a situation which allows two valuations to be added with non-negative coefficients to obtain another valuation, this provides technical details we will need in Section \ref{sl2theory}.   Then we review the Rees algebra construction associated to a filtration, this will be used to prove Theorems \ref{maingen1} and  \ref{strattheorem}.   

\subsection{Valuations and filtrations}\label{sumval}

In what follows, a rank $1$ valuation $v: A \to \R \cup \{-\infty\}$ on a commutative $\C$-domain is a function which satisfies
the following properties. 

\begin{enumerate}
\item $v(ab) = v(a) + v(b)$\\
\item $v(a + b) \leq max\{v(a), v(b)\}$\\
\item $v(C) = 0$ for all $C \in \C^*.$\\
\item $v(0) = -\infty$\\
\end{enumerate}

\noindent
In particular we deal exclusively with valuations which induce the trivial valuation on $\C \subset A.$  A valuation defines an increasing filtration
by the subspaces $v_{\leq k}(A) \subset A$, along with an associated graded algebra  $\bigoplus_{k \in v(A)} v_{\leq k}(A)/v_{< k}(A)$.  An increasing filtration $F$ by real numbers on a domain $A$ comes from a valuation precisely when $\C^* \subset F_{\leq 0} \setminus F_{<0}$ and its associated graded algebra $gr_F(A)$ is also a domain.   From now on we use $v$ to denote both the valuation and the induced filtration.

Let $v$ and $w$ be valuations on domains $A$ and $B$ respectively, these functions can be used to define filtrations on the tensor product (always taken over $\C$) $A \otimes B$. It is easy to check that both of these filtrations come from valuations. 

\begin{equation}
v_{\leq k}(A \otimes B) = v_{\leq k}(A) \otimes B \ \ \ \ w_{\leq k}(A \otimes B) = A \otimes w_{\leq k}(B)\\
\end{equation}

The sum of two filtrations $v, w$ is defined by the subspaces, 
 
\begin{equation}
(v \oplus w)_{\leq m} = \sum_{i + j \leq m} v_{\leq i} \cap w_{\leq j}.\\
\end{equation}

\noindent
Similarly, the multiple $Rv$ , $R \in \R_{\geq 0}$ of a filtration is defined by setting $(Rv)_{\leq m} = \sum_{k \leq \frac{m}{R}} v_k(A) \subset A$, 
in particular $Rv \oplus Sv = (R+ S)v$.

\begin{lemma}\label{val2}
Let $v, w: A \otimes B \to \R \cup \{-\infty\}$ be as above. Any linear combination $S v \oplus T w$, $S, T \in \R_{\geq 0}$ 
defines a valuation on $A \otimes B.$
\end{lemma}

\begin{proof}
The associated graded algebra of $v$ as a valuation on $A \otimes B$ is $gr_v(A) \otimes B$, which is manifestly a domain.  This implies that $v$ and $w$ define valuations on $A \otimes B.$   Now we consider the sum valuation $S v \oplus T  w$.  By the first part, without loss of generality, we may take $S, T = 1$.  We must prove that the associated graded algebra is a domain, we claim that it is in fact isomorphic to $gr_v(A) \otimes gr_w(B)$. We compute the quotient $(v \oplus w)_{\leq m}/ (v \oplus w)_{< m}$.

\begin{equation}
(v \oplus w)_{\leq m}/ (v \oplus w)_{< m} = \sum_{i + j \leq m} v_{\leq i} \cap w_{\leq j} \  / \sum_{i + j < m} v_{\leq i} \cap w_{\leq j}\\\\
\end{equation}

\noindent
The space $v_{\leq i}\cap w_{\leq j}$ is equal to $v_{\leq i} \otimes w_{\leq j},$ and $v_{\leq i} \otimes w_{\leq j} \subset v_{\leq i'} \otimes w_{\leq j'}$ if and only if $i \leq i'$ and  $j \leq j'$.  The space $(v \oplus w)_{\leq m}/ (v \oplus w)_{< m}$ is spanned by the images of the spaces $ v_{\leq i} \otimes w_{\leq j}$ with $i + j = m$, and the intersection of two of these spaces $v_{\leq i} \otimes w_{\leq j} \cap v_{\leq i'} \otimes w_{\leq j'}$ is $v_{\leq min\{i, i'\}} \otimes w_{\leq min\{j, j'\}}$, which is contained in $ (v \oplus w)_{< m}$. It follows that the above quotient simplifies to the direct sum of the spaces $v_{\leq i} \otimes w_{\leq j} / v_{< i} \otimes w_{\leq j} + v_{\leq i} \otimes w_{<j} = v_{\leq i}/v_{< i} \otimes w_{\leq j}/ w_{< j},$ this proves the claim. 
\end{proof}

\noindent
Lemma \ref{val2} can be applied to two valuations $v, w: C \to \R \cup \{-\infty\}$ induced from an inclusion of algebras $C \subset A \otimes B$, this is how we use it in Section \ref{sl2theory}.

\subsection{Rees algebras}

Let $A$ be a domain over $\C$ with an increasing $\Z_{\geq 0}-$filtration $F$, the Rees algebra of $A$ with respect to $F$
is defined to be the following graded subring of $A[t].$

\begin{equation}
R = \bigoplus_{m \geq 0} F_m \subset \bigoplus_{m \geq 0} A = A[t]\\
\end{equation}

\noindent
 The following are standard results on Rees algebras, see e.g. \cite{AB}, Section 2. We let $t \in R$ be the copy of the identity in $F_1.$

\begin{enumerate}
\item $R$ is flat over the polynomial ring $\C[t] \subset R,$\\
\item $\frac{1}{t}R \cong A[t, \frac{1}{t}],$\\
\item $R/t \cong gr_F(A).$\\
\end{enumerate}

Suppose that $F_0 = \C,$ we consider $\bar{X} = Proj(R).$ This scheme is covered by the open neighborhoods $D(f) = \{[p] \subset R | f \notin p\}$ for $f \in R$ homogenous, with $\bar{X} \setminus D(f) = Proj(R/f).$  Here $D(f)$ is the affine scheme obtained by taking the spectrum of the degree $0$ subalgebra $[\frac{1}{f}R)]_0 \subset \frac{1}{f}R.$  Points $(2)$ and $(3)$ above imply that $\bar{X}$ is a disjoint union of the open set $X = D(t) = Spec(A)$ and the hypersurface $Proj(gr(A)) = \bar{X} \setminus X.$   Finite generation of the algebra $gr(A)$ is equivalent to finite generation of $R$ (see e.g. \cite{AB}, Proposition 2.2), and implies finite generation of $A.$ If this is the case, $\bar{X}$ embeds as a closed subscheme in a weighted projective space, and is therefore projective.

\subsection{Valuation from a nilpotent operator}\label{nilval}

Let $e: A \to A$ be a $\C$-linear nilpotent derivation on a domain $A.$    Let $(v_e)_{\leq L}(A)$ be the space of elements $a \in A$ such that $e^{L+1}(a) = 0.$ If $e^K(a) \neq 0, e^L(b) \neq 0,$ $e^{K+1}(a) = e^{L+1}(b) = 0,$ then  $e^{L+K }(ab) = \sum \binom{K+L}{m} e^{K+L - m}(a)e^m(b)$ $= \binom{K+L}{K}e^K(a)e^L(b) \neq 0$ , and $e^{K+L+1}(ab) = 0.$ It follows that the associated graded ring $gr_{v_e}(A)$ is a domain, and $v_e$ defines a valuation by the formula $v_e(a) = min\{ L | e^{L+1}(a) = 0\}$.

\subsection{$G-$stable filtrations}

We will also need the following lemma, which relates the associated graded algebra of a filtration by representations
of a reductive group $G$ to the associated graded algebra of the invariant ring.  We say a filtration $F$ is $G-$stable
if each space $F_{\leq m} \subset A$ is a $G-$representation. 

\begin{lemma}\label{val1}
Let $A$ be a $\C$ algebra with a rational action by a reductive group $G$, and let $v$ be a $G-$stable valuation, then the following holds. 

\begin{equation}
gr_v(A^G) \cong (gr_v(A))^G\\
\end{equation}

\end{lemma}

\begin{proof}
This is a consequence of the exactness formulation of reductivity of $G$, as applied to the short exact sequence $0 \to v_{< m} \to v_{\leq m} \to v_{\leq m}/v_{< m} \to 0.$
\end{proof}

\section{The scheme $M_{g, n}(G)$}\label{mgnsection}

In this section define an affine scheme $M_{\Gamma}(G)$ for every semistable graph $\Gamma$ with $\gamma(v) = 0$ for every $v \in V(\Gamma).$ We  show that the isomorphism type of $M_{\Gamma}(G)$  depends only on the number of leaves the first Betti number.   We also give a related construction which depends on the choice of an orientation on $\Gamma.$  First we review important structural features of the coordinate ring of a reductive group $G$ which will be used in Sections \ref{mgnsection} and \ref{algebra}.  In what follows we use the shorthand $\vec{\lambda}$ to mean a tuple of dominant weights $(\lambda_1, \ldots, \lambda_n)$ and $V(\vec{\lambda})$ to mean a tensor product of irreducible representations $V(\lambda_1) \otimes \ldots \otimes V(\lambda_n).$

\subsection{The ring $\C[G]$}\label{Gring}

We review some $G-$linear algebra which plays a part in defining the multiplication operation in the coordinate ring $\C[G]$, this is needed in Section \ref{algebra} .   Recall that the Peter-Weyl theorem gives a decomposition of the coordinate ring $\C[G]$ as the sum of the endomorphism spaces
of the irreducible representations of $G.$

\begin{equation}
\C[G] = \bigoplus_{\lambda \in \Delta_G} End(V(\lambda))\\
\end{equation}

We let $v_{\lambda}, v_{-\lambda^*} \in V(\lambda)$ be highest, respectively lowest weight vectors, and $p_{\lambda}: V(\lambda, \lambda^*) \to \C$ be the unique $G-$invariant map which sends $v_{\lambda}\otimes v_{-\lambda}$ to $1.$  This choice determines $G-$invariant isomorphisms $V(\lambda)^* \cong V(\lambda^*)$, $V(\lambda, \lambda^*) \cong End(V(\lambda)).$ The function $p_{\lambda}$ is identified with the trace map $Tr_{\lambda}: End(V(\lambda)) \to \C$ under this second isomorphism, similarly we let $O_{\lambda} \in V(\lambda, \lambda^*)$ be the tensor which maps to the identity $I_{\lambda} \in End(V(\lambda))$.  Considered as an endomorphism, an element $v \otimes f \in V(\lambda, \lambda^*)$ acts on $V(\lambda)$ by sending $w \in V(\lambda)$ to $v \otimes p_{\lambda}(w, f)$, the pair $v \otimes f \in V(\lambda, \lambda^*)$ defines a regular function on the group $G$ as follows. 

\begin{equation}
g \to p_{\lambda}(g^{-1}v, f)\\  
\end{equation}

\noindent
 Note that $p_{\lambda}$ and $O_{\lambda^*}$ are dual to each other under the identification $End(V(\lambda))^* \cong End(V(\lambda^*)),$ and $p_{\lambda}(O_{\lambda}) = dim(V(\lambda)).$  For any pair $\alpha, \beta$ of dominant weights, the representations $V(\alpha)\otimes V(\beta)$ and $V(\alpha^*)\otimes V(\beta^*)$ have $G-$linear decompositions, 

\begin{equation}
V(\alpha)\otimes V(\beta) = \bigoplus_{\eta} Hom(V(\eta), V(\alpha)\otimes V(\beta))\otimes V(\eta),\\
\end{equation}

\begin{equation}
V(\alpha^*)\otimes V(\beta^*) = \bigoplus_{\eta} Hom(V(\eta^*), V(\alpha^*)\otimes V(\beta^*))\otimes V(\eta^*).\\
\end{equation}

For two intertwiners $f \in $ $Hom(V(\eta), V(\alpha)\otimes V(\beta))$ and $g \in Hom(V(\eta^*), V(\alpha^*)\otimes V(\beta^*))$, the map $p_{\alpha}\otimes p_{\beta}\circ(f \otimes g): V(\eta, \eta^*) \to \C$ is $G-$invariant, this implies that it is a multiple $T_{\eta}^{\beta, \alpha}(f, g)$ of $p_{\eta}$.  The following are consequences of elementary linear algebra. 

\begin{lemma}
The map $T_{\eta}^{\beta, \alpha}: Hom(V(\eta), V(\alpha)\otimes V(\beta) \otimes Hom(V(\eta^*), V(\alpha^*)\otimes V(\beta^*) \to \C$ is non-degenerate, and gives an isomorphism $Hom(V(\eta), V(\alpha)\otimes V(\beta))^* \cong  Hom(V(\eta^*), V(\alpha^*)\otimes V(\beta^*)).$
Furthermore,  

\begin{equation}\label{expand}
p_{\alpha}\otimes p_{\beta} = \sum T_{\eta}^{\beta, \alpha} \otimes p_{\eta},\\
\end{equation}

\noindent
holds in $[V(\alpha) \otimes V(\alpha^*) \otimes V(\beta) \otimes V(\beta^*)]^*.$
\end{lemma}

 For $(v \otimes f) \otimes (w \otimes h)\in V(\alpha, \alpha^*)\otimes V(\beta, \beta^*),$ we can decompose $v \otimes w = \sum F_{\eta}\otimes X_{\eta},$ and $f \otimes h = \sum K_{\eta^*} \otimes Y_{\eta^*}$.  The product is then the function which takes $g \in G$ to the following complex number. 

\begin{equation}
 p_{\alpha}(g^{-1}v, f) p_{\beta}(g^{-1}w, h) = \sum T_{\eta}^{\alpha, \beta}(F_{\eta}, K_{\eta^*})p_{\eta}(g^{-1}X_{\eta}, Y_{\eta^*}).\\ 
\end{equation}

 The linear extension of this map to non-simple tensors is then the multiplication map $m: V(\alpha, \alpha^*)\otimes V(\beta, \beta^*) \to \bigoplus_{\eta} V(\eta, \eta^*)$ in $C[G].$ We let $I_{\eta}^{\beta, \alpha}$ $\in Hom(V(\eta), V(\alpha)\otimes V(\beta)) \otimes Hom(V(\eta^*), V(\alpha^*)\otimes V(\beta^*))$ be the element which represents the identity on the vector space $Hom(V(\eta), V(\alpha)\otimes V(\beta)) \otimes V(\eta)$, by dualizing Equation \ref{expand} we get $O_{\alpha} \otimes O_{\beta} =$ $\sum I_{\eta}^{\alpha, \beta} \otimes O_{\eta}.$  Let $q_{\eta}^{\alpha, \beta}: V(\eta, \eta^*) \to V(\alpha, \alpha^*) \otimes V(\beta, \beta^*)$ be the map which sends $\Phi \in V(\eta, \eta^*)$ to $I_{\eta}^{\beta, \alpha} \otimes \Phi.$ The following is a consequence of this discussion. 

\begin{lemma}\label{multlemma}
The dual of the multiplication map $m: V(\alpha, \alpha^*) \otimes V(\beta, \beta^*) \to \bigoplus_{\eta} V(\eta, \eta^*)$ is the following
$G\times G$ intertwiner, where $\leq$ is the dominant weight ordering. 

\begin{equation}
q = \sum q_{\eta^*}^{\alpha^*, \beta^*} : \bigoplus_{\eta \leq \alpha + \beta} V(\eta^*, \eta) \to V(\alpha^*, \alpha) \otimes V(\beta^*, \beta)\\
\end{equation}

\end{lemma}

 We also make use of the isomorphism $G = [G\times G]/G$, and the corresponding isomorphism of coordinate rings, $\C[G] \cong \C[G\times G]^G.$  On the level of spaces, this isomorphism is given by the map $\phi: G\times G \to G$, $\phi(g, h) = gh^{-1}.$  A section $\psi: G \to G\times G$ is given by $\psi(g)  = (g, 1)$. This is a special case of the isomorphism $[X \times G]/G \cong X$
for any $G-$space, given by $(x, h) \to hx$ with inverse $x \to (x, 1).$

\begin{equation}
\C[G \times G] = \bigoplus_{\lambda, \eta} V(\lambda, \lambda^*)\otimes V(\eta, \eta^*)\\
\end{equation}

\begin{equation}
\C[G \times G]^G =  \bigoplus_{\eta} V(\eta^*)\otimes V(\eta) \otimes \C O_{\eta} \subset \bigoplus_{\eta} [V(\eta^*)\otimes V(\eta)] \otimes [V(\eta)\otimes V(\eta^*)]\\
\end{equation}

In the equation above, $O_{\eta}$ is supported on the second and fourth indices.
The map $\psi^* : \C[G\times G]^G \to \C[G]$ takes $f \otimes v \otimes O_{\eta} \in  V(\eta, \eta^*) \otimes \C O_{\eta}$ to the function $[g \to dim(V(\eta))f(g^{-1}v)] \in \C[G]$, so $\psi^*(V(\eta)\otimes V(\eta^*) \otimes \C O_{\eta}) = V(\eta, \eta^*)$.  We let $\psi_*$ be the linear dual of $\psi^*$, the following lemma is a consequence of the fact that $\psi^*$ is a ring homomorphism.

\begin{lemma}\label{multcommute}
Let  $\alpha = \alpha_1 = \alpha_2^*$, $\beta = \beta_1 = \beta_2^*,$ and $\eta = \eta_1 = \eta_2^*$ for ease of notation. 
The following diagram commutes. 

$$
\begin{CD}
V(\alpha, \alpha^*) \otimes V(\beta, \beta^*)  @>\psi_* \otimes \psi_*>> V(\alpha_1, \alpha_1^*)\otimes V(\alpha_2, \alpha_2^*) \otimes V(\beta_1, \beta_1^*) \otimes V(\beta_2, \beta_2^*)\\
@A\sum q_{\eta}^{\alpha, \beta}AA @A\sum q_{\eta_1}^{\alpha_1, \beta_1} \otimes q_{\eta_2}^{\alpha_2, \beta_2}AA\\\
\bigoplus_{\eta} V(\eta, \eta^*) @>\psi_*>> \bigoplus_{\eta_1, \eta_2} V(\eta_1, \eta_1^*)\otimes V(\eta_2, \eta_2^*)\\
\end{CD}
$$

In particular  $q_{\eta}^{\alpha, \beta}: V(\eta, \eta^*) \to V(\alpha, \alpha^*)\otimes V(\beta, \beta^*)$ is identified with $q_{\eta}^{\alpha, \beta} \otimes q_{\eta^*}^{\alpha^*, \beta^*} :$ $[V(\eta, \eta^*) \otimes V(\eta^*, \eta)]^G \to$ $ [V(\alpha, \alpha^*) \otimes V(\alpha^*, \alpha)]^G \otimes [V(\beta, \beta^*) \otimes V(\beta^*, \beta)]^G$.
\end{lemma}

\subsection{The affine scheme $M_{g, n}(G)$}

Now we define an affine scheme $M_{\Gamma}(G)$ for every finite graph $\Gamma$.  We show that if $\Gamma$ and $\Gamma'$ have
the same number of leaves $n$ and first Betti number $g$, then $M_{\Gamma}(G) \cong M_{\Gamma'}(G)$.  Accordingly, we define $M_{g, n}(G)$ to be $M_{\Gamma}(G)$ for any $\Gamma$ with compatible data.  We also show that $M_{g, 0}(G) = \mathcal{X}(F_g, G)$ and $M_{g, n}(G) = G^{g+n-1}$ for $n >0.$ 

\begin{definition}
The space $M_{0, n}(G)$ is defined to be the following left $GIT$ quotient.

\begin{equation}
M_{0, n}(G) = G \backslash G^n\\
\end{equation}
\end{definition}

We assign a copy of $M_{0, n(v)}(G)$ to each connected component $\Gamma_v \subset \hat{\Gamma},$ each edge in $\Gamma_v$ has a corresponding action of $G$ on $M_{0, n(v)}(G),$ acting on the right.   Each non-leaf edge $e \in E(\Gamma)$ therefore has an associated action of $G\times G$ on the product $M_{\hat{\Gamma}}(G) = \prod_{v \in V(\Gamma)} M_{0, n(v)}(G)$ defined by right actions corresponding to the two edges in $\hat{\Gamma}$ which map to $e$ under $\pi.$  We let $G^{E(\Gamma)}$ be the product of the diagonal subgroups $G \subset G\times G$ over the non-leaf edges in $\Gamma,$ and define $M_{\Gamma}(G)$ as follows.
 
\begin{equation}
M_{\Gamma}(G) = M_{\hat{\Gamma}}(G)/ G^{E(\Gamma)}\\
\end{equation}

\begin{proposition}\label{graph}
$M_{\Gamma}(G)$ depends only on the number of leaves $n$ and $\beta_1(\Gamma) = g.$ 
\end{proposition}

\begin{proof}
We show that $M_{\Gamma}(G) = M_{\Gamma'}(G)$ for any two graphs connected by an admissable map $\pi: \Gamma \to \Gamma'$, this proves the Proposition by Lemma \ref{mutantequivalent}.  We first treat the case of a tree $\tree$ with exactly two internal vertices $v, v' \in V(\tree)$ with valences $k, m$, and the admissable map $\pi: \tree \to \tree_w$ which takes $v, v'$ to a single vertex $w$. This forces $\tree_w$ to be a claw tree with $n = k + m -2$ leaves. 

\begin{equation}
M_{\tree}(G) = [M_{0,k}(G)\times M_{0, m}(G)]/G = [G \backslash G^k] \times [G\backslash G^m]/G\\
\end{equation}

The right hand action above is on the $k-$th and $k+1-$st components.  We identify $[G\times G] / G$ with $G$ as $G\times G$ varieties
using the isomorphisms defined by $(g, h) \to gh^{-1}$ and $g\to (g, Id)$. This implies that $M_{\tree}(G)$ is $G^2 \backslash G^{k+m-1},$ where the first component acts on the first $k$ indices on the left, and the second acts on the last $m-1$ indices on the left and the $k-th$ index on the right.  The variety $G\backslash G^k$ is likewise isomorphic to $G^{k-1}$ by the maps $(g_1, \ldots, g_k) \to (g_k^{-1}g_1, \ldots, g_k^{-1}g_{k-1})$ and $(g_1, \ldots, g_{k-1}) \to (g_1, \ldots, g_{k-1}, Id)$.  This takes the right $G^k$ action to the right $G^{k-1}$ action crossed with the left diagonal $G$ action, forming an isomorphism of $M_{\tree}(G)$ with $G \backslash G^{k +m -2}$, where the action is diagonal on the left.  This construction preserves the residual right $G^{k+m-2}$ action.  It follows that if $\phi: \Gamma \to \Gamma'$ is an admissable map which collapses a single edge in $E(\Gamma),$ we can appeal to the tree construction at this collapsed edge. The proposition then follows by induction.      
\end{proof}

\begin{corollary}
For an isomorphism derived from an admissable map $M_{\Gamma}(G) \cong M_{\Gamma'(G)}$, $\pi: \Gamma \to \Gamma'$, the pullback of a regular function $f \in  \C[M_{\Gamma}(G)]$ to $M_{\Gamma'}(G)$ is computed by plugging the identity $Id \in G$ in for the components associated to the edges collapsed by $\pi$. 
\end{corollary}

\begin{proof}
This follows by the proof of Proposition \ref{graph}, and the general fact that the isomorphism $X \cong [X \times G]/G$ 
is computed by the maps $x \to (x, 1), (x, h) \to hx.$
\end{proof}

 The graph $\Gamma_{g, n}$ has a single vertex, making $\hat{\Gamma}_{g, n}$ a claw tree with $2g + n$ leaves, and $M_{\hat{\Gamma}_{g, n}}(G) = G \backslash G^{2g + n}.$  In order to pass from $M_{\hat{\Gamma}_{g, n}}(G)$ to $M_{\Gamma_{g, n}}(G)$ we quotient by $G^g$ on the right hand side, where the $i-$th factor acts on the $i-$th and $g+i-$th components of $G^{2g +n}.$ Once again, by the identification $G^2/G = G,$ the resulting space is $G \backslash G^{g+n}$, where the first $g$ components have the adjoint action and the last $n$ components have the left action.  Formally, we define $M_{g, n}(G)$ to be $M_{\Gamma_{g, n}}(G)$.

The coordinate ring $\C[M_{0, n}(G)]$ has a direct sum decomposition into the following spaces. 

\begin{equation}
\C[M_{0, n}(G)] = \bigoplus_{\vec{\lambda} \in \Delta^n} (V(\lambda_1) \otimes \ldots \otimes V(\lambda_n))^G \otimes V(\lambda_1^*) \otimes \ldots \otimes V(\lambda_n^*)\\
\end{equation}

\noindent
The coordinate ring $\C[M_{\hat{\Gamma}}(G)]$ is the tensor product $\bigotimes_{v \in V(\Gamma)} \C[M_{0, n(v)}(G)]$, this is a direct
sum of the spaces below, where $\lambda: E(\hat{\Gamma}) \to \Delta$ is an assignment of dominant weights.  We let $e(v, k)$ be the $k-$th
edge incident on $v$, for some  ordering. 

\begin{equation}\label{component}
\bigotimes_{v \in V(\Gamma)}   (V(\lambda(e(v,1)) \otimes \ldots \otimes V(\lambda(e(v ,n(v)))^G \otimes V(\lambda(e(v, 1))^*) \otimes \ldots \otimes V(\lambda(e(v, n(v)))^*)\\
\end{equation}

\noindent
The coordinate ring $\C[M_{\Gamma}(G)] \subset \bigotimes_{v \in V(\Gamma)} \C[M_{0, n(v)}(G)]$ is obtained by taking right $G^{E(\Gamma)}$ invariants. 
One of the components in Equation \ref{component} contains $G^{E(\Gamma)}$ invariants if and only if $\lambda(e(v, i)) = \lambda(e(w, j))^*$ for any two edges $e(v, i)$, $e(w, j)$ which cover the same edge in $\Gamma$.  In this case the invariant space is the following tensor product. 

\begin{equation}
[\bigotimes_{v \in V(\Gamma)} (V(\lambda(e(v,1)) \otimes \ldots \otimes V(\lambda(e(v, n(v)))^G] \otimes V(\lambda(\ell_1)) \otimes \ldots \otimes V(\lambda(\ell_n))\\
\end{equation}

\noindent
The space $M_{\Gamma}(G)$ retains a $G^{L(\Gamma)}$ action, this is represented by $\ell_i$ components above.  The weight $\lambda(\ell_i)$ is dual
to the weight $\lambda(e(v, i))$ in the tensor product invariant space when the edge $e(v, i)$ contains $\ell_i$.

\subsection{An alternative construction of $M_{g, n}(G)$ and the character variety $\mathcal{X}(F_g, G)$}\label{alternative}

We will make use of another construction of $M_{g, n}(G)$ in Sections \ref{sl2theory} and \ref{compact}.  In what follows $\Gamma$ is a graph
with an orientation, namely edge $e \in E(\Gamma)$ is given a preferred direction which orders its endpoints $\delta(e) = (u, v).$  We consider the product
$G^{E(\Gamma)}$ with an action of $G^{V(\Gamma)}$ defined as follows.  For $\vec{g} = ( \ldots, g_e, \ldots) \in G^{E(\Gamma)}$ and $\vec{h} = (\ldots, h_u, \ldots, h_v, \ldots) \in G^{V(\Gamma)}$, the element $\vec{h}$ acts on $\vec{g}$ by the rule $\vec{h} \circ \vec{g} = (\ldots h_ug_eh_v^{-1}, \ldots).$  

\begin{lemma}
The $GIT$ quotient $G^{V(\Gamma)} \backslash G^{E(\Gamma)}$ is isomorphic to $M_{\Gamma}(G).$
\end{lemma}

\begin{proof}
 Recall that $M_{\Gamma}(G) = [\prod_{v \in V(\Gamma)} M_{0, n(v)}(G)]/G^{E(\Gamma)}$.  Using the definition of $M_{0, n}(G)$,  we may rewrite this product as $G^{V(\Gamma)} \backslash [\prod_{v \in V(\Gamma)} G^{n(v)}] /G^{E(\Gamma)}$ $= G^{V(\Gamma)} \backslash \prod_{e \in E(\Gamma)} [G\times G]/G.$ Now we use the fact that the isomorphism $[G \times G]/G \cong G$ identifies the residual $G \times G$ action on $[G \times G]/G$ with the standard $G\times G$ action on $G.$ 
Each pair $G \times G$ in the product above is unordered, in choosing an ordering, we determine an orientation on $\Gamma$, and an isomorphism with the $GIT$ quotient $G^{V(\Gamma)} \backslash G^{E(\Gamma)}.$
\end{proof}

When $n>0$, the space $M_{g, n}(G)$ is isomorphic to $G^{g+n -1}$ by the general principal $X\times G / G = X$ for any $G-$space $X,$ where the $G$ component has the left or right action. In the case $n = 0,$ we take any orientation on the graph $\Gamma_{g, 0}$, since $E(\Gamma_{g, 0})$ contains $g$ edges, and $V(\Gamma_{g, 0})$ has one element, $M_{\Gamma_{g, 0}}(G)$ is the $GIT$ quotient of $G^g$ by the conjugation action of $G$, this is the definition of the complex character variety $\mathcal{X}(F_g, G).$

\begin{equation}
M_{\Gamma_{g, 0}}(G) = \mathcal{X}(F_g, G)\\
\end{equation} 

\begin{remark}
In \cite{FL}, Florentino and Lawton define a similar construction of the space $M_{\Gamma}(G)$, which they call a non-commutative quiver variety.  In addition to our ''admissable maps", they also study the behavior of these spaces under a number of other operations on quivers, and they explore a similar construction for the character variety $\mathcal{X}(J, G)$ of a general discrete group $J.$
\end{remark}

\subsection{$P_{g, n}(G)$ and the universal configuration space $P_{0, n}(G)$}

The space $M_{g, n}(G)$ carries a $G^n$ action, with one copy of $G$ acting
for each leaf of $\Gamma_{g, n}.$  For a maximal unipotent subgroup $U \subset G$ we have
a $U^n$ action on $M_{g, n}(G)$ as well.  

\begin{definition}
The space $P_{g, n}(G)$ is defined to be the following non-reductive $GIT$ quotient. 

\begin{equation}
P_{g, n}(G) = M_{g, n}(G)/U^n
\end{equation}
\end{definition}

\noindent
The action of $U^n$ extends to the action of a reductive group $G^n$ on $M_{g, n}(G)$, so the quotient
$P_{g, n}(G)$ by this group is affine with a finitely generated coordinate ring, see \cite{Kir}. We use the notation
$P_{\Gamma}(G)$ when it is important to emphasize the graph $\Gamma.$

The space $P_{0, n}(G) = G \backslash [G/U]^n$ plays a special role in the algebraic geometry of the
group $G.$ Any flag variety $\mathcal{O}(\lambda) = Gv_{\lambda^*} \subset \mathbb{P}(V(\lambda^*))$
of $G$ can be constructed from $G/U$ by taking the $GIT$ quotient of this space by a maximal torus $T \subset G$
with respect to the character defined by $\lambda$.

\begin{equation}
[G/U]/_{\lambda} T = \mathcal{O}(\lambda)\\
\end{equation}

\noindent
The line bundle induced on $\mathcal{O}(\lambda)$ by this quotient is the canonical line bundle $\mathcal{L}_{\lambda}$, also
obtained by pulling back $\mathcal{O}(1)$ on $\mathbb{P}(V(\lambda^*))$.   It follows that a diagonal $GIT$ quotient $P_{\vec{\lambda}}(G) = G\backslash_{\vec{\lambda}} \prod \mathcal{O}(\lambda_i)$ can be obtained as a $T^n$ $GIT$ quotient of $P_{0, n}(G)$ with respect to the character
$\vec{\lambda}.$  The spaces $P_{\vec{\lambda}}(G)$ are called configuration spaces of $G-$flags, we refer to $P_{0, n}(G)$ as the universal configuration space accordingly.  The coordinate ring of $P_{0,n}(G)$ is a graded direct sum of the invariant spaces $(V(\lambda_1) \otimes \ldots \otimes V(\lambda_n))^G.$

\begin{equation}
\C[P_{0,n}(G)] = \bigoplus_{\vec{\lambda} \in \Delta^n}  V(\vec{\lambda})^G\\
\end{equation}

\section{Correlation and factorization on conformal blocks}\label{factorcor}

In this section we review the factorization and correlation constructions on conformal blocks.   For a more complete account see \cite{TUY}, the book of Ueno, \cite{U}, Chapters 3 and 4, and the papers of Beauville, \cite{B1} Part I, and Looijenga \cite{L} Section 4.

\subsection{Construction of conformal blocks}

Let $\mathfrak{g}$ be the affine Kac-Moody algebra associated to $\mathfrak{g}$, and let $H(0, L)$ be the integrable highest weight module of level $L$, weight $0$, and highest weight vector $v_{0, L}.$     For $\omega$ the longest root of $\mathfrak{g}$, $\Delta_L = \{\lambda \in \Delta | \lambda(H_{\omega}) \leq L\}$ denotes the level $L$ alcove of a Weyl chamber $\Delta$ of $\mathfrak{g}.$   For a collection of dominant weights $\vec{\lambda} = \{\lambda_1, \ldots, \lambda_n\} \subset \Delta_L$, we let $H(\vec{\lambda}^*, L) = H(0, L) \otimes V(\vec{\lambda}^*)$.  Following a construction of Beauville \cite{B1} Part I, (see also $10.1$ of \cite{LS}) the space of conformal blocks  $V_{C, \vec{p}}(\vec{\lambda}, L)$  is the space of invariants in the full vector space dual $H(\vec{\lambda}^*, L)^*$, with respect to an action of the Lie algebra $\C[C \setminus q]\otimes \mathfrak{g},$ for $q \in C$ some point $\neq p_i$.

\begin{equation}
V_{C, \vec{p}}(\vec{\lambda}, L) = [H(\vec{\lambda}^*, L)^*]^{\C[C \setminus \vec{p}]\otimes \mathfrak{g}}\\
\end{equation}

\begin{definition}
We define spaces of extended conformal blocks as follows. 

\begin{equation}
 W_{C, \vec{p}}(\vec{\lambda}, \vec{\lambda}^*, L) = V_{C, \vec{p}}(\vec{\lambda}, L) \otimes V(\vec{\lambda}^*)\\
\end{equation}

\end{definition}

\subsection{Correlation}\label{correlation}

The Lie algebra $\mathfrak{g}$ is naturally realized as the subalgebra $\C1\otimes \mathfrak{g} \subset \C[C \setminus q]\otimes \mathfrak{g},$ and the action of $\mathfrak{g}$ on $H(\vec{\lambda}^*, L)$ restricts to the expected action on $V(\vec{\lambda}^*) = V(\lambda_1^*)\otimes \ldots \otimes V(\lambda_n^*)$.  As a result we have the correlation map.

\begin{equation}
\kappa_{\vec{\lambda}, L}: V_{C, \vec{p}}(\vec{\lambda}, L) \to [V(\vec{\lambda})]^{\mathfrak{g}}\\ 
\end{equation}

\noindent
 When the genus of $C$ is $0$, the image of $\kappa_{\vec{\lambda}, L}$ is constructed in \cite{B1} Part I, Section 4, as the set of invariant tensors in $V(\vec{\lambda})$ which are annhilated by the $L+1$st power $e_{\vec{p}}^{L+1}$ of a nilpotent operator, $e_{\vec{p}}$. When $n=3$ the space $\mathcal{M}_{0,3}$ is the single point $(\mathbb{P}^1, 0, 1, \infty),$ in this case the space of conformal blocks $V_{0, 3}(\lambda, \eta, \mu, L)$ can be described explicitly. Each $\mathfrak{g}$ representation $V(\lambda)$ can be decomposed into isotypical components along the action of the copy of $sl_2(\C)$ in $\mathfrak{g}$ corresponding the longest root $\omega.$

\begin{equation}
V(\lambda) = \bigoplus_{i \geq 0} V(\lambda, i)\\ 
\end{equation} 

\noindent
We decompose $V(\lambda) \otimes V(\eta)\otimes V(\mu)$ in each component accordingly, and define the subspace $V_L$ as the sum of the components $V(\lambda, i)\otimes V(\eta, j) \otimes V(\mu, k)$ with $i + j + k \leq 2L.$ 

\begin{proposition}\label{ueno}(\cite{U}, Corollary 3.5.2)
\begin{equation}
V_{0, 3}(\lambda, \eta, \mu, L) = [V(\lambda) \otimes V(\eta)\otimes V(\mu)]^{\mathfrak{g}} \cap V_L\\
\end{equation}
\end{proposition}

We will make use of the map obtained from $\kappa_{\vec{\lambda}, L}$ by tensoring both sides of the above expression with $V(\vec{\lambda}^*),$ which we also call $\kappa_{\vec{\lambda}, L}$ when no confusion results.  The action by $\mathfrak{g}$ is on the left hand side of $V(\vec{\lambda}, \vec{\lambda}^*)$.

\begin{equation}\label{vectorcorrelation}
\kappa_{\vec{\lambda}, L}: W_{C, \vec{p}}(\vec{\lambda}, \vec{\lambda}^*, L) \to [V(\vec{\lambda}, \vec{\lambda}^*)]^{\mathfrak{g}}\\
\end{equation}

\subsection{Factorization}

 For any singularity $q \in C$, one can form the partial normalization $(\tilde{C}, q_1, q_2),$ where $q_1, q_2$ are the two new points introduced by splitting $q.$  Factorization expresses how the spaces of conformal blocks behave under this operation. 

\begin{theorem}[\cite{TUY}]\label{fact}
\begin{equation}
V_{C, \vec{p}}(\vec{\lambda}, L) \cong \bigoplus_{\alpha \in \Delta_L} V_{\tilde{C}, \vec{p}, q_1, q_2}(\vec{\lambda}, \alpha, \alpha^*, L)\\
\end{equation}  
\end{theorem}

We mention now that if a curve is not connected, $(C, \vec{p}) = (C_1, \vec{p_1}) \cup (C_2, \vec{p}_2),$ the space of conformal blocks is a tensor product. 

\begin{equation}
V_{C_1 \cup C_2, \vec{p}_1, \vec{p}_2}(\vec{\lambda}_1, \vec{\lambda_2}, L) =  V_{C_1, \vec{p}_1}(\vec{\lambda}_1, L) \otimes V_{C_2, \vec{p}_2}(\vec{\lambda}_2, L)\\
\end{equation}

\noindent
The map in Theorem \ref{fact} is constructed from the following map,  recall the vector $O_{\alpha} \in V(\alpha, \alpha^*)$.

\begin{equation}
F_{\alpha}: H(\vec{\lambda}^*, L) \to H(\vec{\lambda}^*, \alpha^*, \alpha, L)\\
\end{equation}

\begin{equation}
F_{\alpha}(v) = v\otimes O_{\alpha^*}\\
\end{equation}

\noindent
After dualizing,  invariants are taken by the respective algebras to produce the factorization map.

\subsection{Factorization for the spaces $W_{C, \vec{p}}(\vec{\lambda}, \vec{\lambda}^*, L)$ }\label{wfactor}

In order to boost factorization to the spaces $W_{C, \vec{p}}(\vec{\lambda}, \vec{\lambda}^*, L)$ we use the following map. 

\begin{equation}
F_{\alpha}: H(\vec{\lambda}^*, \vec{\lambda},L) \to H(\vec{\lambda}^*, \vec{\lambda}, L) \otimes V(\alpha^*, \alpha)\otimes V(\alpha, \alpha^*)\\
\end{equation}

\begin{equation}
F_{\alpha}(v) = v\otimes O_{\alpha^*} \otimes O_{\alpha}\\
\end{equation}

\noindent
Here $O_{\alpha^*}\otimes O_{\alpha}$ are tensors on the first and third, and second and fourth indices  of $ V(\alpha^*, \alpha)\otimes V(\alpha, \alpha^*)$. 

\begin{lemma}
For a curve $C$ with singularity $q \in C$, and normalization $(\tilde{C}, q_1, q_2)$,

\begin{equation}
W_{C, \vec{p}}(\vec{\lambda}, \vec{\lambda}^*, L) = \bigoplus_{\alpha \in \Delta_L } W_{\tilde{C}, \vec{p}, q_1, q_2}(\vec{\lambda}, \vec{\lambda}^*, \alpha, \alpha^*, L)^G.\\
\end{equation}

\end{lemma}

\begin{proof}
We use the maps $F_{\alpha}$ and pass to $G$ invariants, where $G$ acts on the second and fourth indices of $H(\vec{\lambda}^*, \vec{\lambda}, L) \otimes V(\alpha^*, \alpha)\otimes V(\alpha, \alpha^*)$,  we can do this because $O_{\alpha}$ is an invariant. We then dualize the map $F_{\alpha}$ to obtain $F_{\alpha}^*: [H(\vec{\lambda}^*, \vec{\lambda}, L) \otimes V(\alpha^*, \alpha)\otimes V(\alpha, \alpha^*)^G]^* \to  H(\vec{\lambda}^*, \vec{\lambda},L)^*$.   Theorem \ref{fact} then implies that this map gives an injection $F_{\alpha}^*: W_{\tilde{C}, \vec{p}, q_1, q_2}(\vec{\lambda}, \vec{\lambda}^*, \alpha, \alpha^*, L)^G =$ $V_{\tilde{C}, \vec{p}, q_1, q_2}(\vec{\lambda}, \alpha, \alpha^* L) \otimes V(\vec{\lambda}^*) \to$ $V_{C, \vec{p}}(\vec{\lambda}, L) \otimes V(\vec{\lambda}^*) =$ $W_{C, \vec{p}}(\vec{\lambda}, \vec{\lambda}^*, L)$  and that the sum of these maps gives an isomorphism. 
\end{proof}

\section{The coordinate ring of $\bar{B}_{C, \vec{p}}(G)$}\label{algebra}

In this section we construct a commutative $\C$ algebra $W_{C, \vec{p}}(G)$ for every curve $(C, \vec{p}) \in \bar{\mathcal{M}}_{g, n}.$
As a vector space, $W_{C, \vec{p}}(G)$ is a direct sum of the spaces of extended conformal blocks $W_{C, \vec{p}}(\vec{\lambda}, \vec{\lambda}^*, L)$ constructed in Section \ref{factorcor}.  The multiplication operation on $W_{C, \vec{p}}(G)$ is graded
with respect to the level $L$, this gives us the necessary information to define the spaces $B_{C, \vec{p}}(G)$ and $\bar{B}_{C, \vec{p}}(G)$ as follows. 

\begin{definition}
\begin{equation}
\bar{B}_{C, \vec{p}}(G) = Spec(W_{C, \vec{p}}(G))\\
\end{equation}

\begin{equation}
B_{C, \vec{p}}(G) = Proj(W_{C, \vec{p}}(G))\\
\end{equation}

\end{definition}

We finish this section by showing that $W_{\mathbb{P}^1, \vec{p}}(G)$ is a Rees algebra of $\C[M_{0, n}(G)]$.  This implies
that $B_{\mathbb{P}^1, \vec{p}}(G)$ has $M_{0, n}(G)$ as a dense, open subspace. 
  
\subsection{The multiplication map}\label{mmap}

 The map $q_{\eta}^{\alpha, \beta}: V(\eta, \eta^*) \to V(\alpha, \alpha^*) \otimes V(\beta, \beta^*)$ from Subsection \ref{Gring} lifts to a map $[\mu_{\eta}^{\alpha, \beta }]^*: H(\eta, \eta^*, L+K) \to H(\alpha, \alpha^*, L)\otimes  H(\beta, \beta^*, K)$ by pairing the unique
$\hat{\mathfrak{g}}$ map $H(0, L+K) \to H(0, L) \otimes H(0, K)$ which sends $v_{0, L+K} \to v_{0, L}\otimes v_{0, K}$ with $q_{\eta}^{\alpha, \beta}.$ This map is generalized as expected to the $n-$marked case. 

\begin{equation}
\mu^*: \bigoplus_{\vec{\eta} \leq \vec{\alpha} + \vec{\beta}} H(\vec{\eta}, \vec{\eta}^*, L +K) \to H(\vec{\alpha}, \vec{\alpha}^*, L) \otimes H(\vec{\beta}, \vec{\beta}^*, K)\\
\end{equation}

\noindent
Here $\leq$ is the dominant weight ordering.  Dualizing and taking invariants by the Lie algebra $\C[C \setminus q]\otimes \mathfrak{g}$ then yields the following multiplication map. 

\begin{equation}
\mu: W_{C, \vec{p}}(\vec{\alpha}, \vec{\alpha}^*, L) \otimes W_{C, \vec{p}}(\vec{\beta}, \vec{\beta}^*, L) \to \bigoplus_{\vec{\eta} \leq \vec{\alpha} + \vec{\beta}} W_{C, \vec{p}}(\vec{\eta}, \vec{\eta}^*, L)\\
\end{equation}

Next we place  the spaces $W_{C, \vec{p}}(\vec{\lambda}, \vec{\lambda}^*, L)$   in the context of the algebraic geometry of principal bundles.  We consider the scheme $Q\times G^n$, where $Q$ is the affine Grassmannian variety for the group $G,$ see \cite{LS}, Section 10.  For the definition of this space and a description of its place in the theory of principal bundles.   The total coordinate ring of $Q\times G^n$ is the vector space $\bigoplus_{\vec{\lambda}, L} H(0, L)^* \otimes V(\vec{\lambda}, \vec{\lambda}^*).$ The multiplication map on this algebra is given by dualizing the maps $q_{\eta}^{\alpha, \beta}$ on each $G$ component of the product, and by dualizing the map of $\hat{\mathfrak{g}}$ representations given by sending the highest weight vector $v_{0,L+ K}$ to $v_{0,L} \otimes v_{0,K}$, see \cite{Ku}, Section 1.  The space of conformal blocks $V_{C, \vec{p}}(\vec{\lambda}, L)$ is a subspace of $H(0, L)^* \otimes V(\vec{\lambda}),$ following \cite{LS}, Section 10.   An identical argument to their proof of $1.2.1$ shows that $W_{C, \vec{p}}(\vec{\lambda}, \vec{\lambda}^*, L)$ is a space of $\C[C \setminus q]\otimes \mathfrak{g}$ invariants in the coordinate ring of $Q \times G^n$.  This shows that the bilinear operation we've constructed above agrees with the multiplication operation in the $\C[C \setminus q]\otimes \mathfrak{g}$ -invariant subring of the total coordinate ring of $Q \times G^n.$  The vector space $W_{C, \vec{p}}(G) = \bigoplus_{\vec{\lambda}, L} W_{C, \vec{p}}(\vec{\lambda}, \vec{\lambda}^*, L)$, along with the multiplication map $\mu$ give our ring.  From this point of view, $W_{C, \vec{p}}(G)$ is the total coordinate ring of the quotient stack $\mathcal{B}_{C, \vec{p}}(G) = L_C(G) \backslash [Q\times G^n]$, where $L_C(G)$ is the ind-group used in \cite{LS}, Section 10.

\subsection{Extending correlation}\label{ecor}

Next, we extend the correlation map  on the spaces $W_{C, \vec{p}}(\vec{\lambda}, \vec{\lambda}^*, L)$ to a map of commutative algebras.   We apply the construction in Subsection \ref{nilval} to the operator $e_{\vec{p}}$  defined by Beauville in \cite{B1} Part I, Section 4.  As defined, $e_{\vec{p}}$ acts on the spaces $V(\vec{\lambda})^{\mathfrak{g}}$,  by tensoring with $V(\vec{\lambda}^*)$, we can define a left hand side action of $e_{\vec{p}}$ on the components of the coordinate ring $\C[M_{0, n}(G)].$  We let $v_{\vec{p}}$ be the valuation defined by the operator $e_{\vec{p}}.$

\begin{proposition}\label{g0rees}
The algebra $W_{\mathbb{P}^1, \vec{p}}(G)$ is the Rees algebra of $\C[M_{0, n}(G)]$ given by the valuation $v_{\vec{p}}.$
\end{proposition}

\begin{proof}
The operator $e_{\vec{p}}$ defined in \cite{B1} is a $\C-$linear derivation as constructed. Furthermore, the space annhilated by $e_{\vec{p}}^{L+1}$ is canonically identified with $W_{C, \vec{p}}(L) = \bigoplus_{\vec{\lambda}} W_{\mathbb{P}^1, \vec{p}}(\vec{\lambda}, \vec{\lambda}^*, L)$.   By definition of $\mu, q,$ and $m$ the following diagram commutes. 

$$
\begin{CD}
H(\vec{\alpha}, \vec{\alpha}^*, L) \otimes H(\vec{\beta}, \vec{\beta}^*, L) @<\mu^*<< \bigoplus_{\vec{\eta} \leq \vec{\alpha} + \vec{\beta}} H(\vec{\eta}, \vec{\eta}^*, L)\\
@AAA @AAA \\
V(\vec{\alpha}^*, \vec{\alpha}) \otimes V(\vec{\beta}^*, \vec{\beta}) @<q << \bigoplus_{\vec{\eta} \leq \vec{\alpha} + \vec{\beta}} V(\vec{\eta}^*, \vec{\eta})\\
\end{CD}
$$

By dualizing and taking invariants we obtain the following commutative diagram. 

$$
\begin{CD}
W_{C, \vec{p}}(\vec{\alpha}, \vec{\alpha}^*, L) \otimes W_{C, \vec{p}}(\vec{\beta}, \vec{\beta}^*, K) @> \mu>> \bigoplus_{\vec{\eta} \leq \vec{\alpha} + \vec{\beta}} W_{C, \vec{p}}(\vec{\eta}, \vec{\eta}^*, L+K)\\
@V \kappa_{\vec{\alpha}, L} \otimes \kappa_{\vec{\beta}, K}VV @V \kappa_{\vec{\eta}, L+K} VV\\
V(\vec{\alpha}^*, \vec{\alpha}) \otimes V(\vec{\beta}^*, \vec{\beta}) @> m >> \bigoplus_{\vec{\eta} \leq \vec{\alpha} + \vec{\beta}} V(\vec{\eta}^*, \vec{\eta})\\
\end{CD}
$$

This shows that the multiplication operation induced on these spaces from $\C[M_{0, n}(G)]$ agrees with multiplication in $W_{\mathbb{P}^1, \vec{p}}(G)$.   Note that the diagram commutes for arbitrary curves. When $C =\mathbb{P}^1$, this defines a $1-1$ map of algebras, $\kappa_{\mathbb{P}^1, \vec{p}}: W_{\mathbb{P}^1, \vec{p}}(G) \to \C[M_{0, n}(G)]\otimes \C[t].$
\end{proof}

\section{Factorization and the fiber over a stable curve}\label{stablecurvesection}

We prove Theorem \ref{ringfact}, and extend Proposition \ref{g0rees} to show that $W_{C, \vec{p}}(G)$ and $V_{C, \vec{p}}(G)$ are Rees algebras of $\C[M_{g, n}(G)]$ and $\C[P_{g, n}(G)]$ respectively, when $(C, \vec{p})$ is a stable union of projective lines.

\subsection{Factorization for the algebra $W_{C, \vec{p}}(G)$}

The purpose of this subsection is to prove the following proposition. 

\begin{proposition}\label{ringfact}
Let $(\tilde{C}, \vec{p}, q_1, q_2)$ be the partial normalization of a stable curve $(C, \vec{p})$ at a double point $q \in C$, then the following equation holds. 

\begin{equation}
W_{C, \vec{p}}(G) = W_{\tilde{C}, \vec{p}, q_1, q_2}(G)^{G\times \C^*}\\
\end{equation}

\noindent
Here the action of $G$ on $W_{\tilde{C}, \vec{p}, q_1, q_2}(G)$ is on the right hand side of the factors associated to $q_1$ and $q_2,$ and the action of $\C^*$ has character equal to the difference of the levels on on the algebras corresponding to the connected components of $\tilde{C}$.
 
\end{proposition}

\begin{proof}  
In this proof we will use the maps $\mu^*$ for several different spaces at the same time, we will differentiate between these maps with a subscript $\mu^*_n.$ First we consider the following map. 

\small
$$
\begin{CD}
 \bigoplus_{\vec{\gamma} \leq \vec{\lambda} + \vec{\mu}, \eta \leq \alpha + \beta} [H(\vec{\gamma}, \vec{\gamma}^*, \eta, \eta^*, \eta^*, \eta, L+K)] @>\mu_{n+2}^*>>  [H(\vec{\lambda}, \vec{\lambda}^*, \alpha, \alpha^*, \alpha^*, \alpha, L)]\otimes [H(\vec{\mu}, \vec{\mu}^*, \beta, \beta^*, \beta^*, \beta, K)]\\
\end{CD}
$$
\normalsize

\noindent
Taking dual spaces and invariants by  $\C[\tilde{C} \setminus q]\otimes \mathfrak{g}$ and the action by $G$ given in Subsection \ref{wfactor} gives the multiplication operaton in the ring $W_{\tilde{C}, \vec{p}, q_1, q_2}(G)^{G \times \C^*}.$  Here the fact that we are also taking $\C^*$ invariants is hidden in our insistence that the levels $L, K, L+K$ be the same across different connected components of $\tilde{C}.$  

However, by Lemma \ref{multcommute}, in passing to the $G$ invariant spaces  $ [H(\vec{\gamma}, \vec{\gamma}^*, \eta, \eta^*, \eta^*, \eta, L+K)]^G$ $\cong [H(\vec{\gamma}, \vec{\gamma}^*, \eta, \eta^*, L+K)]$ the map induced by $\mu_{n+2}^*$ agrees with $\mu_{n+1}^*.$ Furthermore,  $ [H(\vec{\gamma}, \vec{\gamma}^*, \eta, \eta^*, \eta^*, \eta, L+K)]^G$ is the subspace of those vectors with $O_{\eta}$ in the second and fourth indices.  Therefore, the following diagram commutes. 

\small
$$
\begin{CD}
[H(\vec{\lambda}, \vec{\lambda}^*, L) \otimes H(\vec{\mu}, \vec{\mu}^*, K)] @>F_{\alpha} \otimes F_{\beta}>> [H(\vec{\lambda}, \vec{\lambda}^*, \alpha, \alpha^*, \alpha^*, \alpha, L) ]^G \otimes [H(\vec{\mu}, \vec{\mu}^*, \beta, \beta^*, \beta^*, \beta, K)]^G\\
@A \mu_n^* AA @A\mu_{n+2}^*AA\\
\bigoplus_{\vec{\gamma} \leq \vec{\lambda} + \vec{\mu}}[ H(\vec{\gamma}, \vec{\gamma}^*, L+K)] @>\sum F_{\eta}>> \bigoplus_{\vec{\gamma} \leq \vec{\lambda} + \vec{\mu}, \eta \leq \alpha + \beta} [H(\vec{\gamma}, \vec{\gamma}^*, \eta, \eta^*, \eta^*, \eta, L+K)]^G\\
\end{CD}
$$
\normalsize

  Dualizing this diagram, and taking invariants by $\C[C \setminus q]\otimes \mathfrak{g}$ on the left and $\C[\tilde{C} \setminus q]\otimes \mathfrak{g}$ on the right then produces the following commutative diagram.  

\small
$$
\begin{CD}
W_{C, \vec{p}}(\vec{\lambda}, \vec{\lambda}^*, L) \otimes W_{C, \vec{p}}(\vec{\mu}, \vec{\mu}^*, K) @<<< W_{\tilde{C}, \vec{p}, q_1, q_2}(\vec{\lambda}, \vec{\lambda}^*, \alpha, \alpha^*, L)^G \otimes  W_{\tilde{C}, \vec{p}, q_1, q_2}(\vec{\mu}, \vec{\mu}^*, \beta, \beta^*, K)^G\\
@V\mu VV @V \mu VV\\
\bigoplus_{\vec{\gamma} \leq \vec{\lambda} + \vec{\mu}}W_{C, \vec{p}}(\vec{\gamma}, \vec{\gamma}^*, L+K) @<<<  \bigoplus_{\vec{\gamma} \leq \vec{\lambda} + \vec{\mu}, \eta \leq \alpha + \beta} W_{\tilde{C}, \vec{p}, q_1, q_2}(\vec{\gamma}, \vec{\gamma}^*, \eta, \eta^*, L+K)^G\\
\end{CD}
$$
\normalsize

\noindent
The vertical arrows are multiplication in $W_{C, \vec{p}}(G)$ and $W_{\tilde{C}, \vec{p}, q_1, q_2}(G)^{G \times \C^*}$, respectively. This shows that the factorization map defined in Subsection \ref{wfactor} gives a map of algebras when we sum over all possible dominant weight and level assignments.  
\end{proof}

\subsection{Relationship with the algebra $V_{C, \vec{p}}(G)$}

Now we explain the relationship between $W_{C, \vec{p}}(G)$ and the algebra of conformal blocks $V_{C, \vec{p}}(G)$, studied in \cite{M4}. By working through definitions, we have $W_C(G) = V_C(G)$ in the $n = 0$ case, the following addresses $n > 0.$

\begin{proposition}
The algebra $V_{C, \vec{p}}(G)$ is the $U^n$ invariant subring of $W_{C, \vec{p}}(G).$
Equivalently, the space $\bar{K}_{C, \vec{p}}(G)$ is obtained from $\bar{B}_{C, \vec{p}}(G)$ by a right hand side $U^n \subset G^n$ quotient. 
\end{proposition}

\begin{proof}
We can identify $W_{C, \vec{p}}(G)$ with the $\C[C \setminus q]\otimes \mathfrak{g}$ invariants in the projective coordinate ring of $Q \times G^n.$ Since the $\C[C \setminus q]\otimes \mathfrak{g}$ action commutes with the $U^n \subset G^n$ action on the right, we may take the $U^n$ invariants first.  This yields the coordinate ring of $Q \times (G/U)^n$, which is known to contain $V_{C, \vec{p}}(G)$ as its $\C[C \setminus q]\otimes \mathfrak{g}$ invariant subring, \cite{M4}.
\end{proof}

The pair $\mathcal{O}(\lambda), \mathcal{L}(\lambda)$ denotes the flag variety of parabolic weight $\lambda,$ with its canonical line bundle. We consider the product $\bar{B}_{C, \vec{p}}(G) \times \mathcal{O}(\vec{\lambda})$, with the linearization $\mathcal{\vec{\lambda}} = \mathcal{L}(\lambda_1) \otimes \ldots \otimes \mathcal{L}(\lambda_n)$. This space has an action of $G^n$, where the $i-th$ component acts on the $i-$th index of $\bar{B}_{C, \vec{p}}(G)$ and $\mathcal{O}(\vec{\lambda})$, this is linearized by the trivial line bundle with the $\mathcal{L}(\lambda_i)$. The projective coordinate ring of the GIT quotient of this scheme with respect to $G^n$ is the following invariant subring. 

\begin{equation}
[W_{C, \vec{p}}(G) \otimes \bigoplus_{m \geq 0} V(m\lambda)]^{G^n} = \bigoplus_{m \geq} V_{C, \vec{p}}(m\vec{\lambda}, mL) = R_{C, \vec{p}}(\vec{\lambda}, L)\\
\end{equation}

The $GIT$ quotient $[G/U]/_{\lambda} T$ is equal to $\mathcal{O}(\lambda)$, therefore the above ring can also be identified with the $T^n$-invariants in $V_{C, \vec{p}}(G)$ with respect to the character defined by $\vec{\lambda}.$  The algebra $V_{C, \vec{p}}(G)$ is the total coordinate ring of the moduli stack $\mathcal{M}_{C, \vec{p}}(G)$ of quasi-parabolic principal bundles, and $ R_{C, \vec{p}}(\vec{\lambda}, L)$ is the projective coordinate ring of the line bundle on this stack associated to $(\vec{\lambda}, L).$  The $Proj$ of the corresponding ring is then the coarse moduli $M_{C, \vec{p}}(\vec{\lambda}, L)$.

\subsection{The space $M_{g, n}(G)$ as a subspace of $B_{C_{\Gamma}, \vec{p}_{\Gamma}}(G)$}

Let  $W_{C, \vec{p}}(L)$ be the sum $\bigoplus_{\vec{\lambda} \in \Delta_L^n} W_{C, \vec{p}}(\vec{\lambda}, \vec{\lambda}^*, L)$,
this is the space of all extended conformal blocks of level $\leq L.$ 

\begin{proposition}\label{gcorrelate}
If $C$ is a union of projective lines, then $W_{C, \vec{p}}(G)$ is a Rees algebra of $\C[M_{g, n}(G)].$ 
\end{proposition}

\begin{proof}
We let $\Gamma$ be the dual graph to the arrangment of projective lines defined by $C$. We get the following from Theorem \ref{ringfact}.

\begin{equation}
W_{C, \vec{p}}(G) = [\bigotimes_{v \in V(\Gamma)} W_{\mathbb{P}^1, \vec{p}_i}(G)]^{[G\times \C^*]^{E(\Gamma)}}\\
\end{equation}
\noindent
 The algebra $W_{\mathbb{P}^1, \vec{p}_i}(G)$ for each of these components comes with the correlation morphism from Subsection \ref{ecor}. 

\begin{equation}
\kappa_v: W_{\mathbb{P}^1, \vec{p}_v}(G) \to \C[M_{0, n(v)}(G)]\otimes \C[t_v]\\
\end{equation}

We may define a $\C^*$ action associated to an edge $e \in E(\Gamma)$ on $\bigotimes_{v \in V(\Gamma)} \C[M_{0, n(v)}(G)]\otimes \C[t_v]$ 
by declaring that its character on a graded component to be the difference between the $t_v, t_w$ powers, where $v, w$ are the endpoints
of $e.$ By definition, the map $\bigotimes_{v \in V(\Gamma)} \kappa_v$ intertwines this action with the $(\C^*)^{E(\Gamma)}$ action 
on $\bigotimes_{v \in V(\Gamma)} W_{\mathbb{P}^1, \vec{p}_v}(G)$.  The only monomials in the $t_v$ invariant under this action are powers of $t = \prod_{v \in V(\Gamma)} t_v$.  Likewise, the gluing actions of $G$ on $\bigotimes_{v \in V(\Gamma)}\C[M_{0, n(v)}]$ associated to each edge $e \in E(\Gamma)$ intertwine with the corresponding actions on $\bigotimes_{v \in V(\Gamma)} W_{\mathbb{P}^1, \vec{p}_v}(G)$.  The result is a $1-1$ correlation morphism on the $(G\times \C^*)^{E(\Gamma)}$ invariants. 

\begin{equation}
\kappa_{\Gamma}: W_{C, \vec{p}}(G) \to \C[M_{g, n}(G)]\otimes \C[t],\\
\end{equation}  

\noindent
By Proposition \ref{g0rees}, the image of $W_{C, \vec{p}}(L)$ under $\kappa_{\Gamma}$ is the set of $f \in \C[M_{g, n}(G)]$ such that each $v_{\vec{p}_w}(f) \leq L$ for $\vec{p}_w$ the arrangement of points on the copy of $\mathbb{P}^1$ dual to $w \in V(\Gamma).$ 
\end{proof}

By taking $U^n$ invariants of $W_{C, \vec{p}}(G)$ and $\C[M_{g, n}(G)]$, we obtain an identical description of $V_{C, \vec{p}}(G)$ as a Rees algebra of $\C[P_{g, n}(G)]$, this proves Theorem \ref{maingen1}.  In the case that $C$ is itself a genus $0$ curve, all of the spaces $K_{C, \vec{p}}(G)$ have the universal configuration space $P_{0, n}(G)$ as a dense open subscheme.  The ring theoretic version of this latter fact was observed in \cite{M4}.

\section{The scheme $M_{g, n}(SL_2(\C))$}\label{sl2theory}

The remainder of the paper is devoted to $G = SL_2(\C)$, where we can give a more complete account of the spaces $B_{C, \vec{p}}(SL_2(\C))$  and
$K_{C, \vec{p}}(SL_2(\C)).$  Elements of the representation theory of $SL_2(\C)$ provide combinatorial tools for describing these spaces concretely, in particular the Clebsch-Gordon rules for tensor product decomposition, and the Pl\"ucker equations play an important role. 

 We describe two spanning sets of $\C[M_{g, n}(SL_2(\C))]$ attached to a trivalent graph $\Gamma,$ the set of spin diagram elements $\mathcal{R}(\Gamma)$, and the set $\mathcal{S}(\Gamma)$ of $\Gamma-$tensors. The set $\mathcal{S}(\Gamma)$ is used to show that the algebras $W_{C_{\Gamma}, \vec{p}_{\Gamma}}(SL_2(\C))$, $V_{C_{\Gamma}, \vec{p}_{\Gamma}}(SL_2(\C))$ are presented by skein relations, and $\mathcal{R}(\Gamma)$ is used to build a flat degeneration of $W_{C_{\Gamma}, \vec{p}_{\Gamma}}(SL_2(\C))$ to an affine semigroup algebra $\C[H_{\Gamma}^*]$ associated to a convex polyhedral cone $\mathcal{H}_{\Gamma}^*$.  In Section \ref{stratificationsection} we show that the structure of the boundary divisor $D_{\Gamma}$ of the compactification of $M_{g, n}(SL_2(\C))$ defined by $B_{C_{\Gamma}, \vec{p}_{\Gamma}}(SL_2(\C))$ is determined by the face structure of $\mathcal{H}_{\Gamma}^*.$

\subsection{The coordinate ring $\C[SL_2(\C)].$}\label{sl2coordinates}

As an algebraic variety, the group $SL_2(\C)$ is the locus of the equation $AD- BC = 1$ in the space of $2 \times 2$ matrices $M_{2\times 2}(\C).$ 

\begin{equation} SL_2(\C) = \{A, B, C, D | \ det\left[ \begin{array}{cc}
A & B\\
C & D\\
\end{array} \right] = 1 \} \subset M_{2\times 2}(\C)\\
\end{equation}

This description is connected with the $SL_2(\C) \times SL_2(\C)$ isotypical decomposition given by the Peter-Weyl theorem ( Subsection \ref{Gring})
by identifying the generators $A, B, C, D$ above with matrix elements in $End(V(1)) = M_{2\times 2}(\C)^*.$   Let $X_{ij} \in M_{2\times 2}(\C)$ be the matrix which has $ij$ entry equal to $1$ and all other entries $0$, the generators of $\C[SL_2(\C)]$ are computed on an element $M \in SL_2(\C)$ as follows. 

\begin{equation}
A(M) = Tr(M^{-1}X_{11}) \ \ \ \ B(M) = Tr(M^{-1}X_{01}) \ \ \ \ C(M) = Tr(M^{-1}X_{10}) \ \ \ \ Tr(M^{-1}X_{00})\\
\end{equation}

Let $U_-, U_+ \subset SL_2(\C)$ be the groups of upper, respectively lower triangular matrices in $SL_2(\C).$ It will be necessary to use the coordinate ring of the $GIT$ quotients $SL_2(\C)/U_-, U_+ \backslash SL_2(\C)$, both of which are identified with $\C^2.$  The algebras $\C[SL_2(\C)]^{U_-}, \C[SL_2(\C)]^{U_+}$ are the subspaces of right highest, respectively left lowest weight vectors in the representation $\bigoplus_{i \in \Z_{\geq 0}} V(i, i)$ 

\begin{equation}
\C[SL_2(\C)]^{U_-} = \C[SL_2(\C)]^{U_+} = \bigoplus_{i \in \Z_{\geq 0}} V(i)\\
\end{equation}
\noindent

With actions taken on the right hand side, the algebra $\C[SL_2(\C)]^{U_-} \subset \C[SL_2(\C)]$ is a polynomial ring, generated by $A, C$, whereas $\C[SL_2(\C)]^{U_+} \subset \C[SL_2(\C)]$ is generated by $B, D$. Under the automorphism $(-)^{-1}:SL_2(\C) \to SL_2(\C)$,
$\C[U_-\backslash SL_2(\C)]$ is likewise identified with a polynomial ring in two variables.

The isotypical decomposition of $\C[SL_2(\C)]$ suggests a natural $SL_2(\C) \times SL_2(\C)$-stable filtration defined by the spaces $F_{\leq k} = \bigoplus_{i \leq k} V(i, i) \subset \C[SL_2(\C)]$.  Following Subsection \ref{Gring} (see also \cite{Gr}, Chapter 7), this is indeed an algebra filtration, in particular one can verify that the image of a product $End(V(i))End(V(j))$ lies in the space $\bigoplus_{k \leq i + j} End(V(k)).$   The associated graded algebra of this filtration is the algebra $[\C[SL_2(\C)]^{U_-} \otimes \C[SL_2(\C)]^{U_+}]^{\C^*}$.  Here the the action by $\C^*$ picks out the invariant subalgebra one would expect from the isotypical decomposition, $\bigoplus_{i \in \Z_{\geq 0}} V(i) \otimes V(i) \subset \C[SL_2(\C)]^{U_-} \otimes \C[SL_2(\C)]^{U_+}$.  Hidden in this statement is the important fact that for any non-zero $f \in End(V(i)), g \in End(V(j))$, the component $(fg)_{i + j} \in End(V(i+j))$ is always non-zero.  This is a consequence of the fact that $[\C[SL_2(\C)]^{U_-} \otimes \C[SL_2(\C)]^{U_+}]^{\C^*}$ is a domain.   We let $v: \C[SL_2(\C)] \to \Z \cup \{-\infty\}$ be the valuation associated to this filtration.

From now on we let $SL_2(\C)^c$ be the $GIT$ quotient  $[SL_2(\C)/U_- \times U_+ \backslash SL_2(\C)]/\C^*.$  Multiplication in $\C[SL_2(\C)^c]$ is $SL_2(\C) \times SL_2(\C)$-equivariant, and the isotypical components are irreducible representations of this group, it follows that each product map $[V(i) \otimes V(i)] \otimes [V(j) \otimes V(j)] \to [V(i + j) \otimes V(i + j)]$ is surjective, and this algebra is generated by the subspace $V(1) \otimes V(1)$.  The associated graded algebra can be presented by $\C[A, B, C, D]$ modulo the initial ideal of $<AD - BC - 1>$ with respect to the filtration $F.$ Let $X, Y$ be generating monomials of the coordinate ring $\C[\C^2]$.  With this notation, the algebra $\C[SL_2(\C)^c]$ is isomorphic to $\C[\C^2 \times \C^2/\C^*]$ by the map $A \to X \otimes X, B \to X\otimes Y, C \to Y \otimes X, D \to Y \otimes Y.$

\begin{equation}
\C[SL_2(\C)^c] = \C[A, B, C, D]/<AD-BC>\\
\end{equation}

Let $\mathcal{P}$ be the pointed, polyhedral cone in $\R^3$ defined by the origin $(0, 0, 0)$ and
the rays through the points $\{(0, 0, 1), (1, 0, 1), (0, 1, 1), (1, 1, 1)\}$, and let $P$ be the affine
semigroup  $\mathcal{P} \cap \Z^3.$ We have the following isomorphism of $\C$ algebras. 

\begin{equation}
\C[SL_2(\C)^c] \cong \C[P]\\
\end{equation}

\subsection{The spin diagram basis $\mathcal{R}(\Gamma)$}

Observe that the coordinate ring of the space $SL_2(\C)^{E(\Gamma)}$ has the following isotypical decomposition under the action of $SL_2(\C)^{E(\Gamma)} \times SL_2(\C)^{E(\Gamma)}$. 

\begin{equation}
\C[SL_2(\C)^{E(\Gamma)}] = \bigoplus_{a: E(\Gamma) \to \Z_{\geq 0}} \bigotimes_{e \in E(\Gamma)} V(a(e))\otimes V(a(e))\\
\end{equation}

\noindent
By passing to $SL_2(\C)^{V(\Gamma)}$ invariants we obtain a direct sum decomposition of the coordinate ring of $M_{\Gamma}(SL_2(\C))$ using its alternative description from Subsection \ref{alternative}. 

\begin{equation}
\C[M_{\Gamma}(SL_2(\C))] =  \bigoplus_{a: E(\Gamma) \to \Z_{\geq 0}} [\bigotimes_{e \in E(\Gamma)} V(a(e))\otimes V(a(e))]^{SL_2(\C)^{V(\Gamma)}}\\
\end{equation}

\begin{definition}
Let $v_e: \C[M_{\Gamma}(SL_2(\C))] \to \Z \cup \{-\infty\}$ be the valuation on the coordinate ring of $M_{\Gamma}(SL_2(\C))$
defined by the valuation $v$ on the $e-$copy of $SL_2(\C)$ in $SL_2(\C)^{E(\Gamma)}$ and the inclusion $\C[M_{\Gamma}(SL_2(\C))] \subset \C[SL_2(\C)^{E(\Gamma)}]$.
\end{definition}

\begin{definition}
Let $\mathcal{C}_{\Gamma} = \R_{\geq 0}^{E(\Gamma)}$.  For a point $x \in \mathcal{C}_{\Gamma}$, 
let $v_x$ be the filtration $\oplus_{e \in E(\Gamma)} x(e)v_e$ on $\C[M_{\Gamma}(SL_2(\C))]$.
\end{definition}

\begin{lemma}\label{mlem}
For every point $x \in \mathcal{C}_{\Gamma}$, $v_x$ is a valuation on the coordinate ring $\C[M_{\Gamma}(SL_2(\C))]$. 
Furthermore, for $v_x$ with $x(e) \neq 0$ for all $e \in E(\Gamma)$, the associated graded algebra of $\C[M_{\Gamma}(SL_2(\C))]$ is
the coordinate ring of the scheme $M_{\Gamma}(SL_2(\C)^c).$ 
\end{lemma}

\begin{proof}
This is a direct consequence of Lemmas \ref{val1} and \ref{val2} and our computations in Subsection \ref{sl2coordinates}. 
\end{proof}

\begin{corollary}\label{pcor}
The associated graded algebra of $\C[P_{\Gamma}(SL_2(\C))]$ is the coordinate ring of the scheme $\C[M_{\Gamma}(SL_2(\C)^c)/U_-^n]$.
\end{corollary}

\begin{proof}
Since the action of the unipotent group $U_-^n$ extends to an action of the reductive group $SL_2(\C)^n$, Lemma \ref{val1}
can be applied.
\end{proof}

In the construction of $M_{\Gamma}(SL_2(\C))$ given in Subsection \ref{alternative}, there is a copy of $SL_2(\C)$ for each vertex $v \in V(\Gamma)$ acting on the tensor product $V(a(e)) \otimes V(a(f)) \otimes V(a(g))$, where $v \in \delta(e), \delta(f), \delta(g).$  The Clebsch-Gordon rule implies that the invariant subspace of such a tensor product is at most one dimensional, and this space is nontrivial if and only if the numbers $a(e), a(f), a(g)$ satisfy two conditions.

\begin{enumerate}
\item $a(e) + a(f) + a(g) \in 2\Z$\\
\item $a(e), a(f), a(g)$ are the sides of a triangle: $|a(e) - a(g)| \leq a(f) \leq a(e) + a(g)$.\\
\end{enumerate}

\noindent
We let $\mathcal{U}_{\Gamma}$ be the real cone of $a: E(\Gamma) \to \R_{\geq 0}$ which satisfy condition $2$ above, $L_{\Gamma} \subset \R^{E(\Gamma)}$ be the lattice of integers which satisfy condition $1,$ and $U_{\Gamma} = \mathcal{U}_{\Gamma} \cap L_{\Gamma}$ be the associated affine semigroup.  As $[V(a(e)) \otimes V(a(f)) \otimes V(a(g))]^{SL_2(\C)}$ is multiplicity-free, each space $ [\bigotimes_{e \in E(\Gamma)} V(a(e))\otimes V(a(e))]^{SL_2(\C)^{V(\Gamma)}}$ is isomorphic to $\otimes_{\ell \in L(\Gamma)} V(a(\ell))$ or the $0$ vector space. In particular, this space is non-zero  precisely when $a \in U_{\Gamma}$.

Following Subsection \ref{sl2coordinates} the representation $V(a(\ell))$ is the subspace of the degree $a(\ell)$ homogeneous polynomials
in the polynomial ring with two variables, in particular it has a basis of monomials $x^iy^j$, $ i + j = a(\ell)$.   These monomials are weight
vectors for the maximal diagonal torus $\C^* \subset SL_2(\C)$, and $x^{a(\ell)}$ is the highest weight vector.  By selecting this highest
weight vector at each leaf, we pick out the $1-$dimensional subspace $[[\bigotimes_{e \in E(\Gamma)} V(a(e))\otimes V(a(e))]^{SL_2(\C)^{V(\Gamma)}}]^{U_-^n} \subset \C[M_{\Gamma}(SL_2(\C))]$ of unipotent invariant vectors. We fix once and for 
all an element $\Phi_a$ in this space, this gives the following direct sum decomposition.

\begin{equation}
\C[P_{\Gamma}(SL_2(\C))] = \C[M_{\Gamma}(SL_2(\C))]^{U_-^n} = \bigoplus_{a \in P_{\Gamma}} \C\Phi_a\\
\end{equation}

\noindent
We let $\mathcal{R}(\Gamma) \subset \C[P_{\Gamma}(SL_2(\C))]$ be the collection of these elements, called spin diagrams from now on.
The algebraic properties of these diagrams are explored extensively in \cite{LP}.  

\begin{proposition}\label{ptor}
The associated graded algebra $\C[M_{\Gamma}(SL_2(\C)^c)/U_-^n]$ from Corollary 
\ref{pcor} is the affine semigroup algebra associated to $U_{\Gamma}$.  
\end{proposition}

\begin{proof}
The follows from Lemmas \ref{val1}, \ref{val2}, and the fact that the graded components of the associated
graded algebra are multiplicity free and labelled by $a \in U_{\Gamma}.$ 
\end{proof}

Since $SL_2(\C)^c$ is itself the $GIT$ quotient $[\SL_2(\C)/U_- \times U_+ \backslash SL_2(\C)]/\C^*$, we can use Proposition
\ref{ptor} to obtain a similar result for $M_{\Gamma}(SL_2(\C)^c),$ the degeneration of $M_{\Gamma}(SL_2(\C)).$  With this in mind
we define $\mathcal{H}_{\Gamma} \subset \R^{E(\Gamma) + 2L(\Gamma)}$ to be the cone defined by the triangle inequalities, 
with the addition of two entries $x_i, y_i$ at each leaf $\ell_i$, such that $a(x_i) + a(y_i) = a(\ell_i)$ for each $a \in \mathcal{H}_{\Gamma}.$ 
We define the lattice $L_{\Gamma} \subset \R^{E(\Gamma) + 2L(\Gamma)}$ in the same way as above, and $H_{\Gamma} = \mathcal{H}_{\Gamma} \cap L_{\Gamma}.$ 

\begin{proposition}\label{mtor}
The associate graded algebra $\C[M_{\Gamma}(SL_2(\C)^c)]$ from Lemma \ref{mlem} above is the affine semigroup 
algebra associated to $H_{\Gamma}.$ 
\end{proposition}

\begin{proof}
It follows from the definition of $M_{\Gamma}(SL_2(\C)^c)$ that this scheme is isomorphic to the $GIT$ quotient 
$[M_{\Gamma}(SL_2(\C)^c)/U_-^n \times [U_+\backslash SL_2(\C)]^n]/(\C^*)^n$.  The coordinate ring
of $M_{\Gamma}(SL_2(\C)^c)/U_-^n \times [U_+\backslash SL_2(\C)]^n$ is the $2n$ polynomial ring over
the affine semigroup algebra $\C[U_{\Gamma}]$, with two variables $X_i, Y_i$ for each leaf $\ell_i$.  
The action of $(\C^*)^n$ decomposes this algebra into isotypical spaces spanned by $\Phi_a \otimes \prod X_i^{s_i}Y_i^{t_i}$.
One of these vectors is invariant if and only if $s_i + t_i - a(\ell_i) = 0.$ 
\end{proof}

In the special case $n = 0$, $\C[M_{\Gamma}(SL_2(\C))] = \C[P_{\Gamma}(SL_2(\C))]$ and $U_{\Gamma} = H_{\Gamma}.$ From
now on we let $\Phi_a$ denote the basis member of $\C[M_{\Gamma}(SL_2(\C)^c)]$ associated to $a \in H_{\Gamma}.$  For $a \in U_{\Gamma}$ or $H_{\Gamma}$ and $x \in \mathcal{C}_{\Gamma}$, the following holds by definition for both $\C[M_{\Gamma}(SL_2(\C))]$ and $\C[P_{\Gamma}(SL_2(\C))]$. 

\begin{equation}
v_x(\Phi_a) = \sum_{e \in E(\Gamma)} x(e)a(e)\\
\end{equation}

We place a partial order on $a \in U_{\Gamma}$ or $H_{\Gamma}$, where $b \prec a$ when $a(e) - b(e) \geq 0$
for any edge $e \in E(\Gamma).$   We call this the $\Gamma-$spin diagram filtration, this should be distinguished from the $\Gamma-$level filtration defined in the introduction. 

\begin{proposition}\label{lowtri}
For $\Phi_a, \Phi_b$ both members of $\C[M_{\Gamma}(SL_2(\C))]$ or $\C[P_{\Gamma}(SL_2(\C))]$, $\Phi_a\Phi_b$ is a sum
$\sum_{t \prec a + b} C_t\Phi_t$.  Furthermore, the coefficient $C_{a + b}$ is always non-zero in this sum. 
\end{proposition}

\begin{proof}
The first property holds because it also holds in the coordinate ring $\C[SL_2(\C)^{E(\Gamma)}]$.  For the
second property, note that $v_x$ would not be a valuation for generic $x$ if this were not the case.
\end{proof}

\subsection{The space $M_{0, 3}(SL_2(\C))$}

The cone $\mathcal{H}_{0, 3}$ is the set of labellings of the diagram in Figure \ref{P3el} by non-negative real numbers which satisfy $x_1 + y_1 = a, x_2 + y_2 = b, x_3 + y_3 = c$, such that $a, b, c$ form the sides of a triangle, $|a-c| \leq b \leq a+c.$ The lattice $L_3$ is the set of integer labellings of this diagram with  $a + b +c \in 2\Z$.  

\begin{figure}[htbp]
\centering
\includegraphics[scale = 0.5]{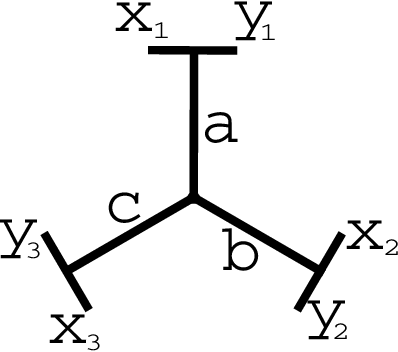}
\caption{Defining diagram of $\mathcal{H}_{0, 3}$.}
\label{P3el}
\end{figure}

The affine semigroup $H_{0,3}$ is generated by the $3 \times 4 = 12$ weightings with one of $a, b, c$ equal to $0$
and the other two entries equal to $1.$ We label these generators $X_{(i, a), (j, b)}$, indicating a path from $i$ to $j$
with orientation markings on the $a, b$ ends of these paths, respectively.  The algebra $\C[M_{0, 3}(SL_2(\C)^c)]$ is the
$(\C^*)^3$ invariant subalgebra of $\C[\C^2 \times \C^2 \times \C^2]^{SL_2(\C)}\otimes (\C^2)^3$ following Proposition \ref{mtor}. 
The element $X_{(i, a)(j, b)}$ is the Pl\"ucker generator $p_{ij} \in \C[\C^2 \times \C^2 \times \C^2]^{SL_2(\C)}$ tensored
with the appropriate monomial $X_iX_j,$ $X_iY_j,$ $Y_iX_j,$ $Y_iY_j$ as indicated by the data $a, b$.  From now on we refer to
the latter as orientation data, an $X$ is considered ''[UP]" and a $Y$ is considered ''[DOWN]".

In order to describe $\C[M_{0, 3}(SL_2(\C))],$ we view $SL_2(\C)^3$ as the space of $2\times 6$ matrices satisfying three determinant equations. 

$$
\left[ \begin{array}{cc|cc|cc}
a_1 & b_1 & a_2 & b_2 & a_3 & b_3\\
c_1 & d_1 & c_2 & d_2 & c_3 & d_3\\
\end{array} \right], \ \  det\left[ \begin{array}{cc}
a_i & b_i\\
c_i & d_i\\
\end{array} \right] = 1.
$$

\noindent
Each of the $2\times 2$ determinants of the matrix above is invariant with respect to the left $SL_2(\C)$ action, as are
the Pl\"ucker relations which hold between them. It follows that we may view $M_{0, 3}(SL_2(\C))$ as the space of $2 \times 6$ matrices $[C_{1, 1}C_{1, 2}C_{2,1}, C_{2,2}, C_{3,1}C_{3,2}]$ which satisfy $det(C_{i,1}C_{i,2}) = 1,$ modulo the left action of $SL_2(\C).$ We let $p_{(i,a),(j,b)}$ be the function $det(C_{i,a}C_{j,b})$.  Note that the orientation data $[UP]$ refers to the first column of a matrix, whereas $[DOWN]$ refers to the second column. For example, $p_{(1, UP)(2, UP)}$ means ''take the determinant of the first columns of the first and second matrices."

\begin{proposition}
The initial terms $in(p_{(i,a),(j, b)})$ give a generating set of $H_{0,3}.$
\end{proposition}

\begin{proof}
Recall that the algebra $\C[SL_2(\C)^c]$ is isomorphic to $\C[\C^2 \times \C^2]^{\C^*}$, and the isomorphism does the following to generators:
$A \to X \otimes X$, $B \to X \otimes Y$, $C \to Y \otimes X$, $D \to Y \otimes Y.$  We consider the image of the tensor $p_{(1, UP), (2, UP)} \in \C[M_{0, 3}(SL_2(\C)^c],$  given by the determinant $A_1C_2 - A_2C_1 =$ $(X_1 \otimes X_1)(Y_2 \otimes X_2) -$ $(X_2 \otimes X_2)(Y_1 \otimes X_1)=$ $(X_1Y_2 - X_2Y_1) \otimes X_1X_2.$  We obtain the Pl\"ucker invariant $X_1Y_2 - X_2Y_1 \in \C[\C^2 \times \C^2 \times \C^2]^{SL_2(\C)}$ $=\C[M_{0, 3}(SL_2(\C)/U_-^3]$ tensored with the leaf data $X_1X_2$, this is precisely the element $X_{(1, UP), (2, UP)} \in H_{0, 3}$.  This computation works for all generators $X_{(i, a)(j, b)}.$  
\end{proof}

\noindent
Relations among the generators of $\C[M_{0, 3}(SL_2(\C))]$  can be represented graphically as below.  

\begin{equation}
\Pccn \Pnkk - \Pckn \Pnck + \Pcnk = 0\\
\end{equation}

\begin{equation}
1 - \Pccn \Pkkn + \Pckn\Pkcn = 0\\
\end{equation}

\bigskip

The coordinate ring of $M_{0, n}(SL_2(\C))$ is likewise generated by Pl\"ucker elements $p_{(i, a), (j, b)}$, subject to Pl"ucker relations.  To describe the latter, we fix a lexicogrphic ordering on the indices $(i, a)$, where $DOWN$ $<$ $UP$.  Notice that this choice ascribes a cyclic ordering
to the edges of the claw tree with $n$ leaves.  For $(i_1, a_1) < (i_2, a_2) < (i_3, a_3) < (i_4, a_4)$, we have the following
Pl\"ucker relation, these give a presentation of $\C[M_{0, n}(SL_2(\C))]$ when coupled with $p_{(i, UP), (i, DOWN)} = 1$. . 

\begin{equation}
p_{(i_1, a_1), (i_2, a_2)} p_{(i_3, a_3), (i_4, a_4)} - p_{(i_1, a_1), (i_3, a_3)} p_{(i_2, a_2), (i_4, a_4)} + p_{(i_1, a_1), (i_4, a_4)} p_{(i_2, a_2), (i_3, a_3)} = 0\\
\end{equation}

\subsection{$\Gamma$-tensors}

We define a distinguished set of elements $\mathcal{S}(\Gamma) \subset \C[M_{\Gamma}(SL_2(\C))]$ built from the Pl\"ucker generators in $\C[M_{0, 3}(SL_2(\C))]$, with $\Gamma$ serving as a combinatorial guide.   First we define an abstract $\Gamma$-form, which contains the necessary combinatorial information to define a member of $S(\Gamma)$.

\begin{definition}
An abstract $\Gamma$-form $\mathcal{V}(P, \phi, A)$ is the following information. 

\begin{enumerate}
\item For $v \in V(\Gamma),$ a collection $P_v$ of directed paths with endpoints in the leaves of $\Gamma_v \subset \Gamma$.\\
\item For $e \in E(\Gamma)$ which connects vertices $v, w$, an isomorphism of sets $\phi_e$ which identifies paths leaving $v$ with those going into $w$, and vice-versa.\\
\item A choice of orientation data $a_{\ell} \in \{UP,DOWN\}$ for each end point of a path which terminates at a leaf $\ell$ of $\Gamma$.\\
\end{enumerate}
\end{definition}

\begin{figure}[htbp]
\centering
\includegraphics[scale = 0.3]{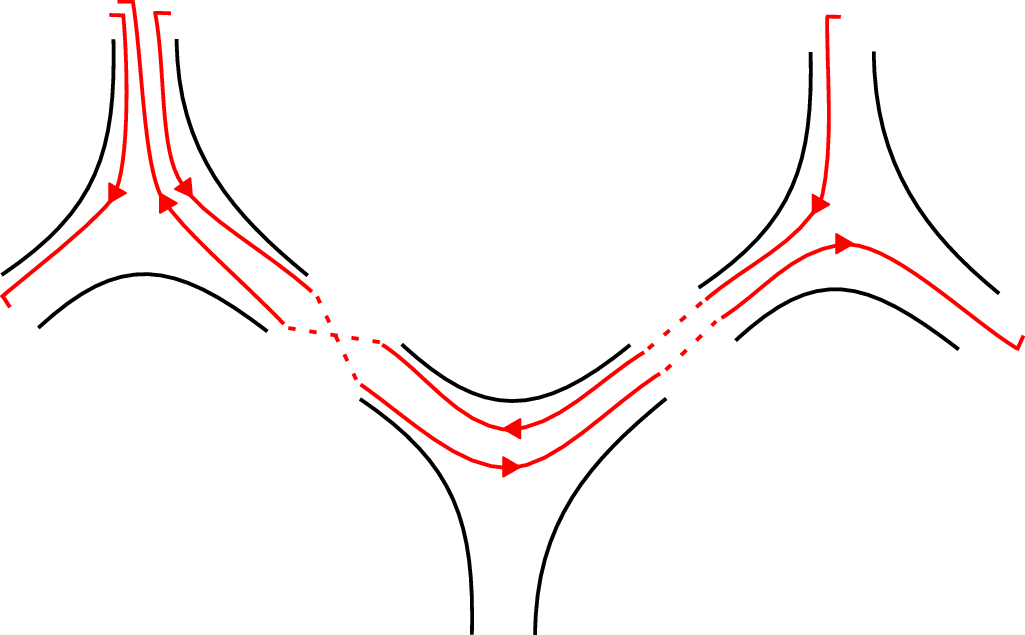}
\caption{A $\Gamma-$tensor.}
\label{agammatensor}
\end{figure}

\noindent
   For an abstract $\Gamma$-form $\mathcal{V}(P,\phi, A)$, we say that an assignment $B$ of orientation data to the non-leaf endpoints of paths in $P$ is $coherant$ if paths identified by a $\phi_e$ are given opposite orientations.  We place a copy of $M_{0, 3}(SL_2(\C))$ at each vertex $v \in V(\Gamma)$, with each leaf of $\Gamma_v$ assigned to one of the $3$ $SL_2(\C)$ actions.  The data $\mathcal{V}(P, \phi, A),$ $B$ then allows us to assign a Pl\"ucker generator in $\C[M_{0, 3}(SL_2(\C))]$ to each path, defining a tensor $V(P, \phi, B, A) \in \C[M_{0, 3}(SL_2(\C))^{V(\Gamma)}]$.    We let the signature $(-1)^{\sigma(P, \phi, B, A)} \in \{1, -1\}$ be $-1$ to the number times an outgoing path at a connecting edge $e$ is given the [DOWN] orientation.  We now define the concrete $\Gamma$-tensor $V(P, \phi, A)$ as the following sum in $\C[M_{0, 3}(SL_2(\C))^{V(\Gamma)}]$.

\begin{equation}
V(P, \phi, A) = \sum_{B} (-1)^{\sigma(P, \phi, B, A)} V(P, \phi, B, A)\\
\end{equation}

 We have already established $M_{0, 4}(SL_2(\C)) =$ $SL_2(\C) \backslash [M_{0, 3}(SL_2(\C)) \times M_{0, 3}(SL_2(\C))]$, we will now use this to construct $\mathcal{S}(\tree) \subset \C[M_{\tree}(SL_2(\C))] = \C[M_{0, 4}(SL_2(\C))]$ for $\tree$ a trivalent tree with $4$ leaves. 

\begin{lemma}\label{4plucker}
The Pl\"ucker generator $p_{(i,a),(j,b)}$ of $\C[M_{0, 4}(SL_2(\C))]$ is a member of $\mathcal{S}(\tree)$ for any $4-$tree $\tree$. 
\end{lemma}

\begin{proof}
We depict two trinodes meeting at a common edge in Figure \ref{Gammacompute1}, let $A_1, A_2$ be 
the $2\times 2$ matrices on the inside edges. 

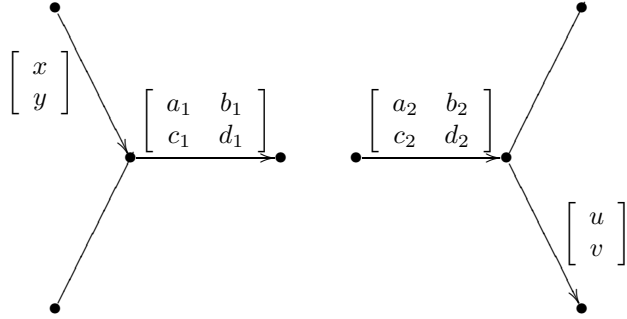
\begin{figure}[htbp]

$$
\begin{xy}
(-25, 0)*{\bullet} = "A1";
(-35,20)*{\bullet} = "A2";
(-35,-20)*{\bullet} = "A3"; 
(-5,0)*{\bullet} = "A4"; 
(5, 0)*{\bullet} = "B1";
(25, 0)*{\bullet} = "B2"; 
(35, 20)*{\bullet} = "B3";
(35, -20)*{\bullet} = "B4";
(-37, 10)*{\left[ \begin{array}{c} x\\ y \\ \end{array} \right]  };
(-15, 5)*{\left[ \begin{array}{cc} a_1 & b_1\\ c_1 & d_1\\ \end{array} \right]  };
(15, 5)*{\left[ \begin{array}{cc} a_2 & b_2\\ c_2 & d_2\\ \end{array} \right]  };
(37, -10)*{\left[ \begin{array}{c} u\\ v \\ \end{array} \right]  };
"A4"; "A1";**\dir{-}? >* \dir{>};
"A1"; "A2";**\dir{-}? >* \dir{>};
"A1"; "A3";**\dir{-}? >* \dir{-};
"B2"; "B1";**\dir{-}? >* \dir{>};
"B4"; "B2";**\dir{-}? >* \dir{>};
"B3"; "B2";**\dir{-}? >* \dir{-};
\end{xy}
$$\\
\caption{A Pl\"ucker generator as a $\tree-$tensor.}
\label{Gammacompute1}
\end{figure}

The $\Gamma$ tensor formed by taking 
some choice of columns $[x, y], [u, v]$ is computed as follows. 

\begin{equation}
 det\left[ \begin{array}{cc} x & b_1\\ y & d_1\\ \end{array} \right] det\left[ \begin{array}{cc} a_2 & u\\ c_2 & v\\ \end{array} \right] -det\left[ \begin{array}{cc} x & a_1\\ y & c_1\\ \end{array} \right] det\left[ \begin{array}{cc} b_2 & u\\ d_2 & v\\ \end{array} \right]\\
\end{equation}

\begin{equation}\label{pregamma}
xu(c_1d_2 - d_1c_2) -xv(c_1b_2-d_1a_2) +yv(a_1b_2 - b_1a_2) - yu(a_1d_2 - b_1c_2)\\
\end{equation}

Each of the forms in parentheses are minors in the rows of $A_1, A_2,$ and are therefore invariant with respect
to the right hand side action of $SL_2(\C)$.  It follows that this tensor is in the subalgebra $\C[M_{0,3}(SL_2(\C))^2/SL_2(\C))] \subset \C[M_{0, 3}(SL_2(\C))^2]$.  Recall (Section \ref{mgnsection}) that this subalgebra is isomorphic to $\C[M_{0, 4}(SL_2(\C))]$, and the isomorphism is computed by evaluating $A_1, A_2$ at the identity, this produces the Pl\"ucker element $xv-yu$. 
\end{proof}

The previous lemma gives an inductive proof of the following proposition. 

\begin{proposition}\label{plucker}
The Pl\"ucker generator $p_{(i, a), (j,b)}$ of $\C[M_{0, n}(SL_2(\C))]$ is a $\tree$-tensor for any tree $\tree$ with $n$ leaves. 
\end{proposition} 

As a consequence we obtain the following. 

\begin{proposition}\label{goinv}
The set of Pl\"ucker monomials in $\C[M_{0, n}(SL_2(\C))]$ agrees with the set of $\tree$-tensors $\mathcal{S}(\Gamma) \subset \C[M_{0, 3}(SL_2(\C))^{V(\tree)}]$ for any tree $\tree$ with $n$ leaves.
\end{proposition}

\begin{proof}
The set of $\tree$-tensors is closed under multiplication in $\C[M_{0, 3}(SL_2(\C))^{V(\tree)}]$, so all Pl\"ucker monomials  must be in the set of $\tree-$tensors by Proposition \ref{plucker}.   Furthermore, for any $\tree$ tensor, we can follow the bijections $\phi$ at pairs of joined edge elements to pick out a path in $\tree.$  It is easily verified that the $\tree$ tensor defined by this path is a Pl\"ucker generator. Since all $\tree$-tensors are monomial products of paths, it follows that any $\tree$-tensor can be identified with a monomial product of Pl\"ucker generators. 
\end{proof}

We can now prove the main result of this subsection.

\begin{proposition}\label{gammaequiv}
The sets of $\Gamma$-tensors, $\mathcal{S}(\Gamma) \subset \C[M_{0, 3}(SL_2(\C))^{V(\Gamma)}]$ are all in the invariant subalgebra $\C[M_{\Gamma}(SL_2(\C))] \subset \C[M_{0, 3}(SL_2(\C))^{V(\Gamma)}]$.  Furthermore, these sets all coincide $\mathcal{S}(\Gamma') = \mathcal{S}(\Gamma)$ under the isomorphisms $M_{\Gamma'}(SL_2(\C)) \cong M_{\Gamma}(SL_2(\C))$ defined by admissable maps of graphs.  
\end{proposition}

\begin{proof}
We take an edge $e$ which connects two distinct vertices $v, u \in V(\Gamma),$ and we let $\tree_e \subset \Gamma$ be the link of these endpoints.  Fixing a $\Gamma-$tensor $V(P, \phi, A)$, we consider a linear decomposition into multiplies of $\tree_e$-tensors given by summing over fixed orientations $B_e$ on the leaves of $\tree_e$.

\begin{equation}\label{treedecomposition}
V(P, \phi, A) = \sum_{B_e} (-1)^{\sigma_{B_e}} V(P_{\tree_e}, \phi_e, B_e) \otimes V(P_{\Gamma \setminus \{e\}}, \phi_{\Gamma \setminus \{e\}}, B_{\Gamma \setminus \{e\}})\\
\end{equation}

\noindent
Here $V(P_{\Gamma \setminus \{e\}}, \phi_{\Gamma \setminus \{e\}}, B_{\Gamma \setminus \{e\}})$ is a tensor on the graph $\Gamma \setminus \{e\}$ which depends on dual orientation data,  and $(-1)^{\sigma_{B_e}}$ is some sign.  Notice that if $V(P_{\tree_e}, \phi_e B_e)$ is an $SL_2(\C)$ invariant for each $B_e,$ then $V(P, \phi, A)$ is as well.  But as each of these forms is a $\tree_e$ tensor, this follows from Theorem \ref{goinv}.  

Each $V(P_{\tree_e}, \phi_e, A_e)$ is a product of Pl\"ucker generators independent of $\tree_e$.  This implies that each $V(P_{\tree_e}, \phi_e A_e)$ can be replaced with $\tree$-tensor for any fixed tree $\tree$ with $k+m$ leaves, and so $\mathcal{S}(\Gamma) = \mathcal{S}(\Gamma')$ for any mutation equivalent graph.

It remains to establish $SL_2(\C)$ invariance when $e$ connects a vertex $v$ to itself. If $\Gamma$ is not a single trinode with a loop, we may use admissable map equivalence on another edge to reduce to the previous case.  To treat the remaining case, we consider $M_{0, 3}(SL_2(\C))/SL_2(\C)$, where $SL_2(\C)$ acts diagonally through the actions on two leaves.  The proof of Lemma \ref{4plucker} can be used to handle this case, we leave this to the reader. 
\end{proof}

For any $\Gamma$-tensor $V(P, \phi, A)$, we can produce a new $\Gamma$-tensor $V(P_{\gamma}, \phi, A)$ by reversing the Pl\"ucker monomials $p_{ij} \to p_{ji}$ along a (possibly non-simple) path $\gamma$.  The following are easily verified by hand.

\begin{proposition}\label{gammaorient} If $\gamma$ is a closed, $V(P_{\gamma}, \phi,  A) = V(P, \phi, A),$ if $\gamma$ is open, $V(P_{\gamma}, \phi, A) = -V(P, \phi, A).$
\end{proposition}

\subsection{Initial forms of $\Gamma$-tensors}\label{weightfiltration}

There is a spin diagram $\Phi_{a(P, \phi, A)}$ associated to any $\Gamma$ tensor $V(P, \phi, A)$.  For an edge $e \in E(\Gamma)$, 
$a(P, \phi, A)(e)$ is the number of paths crossing $e$, and for any leaf $\ell_i$, the quantities $a(P, \phi, A)(x_i), a(P, \phi, A)(y_i)$
are the number of $UP$, respectively $DOWN$ orientations.    

\begin{proposition}\label{gammainitial}
Let $\Gamma$ be trivalent, then $\Phi_{a(P, \phi, A)}$ is the leading term of $V(P, \phi, A)$ in the spin diagram filtration of
$\C[M_{\Gamma}(SL_2(\C))]$. 
\end{proposition}

\begin{proof}
We observe that the initial term $in(V(P, \phi, A)) \in \C[M_{\Gamma}(SL_2(\C))]$ can be computed in the algebra $\C[M_{0, 3}(SL_2(\C))^{V(\Gamma)}]$, as the $SL_2(\C)^{E(\Gamma)}$ invariance of the spin diagram filtration guarantees that the initial term of an invariant in $\C[M_{0, 3}(SL_2(\C))^{V(\Gamma)}]$ is also invariant.  For any choice $B$ of orientations on the internal edges of $\Gamma$, the initial term $in(V(P, \phi, B, A)) \in \C[M_{0, 3}(SL_2(\C))^{V(\Gamma)}$ is a tensor product (over $v \in V(\Gamma)$) of initial terms of the monomials defined by the $P_v$ in $\C[M_{0, 3}(SL_2(\C))]$.  Each of the $in(V(P, \phi, B, A))$ lies in the $a(P, \phi, A)$ isotypical component of $\C[M_{0, 3}(SL_2(\C))^{V(\Gamma)}]$, so it follows that $in(V(P, \phi, A))$ equals the sum of these terms.   Initial terms with different quantities of the two possible orientations $\{UP, DOWN\}$ at an edge are linearly independent, as this is the case for the associated elements of $\C[H_{0,3}^{V(\Gamma)}]$. Since  $\mathcal{V}(P, \phi, A)$ always has exactly one term with all outgoing assignments $UP$, this sum cannot vanish.  
\end{proof}

\begin{corollary}\label{valgamma}
A valuation $v_x$,$x \in C_{\Gamma},$ is computed on $\mathcal{V}(P, \phi, A) \in \mathcal{S}(\Gamma) \subset \C[M_{\Gamma}(SL_2(\C))]$ as follows. 

\begin{equation}
v_x(\mathcal{V}(P, \phi, A)) = \sum_{e \in E(\Gamma)} x(e)a(P, \phi, A)(e)\\
\end{equation}

\end{corollary}

We have stated this corollary for trivalent $\Gamma$, but the general case can be recovered by considering
a weighting on $\Gamma$ as a weighting with $0$ entries on a trivalent cover $\tilde{\Gamma} \to \Gamma.$

\subsection{Planar $\Gamma$-tensors}

Now we use a ribbon structure on $\Gamma$ to find a way to lift spin diagrams $a\in H_{\Gamma}$ back to $\C[M_{\Gamma}(SL_2(\C))]$. The elements which result from this construction are called planar $\Gamma$-tensors. For a vertex $v \in V(\Gamma),$ with edges $e, f, g$, we consider the weightings $a(e), a(f), a(g)$ defined by $a \in H_{\Gamma}.$  We place linear orders on the endpoints of these paths which are consistent with the cyclic ordering.  The arrangement of paths is declared to be planar if for every ordering of edges $e \to f$, and pair of paths $p_1, p_2$ if the endpoint of $p_1$ comes before the endpoint of $p_2$ in $e$, then the endpoint of $p_2$ comes before the endpoint of $p_1$ in $f.$  There is a unique planar way to arrange $x_{ij}$ paths inside this object such that the number paths passing through a given edge $i$ is $w(i),$ this is given by the equations  $x_{ef} = \frac{1}{2}(w(e) + w(f) - w(g))$, $ x_{fg} = \frac{1}{2}(w(f) + w(g) - w(e))$, $x_{ge} = \frac{1}{2}(w(g) + w(e) - w(f))$. 

\begin{figure}[htbp]
\centering
\includegraphics[scale = .35]{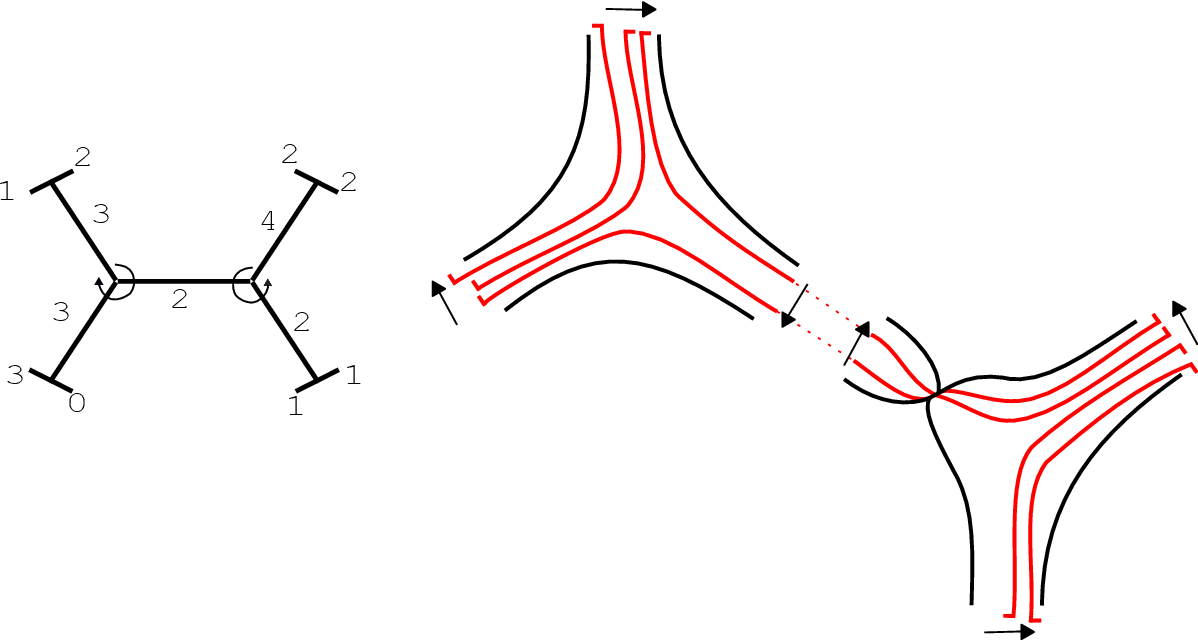}
\caption{A planar $\Gamma-$tensor.}
\label{Tube}
\end{figure}

 If some edge of $v$ is a leaf $\ell_i$, we must decide how to lift the data $a(x_i), a(y_i).$ In our description of $\C[M_{0, 3}(SL_2(\C))]$, these numbers count the number of orientations of each type, we place these on the linearly ordered ends of the paths, with all of the $a(x_i)$ [UP] orientations coming first.   If two trinodes $w, v$ share a common edge $e$, the chosen linear orders of the paths meeting along this edge define a unique identification map $\phi$ by sending first path endpoint to last.  Figure \ref{Tube} depicts an example of this lift.  Finally, following Proposition \ref{gammaorient}, we assign directions on the paths which have endpoints, going from smallest index to largest.  In this way, $a$, along with the chosen ribbon structure on $\Gamma$ define an abstract $\Gamma$-tensor $V(P_a, \phi_a, A_a)$. The $\Gamma$-tensors obtained in this way must satisfy the following properties by construction. 

\begin{proposition}\label{planarassocgraded}
The $V(P_a, \phi_a, A_a)$ are a basis of $\C[M_{\Gamma}(SL_2(\C))],$ and $in(V(P_a, \phi_a, A_a)) = \Phi_a.$ 
\end{proposition}

\subsection{Relations}\label{skein}

We conclude our $SL_2(\C)$ analysis with the skein relations, see also \cite{FG} page 186, and \cite{PS}, for discussions of these relations. First we introduce the notion of a cap on a $\Gamma-$tensor, this is meant to represent a backtrack in the corresponding arrangement of paths.  We fix two paths in a trinode, both with endpoints in an edge $e$, one terminating in this edge, the other terminating elsewhere.  A cap on these two paths is the form $p_{(i, a)(e, DOWN)}p_{(e, UP),(j, b)} - p_{(i, a)(e, UP)}p_{(e, DOWN),(j, b)} $, for $(i, a), (j, b)$ other indices in the trinode at the begining and end of these paths.   The following are verified by calculation. 

\begin{enumerate}
\item If $\{i, j\}$ are different edges, the cap yields $p_{(i, a)(j, b)}.$\\
\item If $i = j$ a common leaf edge, and $a, b$ above are the leaf orientation data on these paths, then the cap is $0$ if $a = b$
and $\pm 1$ if $a \neq b$. 
\item If $i = j$, a common non-leaf edge, and the paths $p_{(i, -)(e, -)}$ and $p_{(e, -), (j, -)}$ connect
to paths $p_{(k, -)(s, -)}$ and $p_{(s, -)(\ell, -)}$ respectively, then the resulting tensor simplifies to a cap on 
 $p_{(k, -)(s, -)}$ and $p_{(s, -)(\ell, -)}$ in the $s$ indices.\\
\item If a caps are applied to the ends of a pair of paths on the same set of edges, the resulting loops contracts to give a multiple of $2.$
\end{enumerate}

\begin{figure}[htbp]
\centering
\includegraphics[scale = .4]{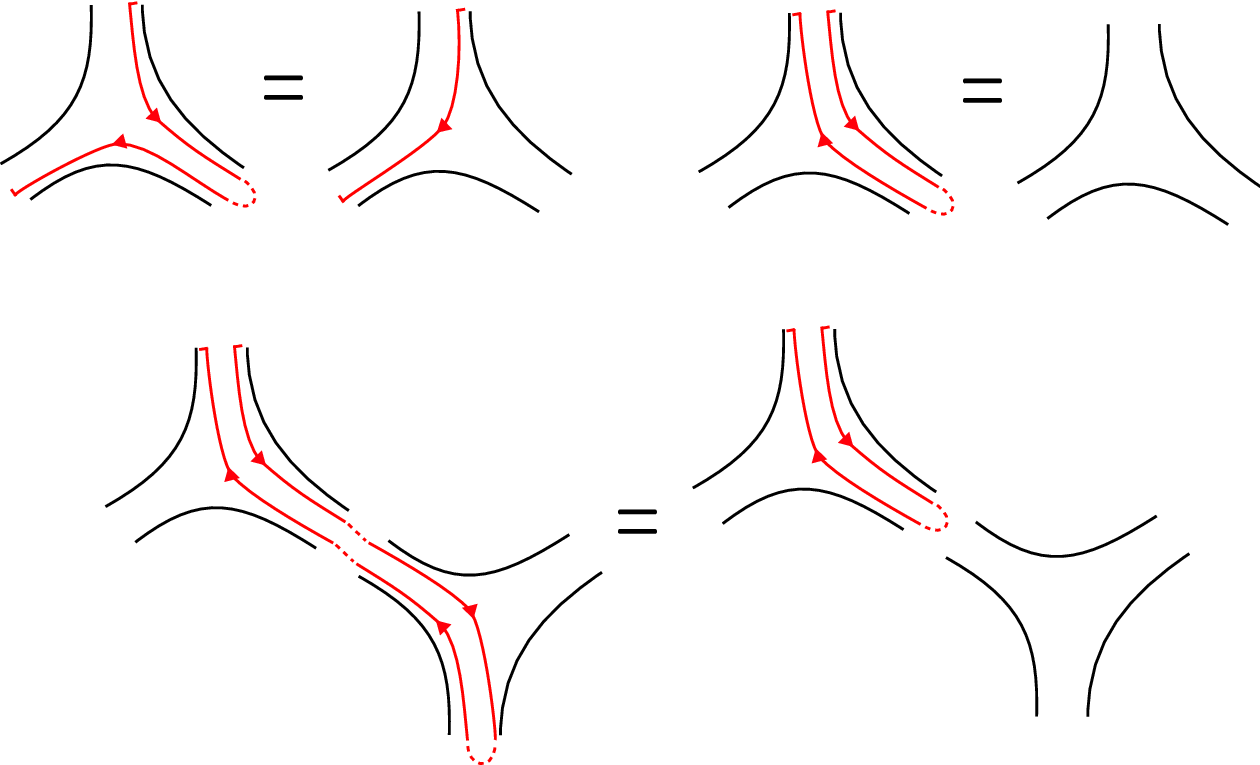}
\caption{Retracting caps.}
\label{Cap}
\end{figure} 

We suppose a linear order has been introduced along the endpoints in each  edge of two paths in a trinode.  In the definition of planar $\Gamma-$tensor, the paths in a trinode are in a planar arrangement if they are connected along these orders, first to last.  If this is not the case, the paths "cross", next we see how to address this situation.

\begin{lemma}\label{triuncross}
Given two paths which cross inside of a trinode, we can perform a Pl\"ucker relation which yields a sum of elements, each with the same ordering of endpoint indices, but with no crossings.
\end{lemma}

\begin{proof}
We make a general calculation, fix four endpoints $(f_i, y_i)$, and a pair of edges $e_1, e_2$ where paths are glued.  Then
a product of paths:

$$(p_{(f_1, y_1),(e_1, x_1)}p_{(e_2, y_2), (f_4, y_4)} - p_{(f_1, y_1),(e_1, y_1)}p_{(e_2, x_2), (f_4, y_4)}) $$

$$\times$$

$$(p_{(f_2, y_2),(e_1, x_1)}p_{(e_2, y_2), (f_3, y_3)} - p_{(f_2, y_2),(e_1, y_1)}p_{(e_2, x_2), (f_3, y_3)})$$

\noindent
minus its "uncrossing"

$$(p_{(f_1, y_1),(e_1, x_1)}p_{(e_2, y_2), (f_3, y_3)} - p_{(f_1, y_1),(e_1, y_1)}p_{(e_2, x_2), (f_3, y_3)}) $$

$$\times$$

$$(p_{(f_2, y_2),(e_1, x_1)}p_{(e_2, y_2), (f_4, y_4)} - p_{(f_2, y_2),(e_1, y_1)}p_{(e_2, x_2), (f_4, y_4)})$$

\noindent
is the product of two caps on $e_1$ and $e_2$ respectively. 

$$(p_{(f_1, y_1),(e_1, x_1)}p_{(e_1, y_1), (f_4, y_4)} - p_{(f_1, y_1),(e_1, y_1)}p_{(e_1, x_1), (f_4, y_4)}) $$

$$\times$$

$$(p_{(f_2, y_2),(e_2, x_2)}p_{(e_2, y_2), (f_3, y_3)} - p_{(f_3, y_3),(e_2, y_2)}p_{(e_2, x_2), (f_3, y_3)})$$
\end{proof}

\begin{figure}[htbp]
\centering
\includegraphics[scale = .45]{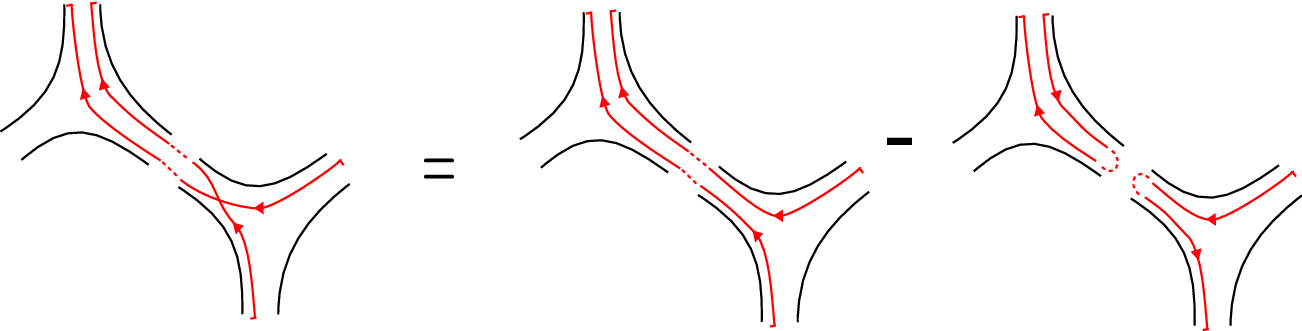}
\caption{An uncrossing relation.}
\label{Cross}
\end{figure}

The endpoints of the paths were not chosen in any particular edges of the meeting trinodes, which allows us to adapt this argument to operations performed locally inside of larger $\Gamma-$tensors. For any tree $\tree \subset \Gamma$, we have seen that we can decompose a $\Gamma$-tensor $V(P, \phi, A)$ as a sum of $\tree-$tensors paired with tensors from $\Gamma \setminus \tree,$ the graph obtained by eliminating all internal edges and  vertices of $\tree,$
as in Equation \ref{treedecomposition}.  Given a ribbon structure on $\Gamma$, we can induce a ribbon structure on a subtree $\tree$, which then gives a cyclic ordering on the leaves.  The Pl\"ucker algebra comes with a notion of planarity, namely a monomial is planar when $(i, a) < (\ell, b) < (j, c) < (k, d)$ never holds for $p_{(i,a)(j,c)}p_{(\ell,b)(k,d)} | M$. We leave it to the reader to verify that our notion of planarity matches this notion.

\begin{lemma}\label{planarplucker}
A Pl\"ucker monomial is planar if and only if it can be given the structure of a planar $\Gamma-$tensor. 
\end{lemma}

The Pl\"ucker relations correspond locally to those we describe in Lemma \ref{triuncross}, we can therefore modify $V(P, \phi, A)$ by replacing each $V(P_{\tree}, \phi_{\tree}, A_{\tree}, B_{\tree})$ in the expansion with a sum of two forms produced by  a Pl\"ucker relation.  Each of the resulting summands have the same index data $A_{\tree}, B_{\tree}.$ This observation allows us to write $V(P, \phi, A) = V(P'\phi', A') - V(P'', \phi'', A'')$, where the two forms on the right hand side satisfy the duality condition on the orientation indices at the boundary edges where $\tree, \Gamma \setminus \tree$ meet.   It is a straightforward case check to show that caps can be retracted, and directions in the resulting paths can be coherantly reversed to conform to the definition of $\Gamma$-tensor.  These calculations yield rules for how closed and open paths interact, see Figure \ref{Tube2}. 

\begin{figure}[htbp]
\centering
\includegraphics[scale = .35]{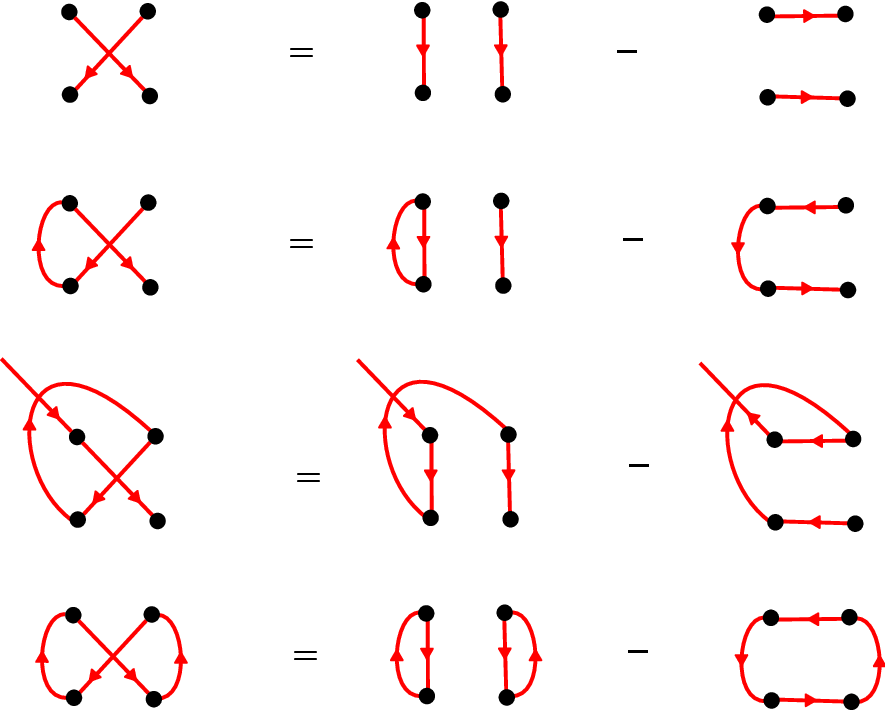}
\caption{Rules for resolving crossings.}
\label{Tube2}
\end{figure}

Now we outline how these relations can be used to expand any $V(P, \phi, A)$ into planar $\Gamma-$tensors (note that we already know that this is the case on a tree $\tree$ by Lemma \ref{planarplucker}). We assign linear orders to the endpoints of the paths terminating at each edge in a trinode of $\Gamma$, such that the $\phi_e$ are all planar bijections, and all leaf indices are as in the definition of planar $\Gamma$-tensor. This presents $V(P, \phi, A)$ as ''almost" planar, with crossings inside the trinodes of $\Gamma.$  We use Pl\"ucker relations as in Lemma \ref{triuncross} to make the arrangements of paths in each trinode planar. Uncrossings do not affect the order of the indices along the edges, so we obtain a sum of planar $\Gamma$-tensors.  Note that this sum is independent of our choices, as the planar $\Gamma-$tensors form a basis of $\C[M_{g, n}(SL_2(\C))]$.

 In order to take the product of two planar $\Gamma-$tensors, one chooses any mixing of the orderings along the edges which satisfies the planar conditions, and expands, as above.  By this process, a product $V(P_w, \phi_w, A_w)V(P_{w'}, \phi_{w'}, A_{w'})$ of two planar $\Gamma$-tensors is $V(P_{w + w'}, \phi_{w + w'}, A_{w +w'})$ plus a sum of planar $\Gamma-$tensors with strictly smaller spin diagrams in the partial order on edge weights.

\subsection{Regular functions associated to $\Gamma$-tensors}\label{idskein}

We restrict attention to the case $n = 0$ and show that $\Gamma-$tensors coincide with a special class of regular functions
on the character variety $\mathcal{X}(F_g, SL_2(\C))$ called trace-word functions. The trace-word function $\tau_{\omega}$ associated to a word $\omega \in F_g$ is the regular function that takes $(A_1, \ldots, A_g) \in \mathcal{X}(F_g, SL_2(\C))$ to the complex number $tr(\omega(A_1, \ldots, A_g))$.   Let $\mathcal{W}_g \subset \C[\mathcal{X}(F_g, SL_2(\C))]$ be the set of these functions, we will prove the following proposition with a series of lemmas.

\begin{proposition}\label{tracetensor}
The set $\mathcal{S}(\Gamma) \subset \C[M_{\Gamma}(SL_2(\C))] \cong \C[\mathcal{X}(F_g, SL_2(\C))]$ coincides with the set of monomials in the elements of $\mathcal{W}_g$. 
\end{proposition}

By Proposition \ref{gammaequiv}, it suffices to prove Proposition \ref{tracetensor} for $\Gamma = \Gamma_{g, 0}.$ We fix a direction
on the loop edges of $\Gamma_{g, 0}$, and show that these sets are combinatorially in bijection with each other.  Fix a connected abstract $\Gamma_{g,0}$-tensor $\mathcal{V}(P, \phi)$.  We let $T_g$ be the tree obtained by splitting the loop edges of $\Gamma_{g, 0}$, the leaves of this tree are labelled with the $2g$ indices $2i-1, 2i$ $1 \leq i \leq g$.  The label $2i-1$ is given to the outgoing leaf according to the direction assigned to $\Gamma_{g, 0}$, and $2i-1,$$2i$ leaves are identified in the quotient map $T_g \to \Gamma_{g,0}.$  

 Let $p_{i_1, j_1} \in P$ be a path in $V(P, \phi).$  The endpoint $j_1$ is connected with the starting point of the next path $p_{(i_2, j_2)},$ so $j_1 = 2a_1 -1, i_2 = 2a_1$ or $j_1 = 2a_1, i_2 = 2a_1-1$.   Continuing this way, we obtain a word $\omega(\mathcal{V}(P, \phi), p_{(i_1, j_1)}) = x_{a_1}^{\epsilon_1}x_{a_2}^{\epsilon_2}\ldots x_{a_m}^{\epsilon_m}$, where the sign $\epsilon_i$ is determined by the rule $\{j_1 = 2a_1 -1, i_2 = 2a_1 = 2\} \implies \epsilon_i = 1$, $\{j_1 = 2a_1, i_2 = 2a_1-1\} \implies \epsilon_i = -1.$  This word is reduced as the paths defined by $\Gamma_{g,0}$ tensors do not have caps by definition. 

 Given a word $\omega$, we may reverse this recipe to obtain a $\Gamma_{g,0}-$tensor $\mathcal{V}(P_{\omega}, \phi_{\omega}).$  All paths in $T_g$ are determined by their endpoints, so we have a $1-1$ map from the set of connected $\Gamma_{g,0}$-tensors $\mathcal{V}(P, \phi)$ with a choice of initial path  $p_{(i_1,j_1)} \in S$ to the set of reduced words in $F_g.$  Changing the initial path amounts to changing the word $\omega(\mathcal{V}(P, \phi), p_{(i_1,j_1)})$ by a cyclic permutation, so we have proved the following lemma. 

\begin{lemma}
The set $\mathcal{S}_{con}(\Gamma_g)$ of connected $\Gamma-$tensors is in bijection with cyclic equivalence classes of reduced words in $F_g.$
\end{lemma}

Now we show that any connected $\Gamma_{g,0}-$tensor can be ''untangled."

\begin{definition}
Let $V_g \in \C[M_{\Gamma_{g,0}}(SL_2(\C))] = \C[SL_2(\C) \backslash SL_2(\C)^{2g}/SL_2(\C)^g]$ be the $\Gamma_{g,0}$ tensor defined 
by joining the paths $p_{(2g, 1)} \to p_{(2, 3)} \to \ldots \to p_{(2g-2, 2g-1)}$, as in Figure \ref{cycle}. 
\end{definition}

\begin{figure}[htbp]
\centering
\includegraphics[scale = 0.5]{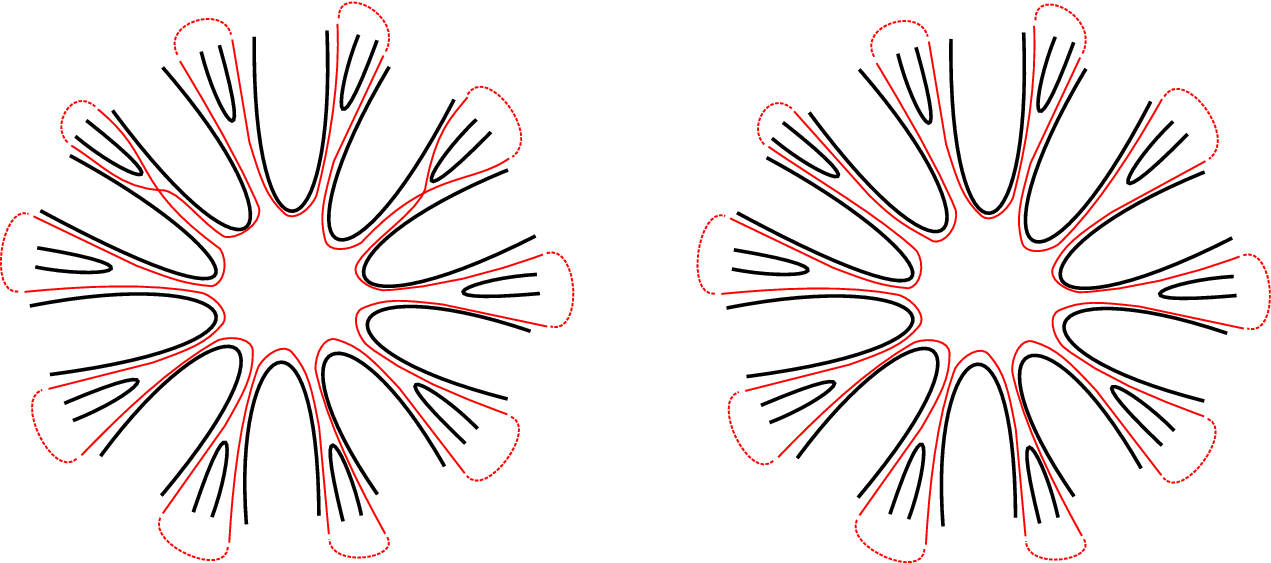}
\caption{Right: $V_{10}$,  Left: A $\Gamma_{10,0}-$tensor obtained from a word with inverses.}
\label{cycle}
\end{figure}

\noindent
Fix a reduced word $\omega \in F_g$, and we let $n = |\omega|$.  There is a natural map $\phi_{\omega}: \mathcal{X}(F_g, SL_2(\C)) \to \mathcal{X}(F_n, SL_2(\C))$ defined by sending $(A_1, \ldots, A_g)$ to the entries of $\omega$ in order, $(\omega_1(\vec{A}), \ldots, \omega_n(\vec{A)}).$ The space $\mathcal{X}(F_g, SL_2(\C))$ is the quotient $SL_2(\C) \backslash SL_2(\C)^{2g}/ SL_2(\C)^g = M_{0, 2g}(SL_2(\C))/SL_2(\C)^g$, by way of the map $\pi_{0, g}(k_1, h_1, \ldots, k_g, h_g) = (k_1h_1^{-1}, \ldots, k_gh_g^{-1}).$  We define a map $\phi_{\omega}'$ by sending $(k_1, h_1, \ldots, k_g, h_g)$ to $(\omega_1(\vec{k}), \omega_1(\vec{h}), \ldots, \omega_n(\vec{k}), \omega_n(\vec{h}))$, this commutes with $\phi_{\omega}$ under the isomorphisms $\pi_{0, g}, \pi_{0, n}.$  We will see where $V_g$ goes under $\phi_{\omega}^*.$ 

\begin{lemma}
\begin{equation}
\phi_{\omega}^*(V_n) = V(P_{\omega}, \phi_{\omega})\\
\end{equation}
\end{lemma}

\begin{proof}
When the Pl\"ucker element $p_{(2i, a), (2i+1, b)}$ is evaluated on $\phi_{\omega}'(k_1, h_1, \ldots, k_g, h_g)$ it gives the determinant of the $a$ column of $\omega_i(\vec{h})$ and the $b$ column of $\omega_{i+1}(\vec{k}).$  The pullback $\phi_{\omega}^*(V_n)$ therefore coincides with $V(P_{\omega}, \phi_{\omega})$ by definition.     
\end{proof}

We now introduce the model trace-word $\tau_g = tr(A_1A_2 \ldots A_g)$.  By the previous lemma, if $\tau_n = V_n \in \C[\mathcal{X}(F_n, SL_2(\C))]$, we have $\mathcal{V}(P_{\omega}, \phi_{\omega}) = \tau_{\omega} = tr(\omega(\vec{A})),$ and Proposition \ref{tracetensor}  follows. 

\begin{lemma}\label{rosetrace}

\begin{equation}
\tau_n = V_n \in \C[\mathcal{X}(F_n, SL_2(\C))]\\
\end{equation}
\end{lemma}

\begin{proof}
We employ the following commutative diagram. 
$$
\begin{CD}
 ([SL_2(\C)\times SL_2(\C)]/SL_2(\C) \times [SL_2(\C) \times SL_2(\C)]/SL_2(\C))^{g-1} @>>> SL_2(\C) \times SL_2(\C)^{g-1}\\
@VVV @VVV\\
([SL_2(\C) \times SL_2(\C)]/SL_2(\C))^{g-1} @>>> SL_2(\C)^{g-1}\\
\end{CD}
$$

$$
\begin{CD}
 (k_1, h_1), (k_2, h_2), \ldots, (k_g, h_g) @>>> (k_1h_1^{-1}, \ldots, k_gh_g^{-1})\\
@VVV @VVV\\
(k_1h_1^{-1}, h_2k_2^{-1}), \ldots, (k_g, h_g) @>>> (k_1h_1^{-1}k_2h_2^{-1}, \ldots, k_gh_g^{-1})\\
\end{CD}
$$

The horizontal arrows in this diagram give the isomorphism between $M_{0, 2g}(SL_2(\C))/SL_2(\C)^g$ and $\mathcal{X}(F_g, SL_2(\C))$. 
The vertical arrow on the right is the matrix composition operation $(A_1, \ldots, A_g) \to (A_1A_2, \ldots, A_g) \in SL_2(\C)^{g-1}.$  The vertical arrow on the left lifts this to $([SL_2(\C) \times SL_2(\C)]/SL_2(\C))^{g}$, where we can talk about $\Gamma-$tensors.  

Notice that $\tau_{g-1}$ pulls back to $\tau_g$ under the right vertical arrow.  It follows that if we can show that $\tau_1 = V_1 \in \C[SL_2(\C)]^{SL_2(\C)},$ and that $V_{g-1}$ pulls back to $V_g$ under the left vertical arrow, we have the lemma.  The first assertion holds by a straightforward calculation.  For the second assertion, it suffices to show that the piece of $V_g$ connecting the $(k_1, h_1), (k_2, h_2)$ is the piece of $V_{g-1}$ connecting the first pair of indices, evaluated at $(k_1h_1^{-1}, h_2k_2^{-1})$.   Fix $2\times 2$ matrices $A, B, C, D \in SL_2(\C)$ and two column vectors $X, U$. Let the columns of a matrix $K$ be $K_1, K_2$, and let $K_1\wedge K_2$ denote the determinant of $K.$ Furthermore, let $M = AB^{-1}, N = DC^{-1}.$ The lemma follows if 

\small
\begin{equation}
(X\wedge A_1) (B_2\wedge C_1) (D_2\wedge U) - (X\wedge A_2) (B_1\wedge C_1) (D_2\wedge U) - (X\wedge A_1) (B_2\wedge C_2) (D_1\wedge U) + (X\wedge A_2) (B_1\wedge C_2) (D_1\wedge U)\\
\end{equation}
\normalsize

\noindent
is equal to
\small
\begin{equation}
(X \wedge M_1)(N_2 \wedge U) - (X \wedge N_2)(N_1 \wedge U)\\
\end{equation}
\normalsize
\noindent
This can be done by a direct computation. 
\end{proof}

Taking this construction back through Proposition \ref{gammaequiv}, a reduced word $\omega$ gives a connected path in $\Gamma$ with spanning tree $\tree$ as follows. For each entry $\omega_i$, place an appropriately directed path $p_i$ on the associated edges in $E(\Gamma) \setminus E(\tree)$.  Then connect the endpoints $p_i \to p_{i+1}$ of these paths through the unique path in $\tree.$ We note that this identification of $\Gamma-$tensors and trace-words depends on the choice of the spanning tree $\tree$ and a direction on each edge in $E(\Gamma) \setminus E(\tree).$  Indeed, this is also the case with the isomorphism $M_{\Gamma}(SL_2(\C)) \cong \mathcal{X}(F_g, SL_2(\C))$, we explore this issue in a companion paper \cite{M16}.

\begin{remark}
The sum the weights along the edges of $\Gamma_{g,0}$ for the spin diagram of $V(P_{\omega}, \phi_{\omega})$ is the word length of $\omega.$ This shows that the natural filtration on $\C[\mathcal{X}(F_g, SL_2(\C))]$ by trace-word length fits into the framework of the filtrations we consider in this paper. 
\end{remark}

\section{Skein relations and conformal blocks}\label{compact}

In this section we focus on the case $\Gamma$ a trivalent graph, with associated curve $(C_{\Gamma}, \vec{p}_{\Gamma}).$ We will show that the valuations $v_x, x \in \mathcal{C}_{\Gamma}$ in Subsection \ref{weightfiltration} extend to $W_{C, \vec{p}}(SL_2(\C))$, and that the associated graded algebra of $v_x,$ when $x(e) \neq 0$ for all $e \in E(\Gamma)$ is a Rees algebra of the semigroup algebra $\C[H_{\Gamma}].$   This gives a method for constructing presentations of the algebras $W_{C_{\Gamma}, \vec{p}_{\Gamma}}(SL_2(\C))$.

\subsection{The $0, 3$ case}

First we compute a presentation of the coordinate ring of $B_{0, 3}(SL_2(\C)).$  Proposition \ref{g0rees} gives the following inclusion. 

\begin{equation}
W_{0, 3}(SL_2(\C)) = \bigoplus_{i, j, k, L} W_{0, 3}(i, i, j,j, k, k, L) \subset\\  
\end{equation}

$$\bigoplus_{i, j, k, L}[V(i, i) \otimes V(j, j) \otimes V(k, k)]^{SL_2(\C)}t^{L} \subset \C[M_{0, 3}(SL_2(\C))]\otimes \C[t]$$

The $SL_2(\C)$ case of Lemma \ref{ueno} and the Clebsh-Gordon rule imply that the  space $W_{0, 3}(i,i, j, j, k, k, L)$ is either $0$ or $V(i)\otimes V(j) \otimes V(k)$, and that it is non-trivial precisely when the following conditions are satisfied. 

\begin{enumerate}
\item (Clebsh-Gordon) $i, j, k$ are the lengths of the sides of a triangle.\\
\item (Clebsh-Gordon) $i + j + k \in 2\Z.$\\
\item (Lemma \ref{ueno}) $i + j + k \leq 2L$\\
\end{enumerate}

The space  $[V(i, i) \otimes V(j, j) \otimes V(k, k)]^{SL_2(\C)}$ is spanned by those  Pl\"ucker monomials in $\mathcal{S}(0, 3) \subset \C[M_{0, 3}(SL_2(\C))]$ 
with $i, j, k$ endpoints in the first, second, and third edges a trinode, respectively. From this it follows that the $L-$th level of the level filtration is spanned by the Pl\"ucker monomials of degree $\leq L.$  As a consequence $W_{0, 3}(SL_2(\C))$ is presented by the $12$ Pl\"ucker generators and any "empty" generator, subject to $15$ homogenized Pl\"ucker relations.

$$\Pncc \Pnck \Pnkc \Pnkk \Pcnc \Pcnk  \Pknc \Pknk \Pccn \Pckn \Pkcn \Pkkn \tri$$ 
\bigskip

$$\tri^2 -  \Pccn  \Pkkn  + \Pckn \Pkcn = \tri \Pncc - \Pccn \Pknc +  \Pcnc \Pkcn = 0$$

\bigskip

In particular, $B_{0, 3}(SL_2(\C))$ is the codimension $2$ subvariety of the Grassmannian $Gr_2(\C^6)$ defined by the equations $p_{12} = p_{34} = p_{56}.$  A  presentation of the graded coordinate ring of the divisor $D_{0,3} \subset B_{0, 3}(SL_2(\C))$ is obtained by setting the empty generator equal to $0$. 

$$\Pccn  \Pkkn - \Pckn \Pkcn = 0 \ \ \ \ \Pccn \Pknc -  \Pcnc \Pkcn = 0$$

\bigskip

\subsection{Skein relations}

Recall the spaces $W_{C_{\Gamma}, \vec{p}_{\Gamma}}(L) = \bigoplus_{r_i \leq L} W_{C_{\Gamma}, \vec{p}_{\Gamma}}(\vec{r}, \vec{r}, L)$. 

\begin{proposition}\label{levelbasis}
The space $W_{C_{\Gamma}, \vec{p}_{\Gamma}}(L)$ has the set of planar $\Gamma-$tensors with $\leq L$ paths through each vertex $v \in V(\Gamma)$ as a basis. The corresponding filtration on  $V_{C_{\Gamma}, \vec{p}_{\Gamma}}(SL_2(\C))$ has an identical description. 
\end{proposition}

\begin{proof}
By Proposition \ref{gcorrelate}, the spin diagrams $\Phi_a$ with $a(e) + a(f) + a(g) \leq 2L$ are a basis of the space $W_{C_{\Gamma}, \vec{p}_{\Gamma}}(L)$, the same therefore holds for the corresponding planar $\Gamma-$tensors $V(P_a, \phi_a, A_a)$.
\end{proof}

\begin{definition}
For a vertex $w \in V(\Gamma)$, let $x_w \in C_{\Gamma}$ be defined by $x_w(e) = 1$ if $e$ has $w$ as an endpoint
and $0$ otherwise.  Let $v_w: \C[M_{g, n}(SL_2(\C)) \to \Z \cup \{-\infty\}$ be the corresponding valuation. 
\end{definition}

\begin{corollary}\label{traceconformal}
The space $W_{C_{\Gamma}, \vec{p}_{\Gamma}}(L)$ has the set of planar $\Gamma-$tensors with $v_w(V(P, \phi, A)) \leq 2L$ 
as a basis.  In the $n = 0$ case, $V_{C_{\Gamma}}(L)$ is spanned by the trace-words $\tau_{\omega}$ with $v_w(\tau_{\omega}) \leq 2L$ for all $w \in V(\Gamma).$ 
\end{corollary}

Next we define an increasing filtration on the affine semigroup $H_{\Gamma}$. 

\begin{definition}
The $L-$th part $H_{\Gamma}(L)$ is defined to be the set of weightings $a \in H_{\Gamma}$ with $a(e) + a(f) + a(g) \leq 2L$ for $e, f, g$ containing a common vertex $w \in V(\Gamma).$
\end{definition}

\begin{proposition}\label{commonedegen}
The Rees algebra $\C[H_{\Gamma}^*]$ of $\C[H_{\Gamma}]$  is an associated graded algebra of $W_{C_{\Gamma}, \vec{p}_{\Gamma}}(SL_2(\C))$. 
\end{proposition}

\begin{proof}
This is a consequence of Propositions \ref{levelbasis} and \ref{mtor}. 
\end{proof}

\begin{definition}
We let $\mathcal{H}_{\Gamma}(1)$ and be the polytope in $\mathcal{H}_{\Gamma}$ defined by $a(e) + a(f) + a(g) \leq 2$.
\end{definition}

Proposition \ref{commonedegen} implies that $B_{C_{\Gamma}, \vec{p}}(SL_2(\C))$ has a toric degeneration to the toric variety associated
to the rational polytope $\mathcal{H}_{\Gamma}(1)$.  Furthermore, $W_{C_{\Gamma}, \vec{p}_{\Gamma}}(SL_2(\C))$ is generated by those $\Gamma$ tensors whose spin diagrams generate $H_{\Gamma}^*$, subject to skein relations of level $\leq$ the degree of relations necessary to present $H_{\Gamma}^*.$  One can define $U_{\Gamma}^*$ and $\mathcal{U}_{\Gamma}^*(1)$ in an identical manner to $H_{\Gamma}^*$, $\mathcal{H}_{\Gamma}(1)$, which then satisfies the analogue of Proposition \ref{commonedegen} with the algebra $V_{C_{\Gamma}, \vec{p}_{\Gamma}}(SL_2(\C))$.  In general, $H_{\Gamma}^*$ is generated by diagrams of level $\leq g +1$. This follows from a modification of the main results of \cite{BBKM} and \cite{M4}, which pertains to the semigroup algebra $\C[U_{\Gamma}^*].$  Now  we discuss these conclusions for some example graphs.

\subsection{Genus $2$ curves}

The algebras $V_{C_{\Gamma}}(SL_2(\C))$ for both of the genus $2$ trivalent graphs are isomorphic to the affine semigroup algebras $\C[H_{\Gamma}^*].$ In Figure \ref{CExamples} we show the defining polytopes $\mathcal{H}_{\Gamma}(1).$

\begin{figure}[htbp]
\centering
\includegraphics[scale = 0.25]{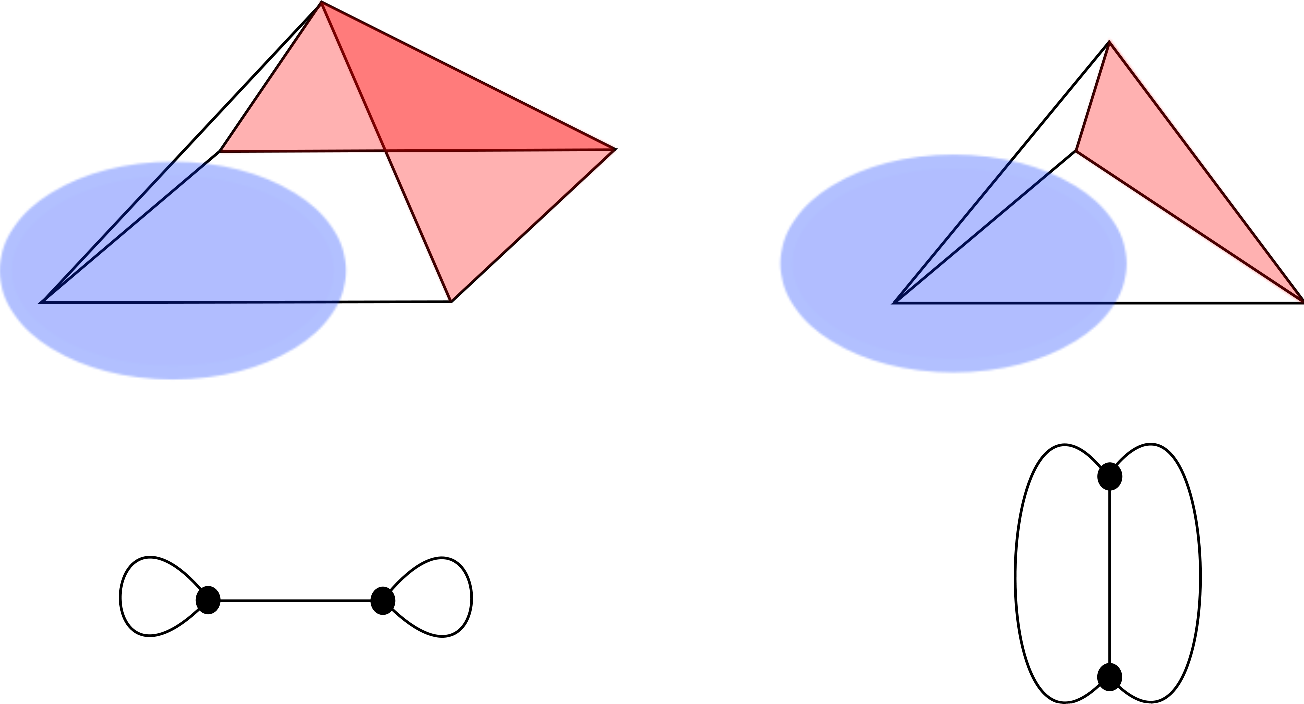}
\caption{The two compactifications of $\mathcal{X}(F_2, SL_2(\C)) = \C^3$.}
\label{CExamples}
\end{figure} 

\noindent
Note that these polytopes agree outside of the highlighted facets, which correspond to the divisor $D_{\Gamma}$.

\subsection{Genus $3$ curves}\label{genus3}

The trivalent genus $3$ graphs are depicted in Figure \ref{g3}.  
We analyze the algebra $V_{C_{\Gamma}}(SL_2(\C))$ for $\Gamma$ the leftmost graph in Figure \ref{g3}.  The semigroup algebra $\C[H_{\Gamma}^*]$ is generated by the simple loops in this graph. Each of these loops $x_{abc}$ is determined by the values $a, b, c$ on the outer edges in Figure \ref{g3label}.

\begin{figure}[htbp]
\centering
\includegraphics[scale = 0.45]{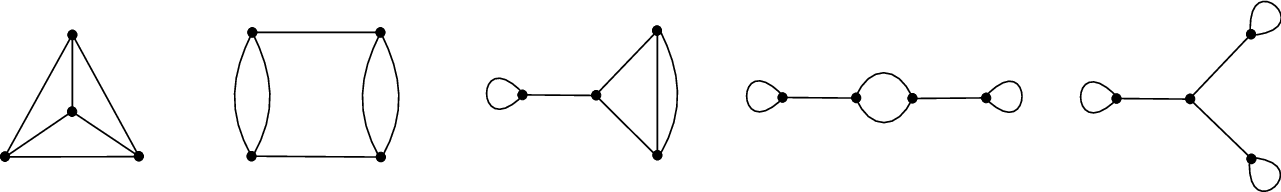}
\caption{The trivalent graphs of genus $3.$}
\label{g3}
\end{figure}

\begin{figure}[htbp]

$$
\begin{xy}
(0, 0)*{\bullet} = "A1";
(16, 0)*{\bullet} = "A2";
(8, 14)*{\bullet} = "A3";
(8, 5)*{\bullet} = "A4";
(8, -2)*{a};
(3, 8)*{b};
(13, 8)*{c};
"A1"; "A2";**\dir{-}? >* \dir{-};
"A2"; "A3";**\dir{-}? >* \dir{-};
"A3"; "A1";**\dir{-}? >* \dir{-};
"A1"; "A4";**\dir{-}? >* \dir{-};
"A2"; "A4";**\dir{-}? >* \dir{-};
"A3"; "A4";**\dir{-}? >* \dir{-};
\end{xy}
$$\\
\caption{The diagram representing $x_{abc}$.}
\label{g3label}
\end{figure}
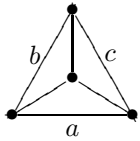

The generating relation is $x_{111}x_{100}x_{010}x_{001} = x_{110}x_{101}x_{011}x_{000}.$  We lift this by expanding the product $x_{110}x_{101}x_{011}x_{000}$ with skein relations. 

$$x_{111}x_{100}x_{010}x_{001} - x_{110}x_{101}x_{011}x_{000} + x_{110}^2x_{000}^2+ x_{101}^2x_{000}^2 + x_{011}^2x_{000}^2$$

$$+ x_{100}^2x_{000}^2 + x_{010}^2x_{000}^2 + x_{001}^2x_{000}^2 + x_{111}^2x_{000}^2 - 4x_{000}^4 + x_{111}x_{100}x_{011}x_{000} + x_{111}x_{010}x_{101}x_{000} + x_{111}x_{001}x_{110}x_{000}$$

$$+ x_{100}x_{010}x_{110}x_{000} + x_{100}x_{001}x_{101}x_{000} + x_{010}x_{001}x_{011}x_{000} = 0$$

\bigskip
The boundary divisor  is defined by the equation $x_{000} = 0$, yielding $x_{111}x_{100}x_{010}x_{001} = 0.$  Each of these generators
can be translated into a trace-word by choosing the spanning tree in $\Gamma$ which misses the edges $a, b, c$ in Figure \ref{g3label}.  The word associated to $x_{abc}$ is then trace of the elements corresponding to the indices $\{a, b, c\}$ with value $1.$

\subsection{A genus $g$ case}

Now we restrict ourselves to the set of graphs $\Upsilon_g$, depicted in Figure \ref{Gammag}. In \cite{M4} the semigroup $H_{\Upsilon_g}^*$ is shown to be generated by those weightings of level $\leq 2$, subject to quadratic relations in these generators.  It follows that generators of $V_{C_{\Upsilon_g}}(SL_2(\C))$ are disjoint unions of loops with edges of multiplicity at most $2$, see Figure \ref{Ggen}. Each of these elements has a pole of order $1$ or $2$ along the divisor $D_{\Upsilon_g} \subset B_{C_{\Upsilon_g}}(SL_2(\C)).$  We identify the generating elements with trace-words by selecting the spanning tree which misses the top edge of each loop in Figure \ref{Gammag} and the loops at the ends of the graph, and ordering these edges from left to right. 

\begin{figure}[htbp]
\centering
\includegraphics[scale = 0.75]{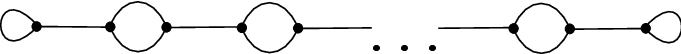}
\caption{The graph $\Upsilon_g$.}
\label{Gammag}
\end{figure} 

\begin{figure}[htbp]
\centering
\includegraphics[scale = 0.5]{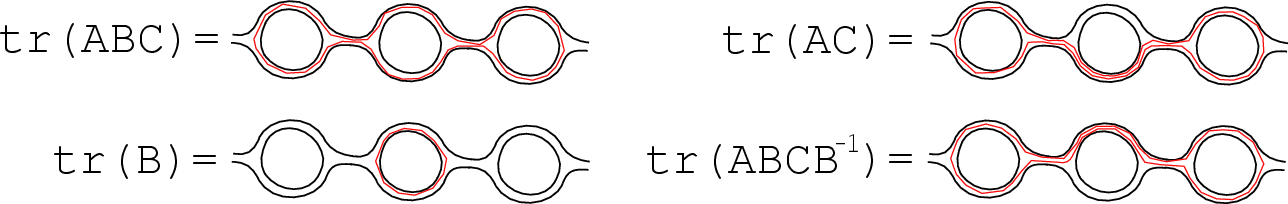}
\caption{Ribbon representations of trace-words.}
\label{Ggen}
\end{figure} 

For an edge $e \in E(\Upsilon_g)$ one can consider products of generators $[w_1][w_2] \in U_{\Upsilon_g}$ which have the same weight on $e$. In this case the left hand sides of these generators with respect to $e$ can be exchanged, producing a relation $[w_1][w_2] = [w_1'][w_2']$, this type of relation is lifted by the operation depicted in Figure \ref{LocalSkein}.  In a similar manner, two loops may intersect in the diagram, these relations are lifted by the operation depicted in Figure \ref{LocalSkein2}. In \cite{M4} it is shown that extensions of the quadratic relations, along with the exchanging relations at edges suffice to generate the relations in $H_{\Upsilon_g}^*$.

\begin{figure}[htbp]
\centering
\includegraphics[scale = 0.65]{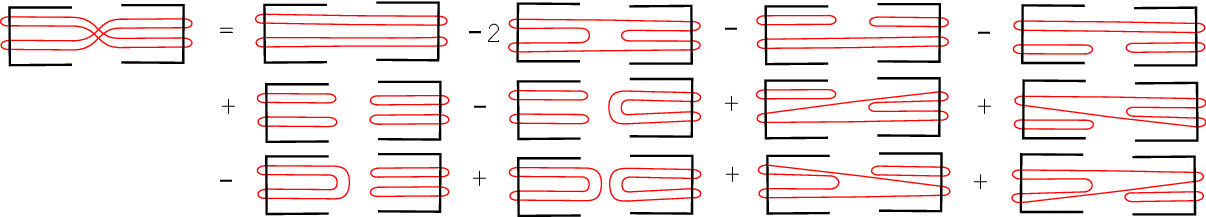}
\caption{Skein relation which resolves a crossing along an edge.}
\label{LocalSkein}
\end{figure} 

\begin{figure}[htbp]
\centering
\includegraphics[scale = 0.65]{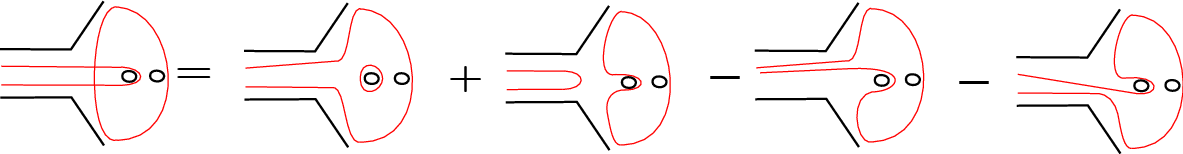}
\caption{Skein relation which resolves a crossing in a trinode.}
\label{LocalSkein2}
\end{figure}

\section{Geometry of $B_{C, \vec{p}}(SL_2(\C))$ and $K_{C, \vec{p}}(SL_2(\C))$}\label{stratificationsection}

  We describe the closed intersection stratification defined by the boundary divisor $D_{\Gamma} \subset B_{C_{\Gamma}, \vec{p}_{\Gamma}}(SL_2(\C))$, and we show that the irreducible components of the compactification of $Spec(\C[H_{\Gamma}])$ defined by the Rees algebra $\C[H_{\Gamma}^*]$ are  degenerations of the components of $D_{\Gamma}.$

\subsection{The ideals $I_S$}

An ideal $I$ in an algebra with a basis of planar $\Gamma-$tensors is called ''planar" if it has a basis of planar $\Gamma$-tensors.  Planar bases exist for any $W_{C_{\Gamma}, \vec{p}_{\Gamma}}(SL_2(\C))/I,$ for $I$ any planar ideal. 

\begin{definition}
For $S \subset V(\Gamma),$ the ideal $I_S \subset W_{C_{\Gamma}, \vec{p}_{\Gamma}}(SL_2(\C))$ is defined to be the homogenous ideal with $L-$th component the span of those $\Gamma-$tensors with $<L$ paths through some $v_i \in S.$
\end{definition}

\begin{proposition}
The ideal $I_S \subset W_{C_{\Gamma}, \vec{p}_{\Gamma}}(SL_2(\C))$, is prime and planar. 
\end{proposition}

\begin{proof}
Primeness of $I_S$ follows from the observation that the product of two $\Gamma$-tensors $V(P_1, \phi_1, A_1),$ $V(P_2, \phi_2, A_2)$ of levels $L_1, L_2$ respectively at $v$ must have $L_1 + L_2$ paths through $v.$     Expanding a $\Gamma-$tensor in $I_S$ by skein relations gives planar $\Gamma-$tensors in $I_S,$ as the resulting planar $\Gamma-$tensors have strictly smaller associated spin diagrams, this proves that $I_S$ is planar. 
\end{proof}

It follows that the graded component $I_S(L)$ has a basis of those planar $\Gamma-$tensors with $<L$ paths through some vertex in $S.$  The component $W_{C_{\Gamma}, \vec{p}_{\Gamma}}(L)/I_S(L)$ has a basis of planar $\Gamma-$tensors which have exactly $L$ paths through each vertex in $S,$ and $\leq L$ paths through each $v \in V(\Gamma) \setminus S$.  Multiplication in $W_{C_{\Gamma}, \vec{p}_{\Gamma}}(SL_2(\C))/I_S$ is computed by expanding a product into planar $\Gamma-$tensors, and removing any summand with $<L$ paths through some vertex in $S$. The $0$-locus $D_S$ of $I_S$ is $\cap_{v \in S} D_v$ as $I_S = \sum_{v \in S} I_v$.

The scheme $Proj(\C[H_{\Gamma}^*])$ is $Spec(\C[H_{\Gamma}])$ compactified by a divisor $K_{\Gamma},$  by arguments which are identical to those used for $V_{C_{\Gamma}, \vec{p}_{\Gamma}}(SL_2(\C))$ and $W_{C_{\Gamma}, \vec{p}_{\Gamma}}(SL_2(\C)).$ The ideals $J_S \subset \C[H_{\Gamma}^*]$ which cut out the irreducible components $K_S \subset K_{\Gamma}$ have an identical description to the $I_S$ as vector spaces, following the bijection between planar $\Gamma-$tensors and spin diagrams.  

\begin{proposition}\label{toricdegenstrata}
Each $D_S$ has a degeneration to the toric scheme $K_S$.
\end{proposition}

\begin{proof}
The algebras $\C[D_S]$ and $\C[K_S]$ have matching planar bases.  Multiplication in  $\C[D_S]$  is computed by expanding the product in $W_{C_{\Gamma}, \vec{p}_{\Gamma}}(SL_2(\C))$, and cutting off terms with $<L$ paths through some vertex in $S$, a process which does not eliminate the highest term.  The associated graded algebra of $\C[D_S]$ by the filtration defined by any $x \in \mathcal{C}_{\Gamma}$ with $x(e) \neq 0$ for all $e \in E(\Gamma)$ is therefore $\C[K_S]$.
\end{proof}

\begin{corollary}
The scheme $K_{V(\Gamma)}$ is isomorphic to $D_{V(\Gamma)}$
\end{corollary}

\begin{proof}
The ideal $I_{V(\Gamma)}$ has a basis of those planar $\Gamma$ tensors which are not maximal level for some $v \in V(\Gamma)$.  The coordinate ring $\C[D_{V(\Gamma)}] = W_{C_{\Gamma}, \vec{p}_{\Gamma}}(SL_2(\C))/I_{V(\Gamma)}$ is therefore spanned by planar $\Gamma$-tensors which are maximal everywhere in $\Gamma.$  When the product $V(P_w, \phi_w, A_w)V(P_{w'}, \phi_{w'}, A_{w'})$ is expanded into $V(P_{w + w'}, \phi_{w + w'}, A_{w + w'})$ plus lower planar $\Gamma-$tensors, all of the lower summands are cut off. This is precisely the affine semigroup multiplication in $\C[H_{\Gamma}^*]/J_{V(\Gamma)}$. 
\end{proof}

\subsection{The stratification of $K_{C_{\Gamma}, \vec{p}_{\Gamma}}(SL_2(\C))$}

We construct the strata of the scheme $K_{C_{\Gamma}, \vec{p}_{\Gamma}}(SL_2(\C))$ by taking $U^n$ quotients of each of the $D_S.$ The scheme $E_S = D_S/U^n$ is cut out of $K_{C_{\Gamma}, \vec{p}_{\Gamma}}(SL_2(\C))$ by the ideal $I_S^{U^n} = I_S \cap V_{C_{\Gamma}, \vec{p}_{\Gamma}}(SL_2(\C))$.  This is the planar ideal in $V_{C_{\Gamma}, \vec{p}_{\Gamma}}(SL_2(\C))$ with a basis given by those planar $\Gamma-$tensors in $I_S$ which have only [UP] orientations at their leaves.  In particular, the stratification poset of $K_{C_{\Gamma}, \vec{p}_{\Gamma}}(SL_2(\C))$ by the $E_S$ is isomorphic to the stratification poset defined by the $D_S.$

We address when $I_S = I_T$ for distinct subsets $S, T \subset V(\Gamma).$ To simplify our arguments, we treat the case of $K_{C_{\Gamma}, \vec{p}_{\Gamma}}(SL_2(\C))$ and the semigroup $U_{\Gamma}^*$, however everything we say transfers immediately to $B_{C_{\Gamma}, \vec{p}_{\Gamma}}(SL_2(\C))$ and the semigroup $H_{\Gamma}^*.$

\begin{lemma}
The stratification poset of $K_{C_{\Gamma}, \vec{p}_{\Gamma}}(SL_2(\C))$ defined by the $E_S$ is isomorphic to the stratification poset of $Proj(\C[U_{\Gamma}^*])$ by the components defined by the corresponding toric ideals.  
\end{lemma}

\begin{proof}
This follows from the planarity of the ideals $I_S^{U_-^n}$ and Proposition \ref{toricdegenstrata}. 
\end{proof}

\begin{corollary}\label{strat}
For $S \subset V(\Gamma),$ if there is a spin diagram $a \in U_{\Gamma}(L)$ with the property that the sum $a(e) + a(f) + a(g)$ for the edges $e, f, g$ around a vertex $v \in V(\Gamma)$ is $2L$ precisely when $v \in S,$ then $I_S$ is distinct from all $I_T$ with $S \subset T.$
\end{corollary}

We pass to the convex cone $\mathcal{U}_{\Gamma}^*$ in $\R^{E(\Gamma)}\times \R$, and use the fact that  if a rational point can be found which satisfies the requirements of Corollary \ref{strat}, it can be sufficiently multiplied to procure a spin diagram.  

\begin{proposition}\label{leaftop}
If $\Gamma$ has a leaf, then for any $S \subset V(\Gamma)$, there is a point in $\mathcal{U}_{\Gamma}^*$ which satisfies
the condition of Corollary \ref{strat}. 
\end{proposition}

\begin{proof}
We let $w \in \mathcal{U}_{\Gamma}^*$ be the weighting of level $3$ which assigns each edge $2.$ Each vertex 
has the sum $w(e) + w(f) + w(g) = 6$, we show that the value of this sum can be lowered for $v$ in any given set of vertices. 
Fix $v \in V(\Gamma)$, choose a simple path $e_1, \ldots, e_k$ from an edge $e_1$ which borders $v$ to a leaf $e_k = \ell.$  
We obtain a new weighting from $w$, by changing the weight along this path with some rational, sufficiently small $\epsilon.$

$$
w'(e_1) = w(e_1) - \epsilon, w'(e_2) = w(e_2) + \epsilon, \ldots, w'(e_k) = w(e_k) \pm \epsilon
$$

This does not change the total sum around any vertex, except $v$, which has its total sum lowered by $\epsilon.$
We can inductively repeat this procedure for all vertices in $V(\Gamma) \setminus S$ to produce the required weighting. 
\end{proof}

We can adapt the proof of Proposition \ref{leaftop} to prove the following for graphs without leaves. 

\begin{proposition}\label{g0top}
If $S \subset V(\Gamma)$ has a pair of points connected by a simple path of odd length, then there is an element in $U_{\Gamma}^*$ which satisfies
the condition of Corollary \ref{strat}.  If $\Gamma$ has an odd length simple cycle, then for all $S \subset V(\Gamma)$, there is a point in $\mathcal{U}_{\Gamma}^*$ which satisfies the condition of Corollary \ref{strat}. 
\end{proposition}

\begin{proof}
The idea is the same as the proof of Proposition \ref{leaftop}, we can construct paths from any vertex to the odd length cycle, or a node which
has already been "lowered".  We then observe that if $v, w \in V(\Gamma)$ are seperated by a path of odd length, their total sums can be simultaneously lowered by shrinking the weight on the first, respectively last edges in such a path by a common $\epsilon,$ and placing alternating $\pm \epsilon$ on the edges in between.  
\end{proof}

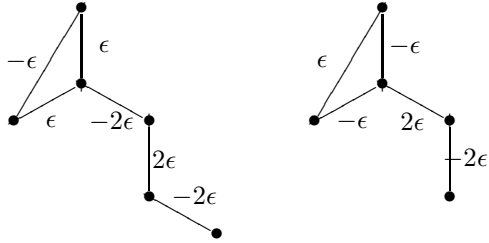
\begin{figure}[htbp]

$$
\begin{xy}
(0, 0)*{\bullet} = "A1";
(0, 10)*{\bullet} = "A2";
(9, -5)*{\bullet} = "A3";
(-9, 15)*{\bullet} = "A4";
(-9, 25)*{\bullet} = "A5";
(-18, 10)*{\bullet} = "A6";
(6, 0)*{-2\epsilon};
(2, 5)*{2\epsilon};
(-5, 10)*{-2\epsilon};
(-13, 10)*{\epsilon};
(-17, 18)*{-\epsilon};
(-6, 20)*{\epsilon};
"A2"; "A1";**\dir{-}? >* \dir{-};
"A1"; "A3";**\dir{-}? >* \dir{-};
"A4"; "A2";**\dir{-}? >* \dir{-};
"A4"; "A5";**\dir{-}? >* \dir{-};
"A6"; "A4";**\dir{-}? >* \dir{-};
"A5"; "A6";**\dir{-}? >* \dir{-};
(40, 0)*{\bullet} = "B1";
(40, 10)*{\bullet} = "B2";
%(49, -5)*{\bullet} = "B3";
(31, 15)*{\bullet} = "B4";
(31, 25)*{\bullet} = "B5";
(22, 10)*{\bullet} = "B6";
%(46, 0)*{\epsilon};
(42, 5)*{-2\epsilon};
(35, 10)*{2\epsilon};
(27, 10)*{-\epsilon};
(23, 18)*{\epsilon};
(34, 20)*{-\epsilon};
"B2"; "B1";**\dir{-}? >* \dir{-};
%"B1"; "B3";**\dir{-}? >* \dir{>};
"B4"; "B2";**\dir{-}? >* \dir{-};
"B4"; "B5";**\dir{-}? >* \dir{-};
"B6"; "B4";**\dir{-}? >* \dir{-};
"B5"; "B6";**\dir{-}? >* \dir{-};
\end{xy}
$$\\
\caption{Changing the weighting along a path to an odd cycle.}
\label{Gammacompute}
\end{figure}

  If $\Gamma$  has only even length simple cycles, $E_S$ is distinct from each $E_T$ for $S \subset T$, if $V(\Gamma) \setminus S$ contains a pair of points an odd distance apart.  Notice that this condition is heritable, if $S \subset T$ and $S$ has this property, then so does $T$.  Furthemore if $S$ does not have this property, then neither does $R \subset S$.  

\begin{proposition}\label{bippos}
If $\Gamma$ is bipartite with no leaves, then the lattice formed by the $D_S$ is the quotient of the Boolean lattice on $V(\Gamma)$ 
by the ideal of those $S \subset V(\Gamma)$ such that $V(\Gamma) \setminus S$ is contained in one of the sets in the partition $B_1 \cup B_2 = V(\Gamma)$ defined by the bipartite structure. 
\end{proposition}

\begin{proof}
We have already shown that if $V(\Gamma)\setminus S$ has an odd length path, we can distinguish $E_S$ from other strata.  This is not the case if and only if $V(\Gamma) \setminus S$ is contained in a set of the partition, say $B_1$.  The graph $\Gamma$ is regular, therefore $|B_1| = |B_2|$.  If we assign weights to $E(\Gamma)$, The total sums of the weights around the vertices in $B_1$ equals this same sum for $B_2$.  It follows that if $w \in \mathcal{U}_{\Gamma}(L)$ assigns all the vertices of $B_2$ weight $2L$, the same must be the case for all vertices in $B_1$.  
\end{proof}

It remains to determine the dimensions of each of the $E_S,$  we do this by calculating the dimensions of the corresponding
faces of the polytope $\mathcal{U}_{\Gamma}(L).$  First we describe $\mathcal{U}_{\Gamma}(L)$ in a more convenient way.  We choose a set of $g$ edges $\{e_1, \ldots, e_g\}$, which give a tree $\tree$ when they are split $\{e_1, e_1', \ldots, e_g, e_g'\}$, this defines a covering $\pi: \tree \to \Gamma.$  We Choose a leaf $\ell_n$ of $\Gamma$, and set this to be a sink in $\tree$ by placing a direction on all edges in $\tree$ toward $\ell_n.$ 

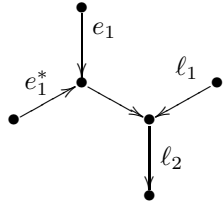
\begin{figure}[htbp]

$$
\begin{xy}
(0, 0)*{\bullet} = "A1";
(0, 10)*{\bullet} = "A2";
(9, 15)*{\bullet} = "A3";
(-9, 15)*{\bullet} = "A4";
(-9, 25)*{\bullet} = "A5";
(-18, 10)*{\bullet} = "A6";
(3, 5)*{\ell_2};
(5, 17)*{\ell_1};
(-6, 22)*{e_1};
(-15, 15)*{e_1^*};
"A1"; "A2";**\dir{-}? >* \dir{>};
"A2"; "A3";**\dir{-}? >* \dir{>};
"A2"; "A4";**\dir{-}? >* \dir{>};
"A4"; "A5";**\dir{-}? >* \dir{>};
"A4"; "A6";**\dir{-}? >* \dir{>};
\end{xy}
$$\\
\caption{Introducing a direction on $\tree.$}
\label{Gammacompute2}
\end{figure}

We introduce slack variables, $L_v$ , $v \in V(\Gamma)$ at the vertices of $\Gamma,$ and we identify $\mathcal{U}_{\Gamma}(L)$ with the labellings $w$ of $E(\tree)$ which satisfy the triangle inequalities at each trinode, $w(e_i) = w(e_i')$, and $w(e_v) + w(f_v) + w(h_v) = L_v$, $L_v \leq L.$  The value of a $w$ on $\tree$ is determined by the $L_v$ and the edges $e_1, \ldots, e_g, \ell_1, \ldots, \ell_{n-1}.$  The dimension of $\mathcal{U}_{\Gamma}(L)$
is $E(\Gamma)$, so no linear relations hold among these parameters.     Forgetting for a moment that $w(e_i) = w(e_i'),$ the value of $w$ on $\ell_n$ is a linear combination of the $L_v, w(e_i), w(e_i'), w(\ell_i)$, with all coefficients equal to $\pm 1,$   with the following signs.

\begin{enumerate}
\item The coefficient of the $w(e_i), w(e_i'), w(\ell_k)$ is equal to $(-1)^{d+1},$ where $d$ is the length of the directed path containing the leaf and $\ell_n.$\\
\item The coefficient of $L_v$ is equal to $(-1)^{d+1}$, where $d$ is the length of the directed path which starts at $v$ and includes $\ell_n.$\\
\end{enumerate}

If we pass to the face corresponding to $E_{V(\Gamma)}$, we set all $L_v$ equal to $L$. The result is freely determined by $n + g -1 = |E(\Gamma)| - |V(\Gamma)|$ parameters.  This allows us to prove the following result on $E_{\Gamma}.$

\begin{proposition}\label{codim}
If $\Gamma$ has a leaf, for any $S \subset V(\Gamma)$, the codimension of $E_S$ is $|S|$
\end{proposition}

\begin{proof}
This follows from the fact that each $E_S$ is distinct, and $codim(E_{V(\Gamma)}) = |V(\Gamma)|.$
\end{proof}

If $\Gamma$ has no leaves we can repeat the above construction for $w \in \mathcal{U}_{\Gamma}(L)$, expressing $w(e_1)$ as a linear combination of $w(e_i), w(e_i^*), w(e_1^*)$ and the $L_v.$  If the edge which gives $e_1, e_1^*$ lies in a simple cycle of odd length, then the coefficient of $w(e_1^*)$  is $-1$.  It follows that $w(e_1), w(e_1^*), \ldots, w(e_g), w(e_g^*)$ satisfy a linear equation when the $L_v$ are set to $L$. As a consequence $codim(E_{V(\Gamma)}) = |V(\Gamma)|$ in this case.  If $\Gamma$ has no leaves and is bipartite, the coefficient of $w(e_i)$ in the expansion of $w(e_1)$ is opposite of that of $w(e_i^*)$, and the coefficient of $w(e_1)$ is $1.$  It follows that when all $L_v$ are set to $L$, the remaining parameters satisfy no linear equations.  The following finishes the proof of Theorem \ref{strattheorem}. 

\begin{proposition}
If $\Gamma$ has no leaves, but has a simple cycle of odd length, $E_S$ for any $S \subset V(\Gamma)$ has codimension
$|S|.$ If $\Gamma$ is bipartite and has no leaves, $E_S$ has codimension $|S|$ if $V(\Gamma) \setminus S$ is not contained
in partition set $B_i$, and $2g-3$ otherwise. 
\end{proposition} 

\begin{proof}
By Propositions \ref{g0top} and \ref{bippos} and the argument above, it remains to observe that there is a
complete flag of subsets of $V(\Gamma)$, such that no subset of size $\geq 2$ is contained in a partition set $B_i$.  
\end{proof}

\begin{figure}[htbp]
\centering
\includegraphics[scale = 0.5]{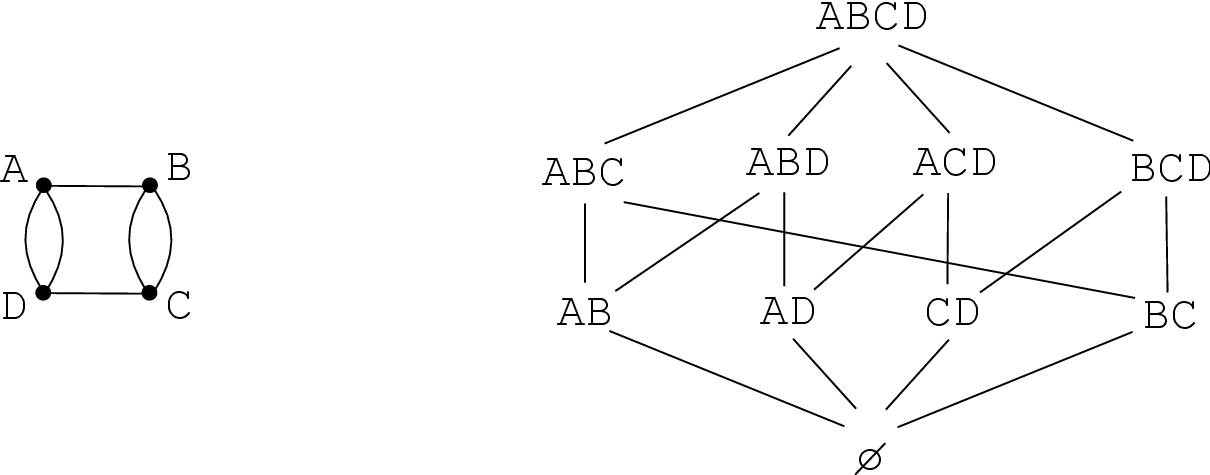}
\caption{A non-Boolean stratification poset.}
\label{Lattice}
\end{figure}

\bigskip
\noindent
Christopher Manon:\\
Department of Mathematics,\\ 
George Mason University\\ 
Fairfax, VA 22030 USA 

\end{document}